\documentclass[11pt. oneside]{amsart}
\usepackage[lmargin=1in,rmargin=1in, bmargin=1in, tmargin=1in]{geometry}
\usepackage{amsmath, amssymb,tikz}
\usepackage{tikz-cd}
\usepackage{tikz}
\usepackage{todonotes}
\usepackage{mathrsfs}
\usepackage{extarrows}
\usepackage{graphicx}

\usepackage{hyperref}
\usepackage{IEEEtrantools}
\usepackage{mathrsfs}
\usepackage[utf8]{inputenc}
\usepackage[english]{babel}
 
\usepackage{comment}
\usepackage{csquotes}

\usepackage{amsmath,amsthm,amssymb}
\usepackage{colonequals}

\newcommand{\QQ}{\mathbb{Q}}
\newcommand{\Qp}{\QQ_p}
\newcommand{\Zp}{\mathbb{Z}_p}
\newcommand{\ZZ}{\mathbb{Z}}
\newcommand{\WW}{\mathcal{W}}

\newcommand{\defeq}{\colonequals}

\newcommand{\Iw}{\mathrm{Iw}}

\DeclareMathOperator{\Gal}{Gal}

\numberwithin{equation}{subsection}

\newtheorem{prop}[equation]{Proposition}
\newtheorem{theorem}[equation]{Theorem}
\newtheorem{proposition}[equation]{Proposition}
\newtheorem{lemma}[equation]{Lemma}
\newtheorem{corollary}[equation]{Corollary}
\newtheorem*{expectation*}{Expectation}
\newtheorem*{conjecture*}{Conjecture}

\newtheorem{theoremx}{Theorem}

\theoremstyle{definition}
\newtheorem{defn}[equation]{Definition}
\newtheorem{definition}[equation]{Definition}

\theoremstyle{remark}
\newtheorem{remark}[equation]{Remark}

\usepackage{tikz-cd}
\usepackage{mathrsfs}

\DeclareFontFamily{U}{wncy}{}
\DeclareFontShape{U}{wncy}{m}{n}{<->wncyr10}{}
\DeclareSymbolFont{mcy}{U}{wncy}{m}{n}
\DeclareMathSymbol{\sha}{\mathord}{mcy}{"58}

\newcommand{\ide}[1]{\mathfrak{#1}}

\newcommand{\mbb}[1]{\mathbb{#1}}

\newcommand{\exactseq}[3]{0 \longrightarrow #1 \longrightarrow #2 \longrightarrow #3 \longrightarrow 0}

\newcommand{\cohom}[3]{\mathrm{H}^{#1}(#2, #3)}

\newcommand{\ordd}{\mathcal{O}}

\newcommand{\rcont}[2]{\mathbf{R}\Gamma_{\mathrm{cont}}(#1, #2)}

\newcommand{\selcom}[3]{\mathbf{R}\widetilde{\Gamma}(#1, #2; #3)}

\newcommand{\rcontl}[1]{\rcont{G_v}{#1}}
\newcommand{\selcomo}[1]{\selcom{G_{\Sigma}}{#1}{\Delta}}

\newcommand{\selcomof}[1]{\mathbf{R}\widetilde{\Gamma}_f(G_{\Sigma}, #1)}

\newcommand{\invs}[1]{\mathcal{#1}}

\newcommand{\esel}[2]{\widetilde{H}^{#1}(G_{\Sigma}, #2; \Delta)}
\newcommand{\eseldel}[3]{\widetilde{H}^{#1}(G_{\Sigma}, #2; #3)}

\newcommand{\selmergrp}[2]{\widetilde{\opn{H}}_f^{#1}(\mbb{Q}, #2)}
\newcommand{\selmergrprel}[2]{\widetilde{\opn{H}}_{\mathrm{rel}}^{#1}(\mbb{Q}, #2)}
\newcommand{\selmergrpstr}[2]{\widetilde{\opn{H}}_{\mathrm{str}}^{#1}(\mbb{Q}, #2)}

\newcommand{\rankiw}[2]{_c \invs{RI}_{#1, #2, 1}^{[j]}}
\newcommand{\beilf}[1]{_c \invs{BF}_{#1, 1}^{[\invs{U}_1, \invs{U}_2, j]}}
\newcommand{\beilfj}[2]{_c \invs{BF}_{#1, 1}^{[\invs{F}, #2, j]}}
\newcommand{\beilfd}[2]{_c \invs{BF}_{#1, 1}^{[\invs{F}, #2]}}

\newcommand{\opn}[1]{\operatorname{#1}}

\newcommand{\hatot}{\hat{\otimes}}

\newcommand{\bdrig}{\mathbf{B}^{\dagger}_{\mathrm{rig}, \mbb{Q}_p}}
\newcommand{\bdrigE}{\mathbf{B}^{\dagger}_{\mathrm{rig}, \mbb{Q}_p} \hat{\otimes} E}
\newcommand{\bdrigA}[1]{\mathbf{B}^{\dagger}_{\mathrm{rig}, \mbb{Q}_p} \hat{\otimes} #1}
\newcommand{\bprig}{\mathbf{B}^{+}_{\mathrm{rig}, \mbb{Q}_p}}
\newcommand{\bdr}{\mathbf{B}_{\mathrm{dR}}}
\newcommand{\bcris}{\mathbf{B}_{\mathrm{cris}}}

\newcommand{\ddrig}[1]{\mathbf{D}^{\dagger}_{\mathrm{rig}}(#1)}
\newcommand{\nrig}[1]{\mathbf{N}_{\mathrm{rig}}(#1)}

\newcommand{\Cris}[1]{#1_{\mathrm{cris}}}
\newcommand{\derham}[1]{#1_{\mathrm{dR}}}

\newcommand{\F}[2]{\mathscr{F}^{#1}#2}
\newcommand{\Ff}[3]{\mathscr{F}^{#1 #2}#3}

\newcommand{\mbf}[1]{\mathbf{#1}}

\newcommand{\tbyt}[4]{\left( \begin{array}{cc} #1 & #2 \\ #3 & #4 \end{array} \right)}

\newcommand{\Addresses}{{
  \bigskip
  \footnotesize

  (Graham) \textsc{Department of Mathematics, South Kensington Campus, Imperial College London, London SW7 2AZ, UK}\par\nopagebreak
  \textit{E-mail address}: \texttt{a.graham17@imperial.ac.uk}

  \medskip

  (Gulotta) \textsc{Mathematical Institute, University of Oxford, Oxford OX2 6GG, UK}\par\nopagebreak
  \textit{E-mail address}: \texttt{Daniel.Gulotta@maths.ox.ac.uk}

  \medskip

  (Xu) \textsc{Department of Mathematics, Harvard University, 1 Oxford Street, Cambridge, MA 02138, U.S.A.}\par\nopagebreak
  \textit{E-mail address}: \texttt{yujiex@math.harvard.edu}

}}

\usepackage{stmaryrd}

\DeclareMathOperator{\Hom}{Hom}
\DeclareMathOperator{\TSym}{TSym}
\newcommand{\db}[1]{\left\llbracket #1 \right\rrbracket}

\newcommand{\FF}{\mathcal{F}}
\newcommand{\GG}{\mathcal{G}}
\newcommand{\OO}{\mathcal{O}}
\newcommand{\GL}{\mathrm{GL}}
\newcommand{\SH}{\mathscr{H}}
\newcommand{\UU}{\mathcal{U}}

\newcommand{\DD}{\mathcal{D}}

\newcommand{\et}{\mathrm{\acute{e}t}}

\newcommand{\pgp}{$(\varphi,\Gamma)$}
\newcommand{\EE}{\mathcal{E}}

\newcommand{\RR}{\mathbb{R}}

\newcommand{\sep}{\mathrm{sep}}

\title{Bounding Selmer groups for the Rankin--Selberg convolution of Coleman families}

\author{Andrew Graham, Daniel R. Gulotta, and Yujie Xu}

\date{}

\begin{document}
\begin{abstract}
Let $f$ and $g$ be two cuspidal modular forms and let $\FF$ be a Coleman family passing through $f$, defined over an open affinoid subdomain $V$ of weight space $\WW$. Using ideas of Pottharst, under certain hypotheses on $f$ and $g$ we construct a coherent sheaf over $V \times \WW$ which interpolates the Bloch-Kato Selmer group of the Rankin-Selberg convolution of two modular forms in the critical range (i.e the range where the $p$-adic $L$-function $L_p$ interpolates critical values of the global $L$-function). We show that the support of this sheaf is contained in the vanishing locus of $L_p$. 
\end{abstract}
\maketitle

\tableofcontents


\section{Introduction}

In \cite{BeilinsonFlach} (and more generally \cite{KLZ-modular}) Kings, Lei, Loeffler and Zerbes construct an Euler system for the Galois representation attached to the convolution of two modular forms. This Euler system is constructed from \emph{Beilinson--Flach classes}, which are norm-compatible classes in the (absolute) \'{e}tale cohomology of the fibre product of two modular curves. It turns out that these Euler system classes exist in families, in the sense that there exist classes 
\[
{_c\invs{BF}_{m, 1}^{[\mathcal{F}, \mathcal{G}]}} \in \opn{H}^1\left( \mbb{Q}(\mu_m), D^{\mathrm{la}}(\Gamma, M(\mathcal{F})^* \hatot M(\mathcal{G})^* ) \right)
\]
which specialise to the Beilinson--Flach Euler system at classical points. Here $\mathcal{F}$ and $\mathcal{G}$ are Coleman families with associated Galois representations $M(\mathcal{F})$ and $M(\mathcal{G})$ respectively, and $D^{\mathrm{la}}(\Gamma, -)$ denotes the space of locally analytic distributions on $\Gamma \cong \mbb{Z}_p^{\times}$.

The above classes are constructed in \cite{lz-coleman} and shown to satisfy an ``explicit reciprocity law'' relating the bottom class ($m=1$) to the three variable $p$-adic $L$-function constructed by Urban \cite{Urban}. This relation can then be used to prove instances of the Bloch--Kato conjecture for the Galois representation attached to the convolution of two modular forms (including the case of an elliptic curve twisted by an Artin representation).

Building on the work of Nekov\'{a}\v{r}, Pottharst \cite{pottharst-selmer} describes how one can put the Bloch--Kato Selmer group of a Galois representation into a family. More precisely, given a family of $G_{\mbb{Q}}$-representations over a rigid analytic space $X$, he constructs a coherent sheaf $\invs{S}$ on $X$ which specialises to the Bloch-Kato Selmer group at certain ``crystalline'' points of $X$, i.e. points where the Galois representation is crystalline at $p$. This gives rise to the natural question:
\begin{itemize}
\item Do the Beilinson--Flach classes ${_c\invs{BF}_{m, 1}^{[\mathcal{F}, \mathcal{G}]}}$ (and hence the three-variable $p$-adic $L$-function) control the behaviour of $\invs{S}$?
\end{itemize}
In this paper, we provide a partial answer to this question.

\subsection{Summary of results} \label{SummaryOfResults}

Fix an odd prime $p \geq 5$. To explain the results we introduce the following notation. Let $N \geq 1$ be an integer prime to $p$ and let $f$ and $g$ be two normalised cuspidal newforms of levels $\Gamma_1(N_1)$ and $\Gamma_1(N_2)$ and weights $k+2$ and $k'+2$ respectively, such that $N_1, N_2$ both divide $N$ and $k', k \geq 0$. 

We assume that $k \neq k'$, and that one of the $p$-stabilisations of $f$ and \emph{both} of the $p$-stabilisations of $g$ are noble (Definition \ref{DefinitionOfNoble}). This implies that all three of these modular forms can be put into Coleman families. We denote the weight space parameterising all continuous characters $\mbb{Z}_p^{\times} \to \mbb{C}_p^{\times}$ by $\invs{W}$ and for an integer $i$, we denote the character $x \mapsto x^i$ simply by $i$.

Let $E$ be a $p$-adic field and $\invs{F}$ and $\invs{G}$ two Coleman families over affinoid domains $V_1 \subset \invs{W}_E$ and $V_2 \subset \invs{W}_E$ passing through $p$-stabilisations of $f$ and $g$ respectively (see Definition \ref{WW_E} for the definition of $\WW_E$). We impose the following hypotheses on $f$ and $g$:
\begin{enumerate}
    \item The image of the Galois representation attached to the convolution of $f$ and $g$ is \emph{big} (see (BI) \S \ref{AVanishingResult}).
    \item The inertia invariants at ramified primes of the Galois representation attached to the convolution of $f$ and $g$ is free (flatness of inertia, see \S \ref{SelmerSheafAssumptions}).
    \item $f$ and $g$ are not congruent modulo $p$ to forms of a lower level (minimally ramified, see \S \ref{SelmerSheafAssumptions}).
    \item The $p$-adic $L$-function attached to the convolution of $f$ and $g$ does not have a trivial zero, which is a condition on the Fourier coefficients of $f$ and $g$ (condition (NLZ) for the point $\mathbf{x}$ corresponding to $f$ and $g$ in \S \ref{AVanishingResult}).  
\end{enumerate}
Under these hypotheses, there exists a coherent analytic sheaf $\invs{S}$ on $X:= V_1 \times V_2 \times \invs{W}$, such that for all $\mathbf{x} = (k_1, k_2, j) \in X$ with $k_1, k_2, j$ integers and $1 \leq k_2+1 \leq j \leq k_1$, the specialisation of $\invs{S}$ satisfies 
\[
\invs{S}_{\mathbf{x}} \cong \opn{H}^1_f(\mbb{Q}, [M(\invs{F}_{k_1}) \otimes M(\invs{G}_{k_2})](1+j))^*
\]
where $M(-)$ denotes the Galois representation attached to a modular form (in the sense of Deligne) and the right-hand side is the (dual of the) Bloch--Kato Selmer group. We recall the construction of this sheaf in section \ref{selmer sheaf} following \cite[\S 3.4]{pottharst-selmer}; the construction relies on the machinery of Selmer complexes developed by Nekov\'{a}\v{r} \cite{nekovar-selmer} and Pottharst \cite{pottharst-selmer}.

To be more precise, one can construct a family $\overline{D}$ of overconvergent $(\varphi, \Gamma)$-modules corresponding to the representation $\overline{M} := [M(\invs{F})^* \hatot M(\invs{G})^*](-\mathbf{j})$, where $-\mathbf{j}$ denotes the twist by the inverse of the universal character of $\invs{W}$, and it is shown in \cite{liu} that this family has a canonical triangulation (provided that $V_1$ and $V_2$ are small enough). In section \ref{selmer sheaf}, we define a Selmer complex with unramified local conditions away from $p$, and at $p$ we choose local conditions defined by the cohomology of a family $\overline{D}^+$ of two-dimensional sub $(\varphi, \Gamma)$-modules appearing in the triangulation of $\overline{D}$; at classical weights, the local condition at $p$ specialises to a so-called \emph{Panchishkin submodule}, i.e. the Hodge-Tate weights for $\overline{D}_{\mathbf{x}}^+$ (resp. $\overline{D}_{\mathbf{x}}/\overline{D}_{\mathbf{x}}^+$) are positive (resp. non-positive). Then, under some very mild conditions, this local condition corresponds to the Bloch--Kato local condition for the specialisation of the representation $\overline{M}$. We define $\invs{S}$ to be $\opn{H}^2$ of this Selmer complex.

In \cite{Urban} (and \cite[Appendix II]{AIFamilies}) Urban constructs a three variable $p$-adic $L$-function, denoted $L_p$, associated to $\invs{F}$ and $\invs{G}$ over $X:= V_1 \times V_2 \times \invs{W}$. This $p$-adic $L$-function is constructed via the theory of families of nearly overconvergent modular forms and interpolates the critical values of the Rankin--Selberg $L$-function at classical specialisations. We recall the interpolation property of $L_p$ in section \ref{ThreeVarSection}. In analogy with the Bloch--Kato conjecture -- which predicts that the Bloch--Kato Selmer group is controlled by the $L$-function for the corresponding representation -- we expect that the sheaf $\invs{S}$ is controlled by the three variable $p$-adic $L$-function. 

More precisely, the ring $\mathcal{O}(X)$ is a disjoint union of $p-1$ integral domains indexed by characters $\eta$ of the group $\left( \mbb{Z}/p\mbb{Z} \right)^{\times}$ (each of which correspond to an irreducible component of $\mathcal{W}$). For each character $\eta$, let $e_{\eta}$ denote the corresponding idempotent of $\ordd(X)$ projecting to the domain indexed by $\eta$. Since $X$ is quasi-Stein, a coherent sheaf on $X$ is determined by its global sections, so we will pass between these two perspectives freely. We expect the analogue of \cite[Theorem 11.6.4]{KLZ2015} to hold in our situation, namely
\begin{conjecture*}
Suppose that $e_{\eta} \cdot L_p \neq 0$. Under the hypotheses on $f$ and $g$ above, we expect that:
\begin{itemize}
    \item $e_{\eta} \cdot \mathcal{S}$ is a torsion $e_{\eta} \cdot \mathcal{O}(X)$-module.
    \item The $0$th Fitting ideal 
    \[
    \opn{Fitt}_0(e_{\eta} \cdot \invs{S})
    \]
    divides the ideal sheaf generated by the $p$-adic $L$-function $e_{\eta} \cdot L_p$.
\end{itemize}
\end{conjecture*}
Note that the factor $\Omega$ appearing in \emph{loc.cit.} is unnecessary for our formulation since it is invertible in $\mathcal{O}(X)$. A particular case of this conjecture is that the support of the sheaf is contained in the vanishing locus of the $p$-adic $L$-function. We prove a partial result in this direction.

\begin{theoremx} \label{SecondMainTheorem}
Let $\invs{S}_{k'}$ denote the specialisation of the above sheaf at $k'$ in the second variable. If $V_1$ is small enough and the above hypotheses hold for $f$ and $g$, then 
\[
\opn{supp}\invs{S}_{k'} \subset \{\mathbf{x} \in V_1 \times \{k'\} \times \invs{W} : L_p(\mathbf{x}) = 0 \}
\] 
where ``$\, \opn{supp}$'' denotes the support of a sheaf. 
\end{theoremx}

\begin{remark}
By Krull's principal ideal theorem, the vanishing locus of the three-variable $p$-adic $L$-function has codimension $\leq 1$ in $X$. Furthermore, since $L_p \neq 0$, there exists a character $\eta$ such that $e_{\eta} \cdot L_p \neq 0$. In this case, the vanishing locus of $e_{\eta} \cdot L_p$ has codimension one in $V_1 \times V_2 \times \mathcal{W}_{\eta}$, where $\mathcal{W}_{\eta}$ denotes the component of weight space indexed by $\eta$.
\end{remark}

To prove Theorem \ref{SecondMainTheorem}, we actually show that if $\mathbf{x}$ is a point in $V_1 \times \{k'\} \times \invs{W}$ and $L_p(\mathbf{x}) \neq 0$ then the group $\selmergrp{2}{\overline{M}_{\mathbf{x}}}$ vanishes. Here $\selmergrp{2}{\overline{M}_{\mathbf{x}}}$ is the cohomology in degree $2$ of a certain Selmer complex attached to the representation $\overline{M}_{\mathbf{x}}$ (see \S \ref{ConvolutionOfTwo}), which will be shown to coincide with the specialisation of $\invs{S}$ at $\mathbf{x}$. To show that this group vanishes, we generalise the proof of Theorem 8.2.1 in \cite{lz-coleman} to non-classical specialisations; this relies heavily on the theory of $(\varphi, \Gamma)$-modules and involves a careful analysis of the Perrin-Riou logarithm (see section \ref{prlog}).

Unfortunately, with the current methods we were unable to prove a three-variable version of this result. Indeed, a crucial step in the proof relies on the fact that $\invs{G}_{k'}$ is the $p$-stabilisation of a classical modular form of level $N_2$, whose other $p$-stabilisation is noble. By putting the other $p$-stabilisation into a Coleman family, we obtain two linearly independent Euler systems which can be used to bound the Bloch-Kato Selmer group, rather than just the strict Selmer group (this is also the technique used in the proof of \cite[Theorem 8.2.1]{lz-coleman}). For a general (non-classical) weight $k_2$, the specialisation $\invs{G}_{k_2}$ will be the unique point on the eigencurve with associated Galois representation $M(\invs{G}_{k_2})$, so the above strategy will not work.

\subsection{Notation}

Throughout the paper fix a prime $p \geq 5$. If $K$ is a field then we often denote its absolute Galois group by $G_K = \Gal(K^{\sep}/K)$, where $K^{\sep}$ denotes a fixed maximal separable closure of $K$. 

Let $R$ be a topological ring and $G$ a topological group. We say $M$ is a $G$-module over $R$ (or an $R[G]$-module) if $M$ is a continuous $R$-module equipped with a continuous homomorphism $\rho: G \to \opn{Aut}_RM$. We will often work within the category of $R[G]$-modules. This is not an abelian category in general, but it is additive and has kernels and cokernels, so we can still talk about its derived category. If $M$ is an $R[G]$-module and the action is commutative (i.e. the map $\rho$ factors through $G^{\opn{ab}}$), then we write $M^{\iota}$ to mean the module $M$ with the action given by $g \cdot m = \rho(g^{-1})m$ for all $g \in G$ and $m \in M$.

We will often take $R$ to be a $\mbb{Q}_p$-Banach algebra (or more generally, the global sections of a rigid analytic space). In this case we write $R^\circ$ for the subring of power-bounded elements. When $R$ is a reduced affinoid algebra, this coincides with the unit ball with respect to the supremum norm. 

For an $R[G]$-module $M$, let $M^* \defeq \opn{Hom}_{\text{cont}}(M, R)$ denote the dual representation of $M$ and, where appropriate, we write $M(n)$ to mean the representation $M$ tensored with the $n$-th Tate twist. We fix a compatible system of $p$-th power roots of unity in $\bar{\mbb{Q}}_p$, so in the case where $M$ is a Galois representation, $M(1)$ is just $M$ twisted by the cyclotomic character $\chi_{\text{cycl}}$. In this paper, the cyclotomic character will always have Hodge-Tate weight $1$. 

If $M$ is an $R[G]$-module, then we denote its $i$-th group cohomology by $\opn{H}^i(G, M)$. If $G_K$ is the absolute Galois group of a field $K$ then we will also sometimes write $\cohom{i}{K}{M}$ for $\opn{H}^i(G_K, M)$. 

When talking about left (resp. right) exact functors $F$, we write $\mathbf{R}F$ (resp. $\mathbf{L}F$) for the right (resp. left) derived functors of $F$. In particular, if $M$ is a $R[G]$-module then we write $\rcont{G}{M}$ for the image of the complex of continuous cochains of $M$ in the derived category of $R$-modules. 

If $X$ is an object defined over a ring $R$ and we have a homomorphism $R \to R'$, then we denote the base change of $X$ to $R'$ by $X_{R'}$.

For a positive integer $m$, we let $\mu_m^\circ$ denote the group scheme (over $\mbb{Q}$) of $m$-th roots of unity. 

Finally, we note that throughout the paper, any \'{e}tale cohomology group refers to continuous \'etale cohomology in the sense of Jannsen \cite{Jannsen1988}.

\subsection{Acknowledgements}
This paper grew out of our group project at the 2018 Arizona Winter School. We would like to thank David Loeffler and Sarah Zerbes for suggesting the project that led to the present article and for offering us their invaluable advice and guidance. We are grateful for their continued encouragements during the preparation of this paper. We also thank Rodolfo Venerucci for many helpful conversations and suggestions. Special thanks are also due to the organizers of the Arizona Winter School for making the winter school an enjoyable and fruitful experience. AG would like to thank Ashwin Iyengar and Pol van Hoften for their helpful comments and suggestions. The authors would like to thank the anonymous referees for their helpful suggestions. 

The Winter School was supported by NSF grant DMS-1504537 and by the Clay Mathematics Institute. AG was supported by the Engineering and Physical Sciences Research Council [EP/L015234/1], the EPSRC Centre for Doctoral Training in Geometry and Number Theory (The London School of Geometry and Number Theory), University College London and Imperial College London. DG was supported in part by Royal Society grant RP\textbackslash{}EA\textbackslash{}180020. YX was supported by Harvard University scholarships.


\section{Modular curves} \label{ModularCurves}

In this section we will define the modular curves that will be used throughout the paper. Let $\invs{H}^{\pm} \defeq \mbb{C} - \mbb{R}$ denote the upper and lower complex half-space and denote the finite adeles of $\mbb{Q}$ by $\mbb{A}_f$.  For a compact open subgroup $K \subset \opn{GL}_2(\mbb{A}_f)$, we let 
\begin{equation} \label{ShimuraVariety}
Y_K \defeq \opn{GL}_2(\mbb{Q}) \backslash \left[ \invs{H}^{\pm} \times \opn{GL}_2(\mbb{A}_f)/K \right] \,.
\end{equation}
We assume that $K$ is sufficiently small so that $Y_K$ has the structure of a
Shimura variety. This Shimura variety has a canonical model over $\mbb{Q}$ (which we will also denote by $Y_K$) and we will refer to this as the modular curve of level $K$. In this paper, we are interested in the following choices of $K$. 

Let $m, N$ be two positive integers such that $m(N+1) \geq 5$. Then the subgroup 
\[
K_{m, N} \defeq \left\{ \left( \begin{array}{cc} a & b \\ c & d \end{array} \right) \in \opn{GL}_2(\widehat{\mbb{Z}}) \left| \begin{array}{ccc} a \equiv 1, & b \equiv 0 & \mod{m\widehat{\mbb{Z}}} \\ c \equiv 0, & d \equiv 1 & \mod{mN\widehat{\mbb{Z}}} \end{array} \right. \right\}
\]
is sufficiently small and we denote the corresponding modular curve by $Y(m, mN) \defeq Y_{K_{m, N}}$. If $m=1$ we simply denote this curve by $Y_1(N)$. 
The modular curve $Y(m, mN)$ represents the contravariant functor taking a $\mbb{Q}$-scheme $S$ to the set of isomorphism classes of triples $(E, P, Q)$, where $E/S$ is an elliptic scheme, $P$ is a torsion section of order $m$ and $Q$ is a torsion section of order $mN$, such that $P$ and $Q$ are linearly independent, in the sense that the map $\mbb{Z}/m\mbb{Z} \times \mbb{Z}/mN\mbb{Z} \to E(S)$ given by $(a, b) \mapsto aP +bQ$ is injective. There is a natural morphism $Y(m, mN) \to \mu_m^\circ$ given by the Weil pairing on the points $P, NQ$, and the fibres of this map are smooth, geometrically connected curves.

For an integer $N \geq 1$ not divisible by $p$, we also set 
\[
K_1(N(p)) \defeq \left\{ \left( \begin{array}{cc} a & b \\ c & d \end{array} \right) \in \opn{GL}_2(\widehat{\mbb{Z}}) \left| c \equiv 0 \mod{pN \widehat{\mbb{Z}}}, \; \; \; d \equiv 1 \mod{N \widehat{\mbb{Z}}} \right. \right\}.
\] 
This is a sufficiently small subgroup and we denote the corresponding modular curve by $Y_1(N(p))$. This has a moduli interpretation as the contravariant functor taking a $\mbb{Q}$-scheme $S$ to the set of isomorphism classes of triples $(E, P, C)$, where $E/S$ is an elliptic scheme, $P$ is a torsion section of order $N$ and $C$ is a finite flat subgroup scheme of $E[p]$ (the $p$-torsion of $E$) of order $p$. 

It will also be useful to introduce several maps between these modular curves.
\begin{itemize}
\item For a positive integer $d$, we define the following map
\begin{equation} \label{alphad}
\Xi_d \colon Y(dm, dmN) \rightarrow Y(m, m N)
\end{equation}
as the morphism which sends a triple $(E, P, Q)$ to the triple $(E/\langle mP \rangle, P \! \! \mod{mP}, dQ \! \! \mod{mP})$, where $\langle mP \rangle$ denotes the cyclic subgroup generated by $mP$.
\item Recall $\mu_m^\circ$ denotes the group scheme (over $\mbb{Q}$) of $m$'th roots of unity. We define the following map 
\begin{equation} \label{teem}
t_m \colon Y(m, mN) \rightarrow Y_1(N) \times_{\mbb{Q}} \mu_m^\circ
\end{equation}
as the morphism given by $(E, P, Q) \mapsto ((E/\langle P \rangle, mQ \! \! \mod{P}), \langle P, NQ \rangle)$, where $\langle -, - \rangle$ denotes the Weil pairing on $E[m]$ and $\langle P \rangle$ is the subgroup generated by $P$. 
\item Let $N'$ be a positive integer dividing $N$. We define the following map 
\begin{equation} \label{esseN}
s_{N'} \colon Y_1(Np) \rightarrow Y_1(N'(p))
\end{equation}
to be the morphism sending $(E, Q)$ to $\left(E, \frac{Np}{N'} \cdot Q, \langle N \cdot Q \rangle \right)$, where $\langle N \cdot Q \rangle$ denotes the cyclic group scheme generated by the $p$-torsion section $N \cdot Q$. 
\end{itemize}
The first two maps are compatible in the following sense.

\begin{lemma} \label{CoresCompatibility}
Let $m, N$ be two positive integers with $m(N+1) \geq 5$ and let $d$ be a positive integer. Then we have the following commutative diagram
\[
\begin{tikzcd}
{Y(dm, dmN)} \arrow[d, "t_{dm}"'] \arrow[r, "\Xi_d"] & {Y(m, mN)} \arrow[d, "t_{m}"] \\
Y_1(N) \times \mu_{dm}^\circ \arrow[r] & Y_1(N) \times \mu_{m}^\circ
\end{tikzcd}
\]
where the bottom map is induced from the $d$-th power map $\mu_{dm}^\circ \to \mu_{m}^\circ$.
\end{lemma}
\begin{proof}
Immediate from the definitions. 
\end{proof}


\section{Families of modular forms and Galois representations}
\label{families}

\subsection{Weight space}
\begin{defn}
Let $\Lambda \defeq \Zp \db{\Zp^{\times}}$. The weight space $\WW$ is defined to be the rigid generic fibre of the formal spectrum $\opn{Spf}\Lambda$. It represents the functor taking a rigid analytic space $X$ over $\opn{Sp}\mbb{Q}_p$ to the set $\opn{Hom}_{\text{cont}}(\mbb{Z}_p^{\times}, \mathcal{O}_X(X)^{\times})$. Let  
\[ \kappa \colon \Zp^{\times} \to \Lambda^{\times} \subset \OO_\WW(\WW)^{\times} \]
denote the tautological character.
\end{defn}
The space $\WW$ is isomorphic to a union of $p-1$ wide open discs (recall that we have assumed $p>2$).

\begin{defn}\label{WW_E}
Let $E$ be a finite extension of $\Qp$ with ring of integers
$\OO_E$, and let
$\UU$ be a wide open disc in $\WW_E \defeq \WW \times_{\opn{Sp}(\Qp)} \opn{Sp}(E)$. Define $\Lambda_{\invs{U}} \defeq \OO_{\WW}(\UU)^{\circ}$, the subring of power-bounded elements of $\OO_{\WW}(\UU)$ (so $\Lambda$ is non-canonically isomorphic to $\ordd_E \db{t}$), and write 
\[ \kappa_\UU \colon \Zp^{\times} \to \Lambda_\UU^{\times} \]
for the map induced by $\kappa$.
\end{defn}

\begin{defn} \label{accessible}
The \emph{$m$-accessible} part of the weight space,
denoted $\WW_m$, is
the union of wide open discs defined by the inequality
\[ \left| \kappa(1+p^{m+1}) - 1 \right|^{p-1} < |p| \,. \]
\end{defn}
We will eventually restrict our attention to $\WW_0$.

\begin{defn}
A \emph{classical point} of $\WW$ is a point corresponding to the character $z \mapsto z^k$ for some nonnegative integer $k$.
\end{defn}
\begin{remark}
The previous definition is an abuse of notation, since the weight $z \mapsto \chi(z) z^k$, for $\chi$ a finite order character, can be the weight of a point on the eigencurve corresponding to a classical modular form. But we will not need to consider this more general class of weights.
\end{remark}

\subsection{Families of overconvergent modular forms}
\begin{defn}
Let $E$ be a finite extension of $\Qp$ with ring of integers
$\OO_E$, and let
$\UU \subseteq (\WW_0)_E$
be a wide open disc containing a classical point.
A \emph{Coleman family} $\FF$ over $\UU$ (of tame level $N$) is formal power
series $\sum_{n=1}^{\infty} a_n(\FF) q^n \in q \Lambda_\UU \db{q}$ satisfying the following properties:
\begin{enumerate}
\item $a_1(\FF) = 1$ and $a_p(\FF) \in \Lambda_\UU[\frac{1}{p}]^{\times}$.
\item For all but finitely many classical weights
$k$ contained in $\UU$, the restriction of $\FF$ to
$k$ is the $q$-expansion of a classical modular form of weight
$k+2$ and level $\Gamma_1(N) \cap \Gamma_0(p)$
that is a normalized eigenform for the Hecke operators (away from $Np$).
\end{enumerate}
We denote the character associated to $\invs{F}$ by $\varepsilon_{\invs{F}}$, so that for all but finitely many classical weights $k$ in $\invs{U}$ the specialisation of $\varepsilon_{\invs{F}}$ at $k$ coincides with the nebentypus of $\invs{F}_k$. 
\end{defn}

The following definition gives a criterion for when a modular form lies in a Coleman family. 

\begin{definition} \label{DefinitionOfNoble} 
We say that a cuspidal eigenform $f$ of level $\Gamma_1(N) \cap \Gamma_0(p)$ and weight $k+2$ is \emph{noble} if the following two conditions are satisfied.
\begin{itemize}
\item $f$ is the $p$-stabilisation of normalised cuspidal newform $f'$ of level $\Gamma_1(N)$ such that the roots $\{\alpha_{f'}, \beta_{f'}\}$ of the Hecke polynomial 
\[
X^2 - a_p(f')X + p^{k+1}\varepsilon_{f'}(p)
\]
are distinct. Here $a_p(f')$ is the $p$-th Fourier coefficient of $f'$ and $\varepsilon_{f'}$ is the nebentypus. 
\item If the $U_p$-eigenvalue of $f$ has $p$-adic valuation $k+1$ then the local Galois representation attached to $f'$ at $p$ is not the direct sum of two characters. 
\end{itemize}  
\end{definition}

\begin{lemma}
Let $f$ be a noble eigenform of weight $k+2$. Then for any sufficiently small $\UU \owns k$ in $\WW$,
there is a unique Coleman family $\FF$ over $\UU$
such that $\FF_k = f$.
\end{lemma}
\begin{proof}
This is essentially proved in \cite[Lemma 2.8]{bellpadic}.
In particular, the lemma shows that if
$x$ is an $E$-point on the eigencurve of
weight $k+2$,
then there is a neighborhood $\mathcal{V} \owns x$
and an open disc $\UU \subseteq \WW$
such that $\mathcal{V} \to \UU$ is finite flat of
degree $\dim M^{\dagger}_{k+2,(x)}$,
where $M^{\dagger}_{k+2,(x)}$ is the space
of overconvergent modular forms
of weight $k+2$ that are Hecke eigenforms
with eigenvalues given by $x$.
Suppose $x$ is noble; then it has slope less than $k+1$, so Coleman's control
theorem \cite[Thm.~6.1]{coleman-overconvergent} implies that $M^{\dagger}_{k+2,(x)}$
consists of classical modular forms.  Moreover, $x$ corresponds
to a newform, so the subspace
of classical modular forms is one-dimensional (see for example
\cite[Thm.~VIII.3.3]{lang-modular}).
\end{proof}

\subsection{Locally analytic distribution modules}

Now we begin defining a family of Galois representations on $\WW$
as in \cite[\S 4]{lz-coleman}.

Let $Y = Y_1(N(p))$ be the modular curve at level $\Gamma_1(N) \cap \Gamma_0(p)$,
and let $\pi \colon \EE \to Y$ be the universal elliptic curve over
$Y$.
Let
\[ \SH \defeq \mbf{R}^1 \pi_* \Zp(1) \]
be the relative Tate module of $\EE$.
We will define several pro-sheaves of ``functions and distributions
on $\SH$.''  By this, we mean the following. Let $Y(p^{\infty}, p^{\infty}N)$ denote the pro-scheme $\varprojlim_n Y(p^n, p^nN)$ and $t \colon Y(p^{\infty}, p^{\infty}N) \to Y$ the natural projection; it is a Galois covering, and its Galois group can be identified with the Iwahori subgroup $U_0(p) \subset \GL_2(\Zp)$ (with respect to the standard Borel). 

The pro-sheaf $t^* \SH$ is canonically isomorphic to the constant
pro-sheaf $\underline{H}$, where $H=\Zp^2$. We will define several spaces of functions and distributions
on subsets of $H$ that are equipped with actions of $U_0(p)$. Since a $U_0(p)$-module determines a 
$U_0(p)$-equivariant pro-sheaf
on $Y(p^{\infty},p^{\infty}N)$, these spaces will descend to pro-sheaves on $Y$.

Firstly, we recall the definition of two locally analytic distribution modules following \cite[\S 4.2]{lz-coleman}.

\begin{defn}
Let $T_0$, $T_0'$ be the subsets of $H$
defined by
\[ T_0 \defeq \Zp^{\times} \times \Zp, \quad T_0' \defeq p\Zp \times \Zp^{\times} \, \]
and let $\Sigma_0(p)$, $\Sigma_0'(p)$ be the submonoids
of $M_2(\Zp)$
defined by
\[ \Sigma_0(p) \defeq \begin{pmatrix} \Zp^{\times} & \Zp \\ p\Zp & \Zp \end{pmatrix}, \quad
\Sigma_0'(p) \defeq \begin{pmatrix} \Zp & \Zp \\ p\Zp & \Zp^{\times} \end{pmatrix}. \]
The monoids $\Sigma_0(p)$ and $\Sigma_0'(p)$ act on the right on $T_0$ and $T_0'$, respectively.
\end{defn}

Let $R$ be a complete topological $\Zp$-algebra, and let
$w \colon \Zp^{\times} \to R^{\times}$ be a continuous homomorphism.
Suppose there exists an integer $m \geq 0$
such that the restriction of $w$ to $1+p^{m+1} \Zp$ is analytic. We are primarily interested in the following cases:

\begin{enumerate}
\item $R = \Lambda_\UU$, $w=\kappa_\UU$ for some finite extension
$E/\Qp$ and some $\UU \subset (\WW_m)_E$. \label{R family}
\item $R = \OO_E$, $w(z) = z^k$ for some finite extension $E/\Qp$ and some nonnegative integer $k$. \label{R classical pt}
\end{enumerate}

\begin{defn} \label{locally analytic}
Let $T$ be either $T_0$ or $T_0'$.
Let $A^{\circ}_{w,m}(T)$ denote the space of functions $f \colon T \to R$
satisfying the following properties:
\begin{enumerate}
\item The function $f$ is homogeneous of weight $w$, i.e. $f(\lambda v) = w(\lambda) f(v)$ for any $v \in T$, $\lambda \in \Zp^{\times}$.
\item The function $f$ is analytic on discs of radius $p^{-m}$, i.e. for any $v \in T$, the restriction of $f$ to $v+p^m T$ is given by a power
series with coefficients in $R$.
\end{enumerate}
Let
\[ D^{\circ}_{w,m}(T) \defeq \Hom_{R, \mathrm{cont}}(A^{\circ}_{w,m}(T), R) \]
\[ D_{w,m}(T) \defeq D^{\circ}_{w,m}(T)\left[ 1/p \right] \,. \]
When $R=\Lambda_{\UU}$, $w=\kappa_\UU$,
we will denote the modules by
\begin{center}
$A^{\circ}_{\UU,m}$, $D^{\circ}_{\UU,m}$, $D_{\UU,m}$.
\end{center}
When $R = \OO_E$, $w(z)=z^k$, we will denote the modules by
\begin{center}
$A^{\circ}_{k,m}$, $D^{\circ}_{k,m}$, $D_{k,m}$.
\end{center}
\end{defn}
The modules $A^{\circ}_{w,m}(T)$, $D^{\circ}_{w,m}(T)$,
$D_{w,m}(T)$ inherit an action of $\Sigma_0(p)$ or $\Sigma_0'(p)$ from the action on $T$. If the disc $\UU$ contains the point corresponding to the homomorphism
$z \mapsto z^k$, then the specialization map $\Lambda_{\UU} \to \Zp$
induces a homomorphism
\[ D^{\circ}_{\UU,m} \to D^{\circ}_{k,m} \]
and similarly there are specialization maps with $D^{\circ}$ replaced by
$A^{\circ}$ or $D$.

As mentioned at the beginning of this subsection,
each of the modules defined above determines a pro-sheaf
on $Y$.  We let
$\DD^{\circ}_{w,m}(\SH_0)$, $\DD^{\circ}_{w,m}(\SH_0')$,
be the pro-sheaves corresponding
to $D^{\circ}_{w,m}(T_0)$, $D^{\circ}_{w,m}(T_0')$, respectively.

\subsection{Galois representations}
Now we define families of Galois representations
coming from the cohomology of the sheaves defined above. For a wide open disc $\UU \subset (\WW_0)_E$ we set $B_{\UU} \defeq \Lambda_{\UU}[\frac{1}{p}]$.
\begin{defn} As before, let $Y = Y_1(N(p))$ denote the modular curve of level $\Gamma_1(N) \cap \Gamma_0(p)$. Set 
\begin{IEEEeqnarray*}{rCl}
M^{\circ}_{w,m}(\SH_0) & \defeq & \opn{H}^1_{\et}(Y_{\overline{\mbb{Q}}}, \DD^{\circ}_{w,m}(\SH_0))(-\kappa_{\UU}) \\
M^{\circ}_{w,m}(\SH_0') & \defeq  & \opn{H}^1_{\et}(Y_{\overline{\mbb{Q}}},
\DD^{\circ}_{w,m}(\SH_0'))(1).
\end{IEEEeqnarray*}
\end{defn}

\begin{prop}[{\cite[Thm.~4.6.6]{lz-coleman}}]
Let $f_0$ be a noble eigenform of weight $k_0 + 2$, and let $\FF$ be the Coleman family passing through $f_0$.  If the disc $\UU \owns k_0$ is sufficiently small,
then:
\begin{enumerate}
\item The modules
\[ M_{\UU}(\FF) \defeq M_{\UU,0}(\SH_0) \left[ T_n = a_n(\FF) \; \;  \forall n \ge 1 \right] \]
\[ M_{\UU}(\FF)^* \defeq M_{\UU,0}(\SH_0') \left[ T_n' = a_n(\FF) \; \;  \forall n \ge 1 \right] \]
are direct summands (as $B_{\UU}$-modules) of $M_{\UU,0}(\SH_0)$
and $M_{\UU,0}(\SH_0')$ respectively, where $[-]$ stands for isotypic component and $T_n'$ is the transpose of the usual Hecke operator. Each is free of rank $2$ over $B_{\UU}$.
\item The Ohta pairing (see \cite[\S 4.3]{lz-coleman}) induces an isomorphism of $B_{\UU}[G_{\QQ}]$-modules
\[ M_{\UU}(\FF)^* \cong \Hom_{B_{\UU}}(M_{\UU}(\FF),B_{\UU}) \,. \]
\end{enumerate}
\end{prop}

Given two Coleman families $\FF$ and $\GG$ defined over $\UU_1, \UU_2 \subset \WW_E$, respectively, we will write $M := M_{\UU_1}(\FF)^* \hatot M_{\UU_2}(\GG)^*$ for the family of Galois representations on $\UU_1 \times \UU_2$ given by the $B_{\UU_1} \hatot B_{\UU_2}$-module 
\[
M_{\UU_1}(\FF)^* \hat{\otimes}_E M_{\UU_2}(\GG)^*
\]
and to ease notation, we will often omit the subscripts when the spaces $\UU_1$ and $\UU_2$ are clear. Furthermore we will often restrict this representation to open affinoids $V_1 \subset \UU_1$ and $V_2 \subset \UU_2$; in this case $M$ is a Banach module over the affinoid algebra $\ordd_{\WW}(V_1 \times V_2)$ that is free of rank four.

\begin{definition}
Let $A$ be a $\mbb{Q}_p$-affinoid algebra and $M$ an $A$-valued (continuous) representation of $G_{\mbb{Q}}$. Let $D^{\mathrm{la}}(\Gamma, A)$ denote the space of locally analytic distributions with values in $A$; this comes equipped with an action of $G_{\mbb{Q}}$ given by
\begin{equation} \label{actionofdist}
\int_{\Gamma} f \; d(g \cdot \mu) := \int_{\Gamma} f([g]^{-1}x) \;  d\mu(x)
\end{equation}
where $[g]$ denotes the image of $g \in G_{\mbb{Q}}$ in $\Gamma = \Gal(\mbb{Q}(\mu_{p^{\infty}})/\mbb{Q})$, and is isomorphic (as $A[G_{\mbb{Q}}]$-modules) to $\ordd_{\WW}(\WW)^{\iota} \hatot A$. The \emph{cyclotomic deformation} of $M$ is defined to be 
\[
M(-\kappa) \defeq D^{\mathrm{la}}(\Gamma, M) \defeq M \hatot_{\mbb{Q}_p} D^{\mathrm{la}}(\Gamma, A).
\]
with the diagonal Galois action.

Similarly, for any $\lambda \in \RR_{\ge 0}$,
let $D_{\lambda}(\Gamma,\Qp)$ be the space of $\Qp$-valued distributions
on $\Gamma$ of order $\lambda$ as in \cite[\S II.3]{colmez-fonctions},
with Galois action given by the same formula in (\ref{actionofdist}).
Define $D_{\lambda}(\Gamma,M) \defeq D_{\lambda}(\Gamma,\Qp) \hat{\otimes}_{\Qp} M$.
\end{definition}

\subsection{Some properties of locally analytic distribution modules}
We mention some properties of the modules defined above that
will be useful in section \ref{beilinson-flach}.

\begin{defn} \label{CGdef}
Define $\Lambda(H)$ to be the space of continuous $\Zp$-valued
distributions on $H$, and let $\Lambda(\SH)$ be the corresponding
pro-sheaf on $Y$. This coincides with the sheaf of Iwasawa modules for $\SH$, i.e. $\Lambda(\SH)$ is the pro-system of \'{e}tale sheaves corresonding to the inverse system $\left( \mbb{Z}/p^n \mbb{Z} [\SH /p^n \SH ] \right)_{n \geq 1}$ with the natural transition maps.

For any nonnegative integer $k$, let $\TSym^k H$ be the
space of degree $k$ symmetric tensors over $H$, i.e.~it is the subgroup of $H^{\otimes k}$ that
is invariant under the action of the symmetric group $S_k$.
Let $\TSym^k \SH$ be the corresponding pro-sheaf on $Y$.

For $j \geq 0$, set $\Lambda^{[j]}(\SH) \defeq \Lambda(\SH) \otimes \opn{TSym}^j\SH$ and $\Lambda^{[j, j]} = \Lambda^{[j]}(\SH) \boxtimes \Lambda^{[j]}(\SH)$. Then there is a Clebsch-Gordon map (see \cite[\S 3.2]{lz-coleman})
\[
\opn{CG}^{[j]} \colon \Lambda(\SH) \rightarrow \left( \Lambda^{[j]}(\SH) \hatot \Lambda^{[j]}(\SH) \right) (-j).
\] 

\end{defn}
For any $\UU$, $m$, there is a natural restriction map
\[ \Lambda(\SH) \to \DD^{\circ}_{U,m}(T) \, \]
and for any nonnegative integer $k$, $\TSym^k H$ can be
identified with the space of distributions on
homogeneous degree $k$ polynomial functions on $H$.
Hence there is a natural surjection
\[ \DD^{\circ}_{k,m}(T) \to \TSym^k \SH \,. \]

\subsection{The three-variable $p$-adic $L$-function} \label{ThreeVarSection}

Let $f$ and $g$ be two normalised cuspidal eigenforms of weights $k +2, k'+2$ and levels $\Gamma_1(N_1)$ and $\Gamma_1(N_2)$ respectively, where $k > k' \geq 0$. Let $\chi$ be a Dirichlet character of conductor $N_{\chi}$ and suppose that $p$ does not divide $N_1 \cdot N_2$. To this data one has the associated (imprimitive) Rankin--Selberg $L$-function, defined as 
\[
L(f, g, \chi, s) = L_{(N_1 N_2 N_{\chi})}(\varepsilon_f \varepsilon_g \chi^2, 2s - 2 - k - k') \cdot \sum_{\substack{n \geq 1 \\ (n, N_{\chi}) = 1}} a_n(f) a_n(g) \chi(n) n^{-s}
\]
for $\opn{Re}(s)$ sufficiently large. Here the subscript $(N_1 N_2 N_{\chi})$ denotes the omission of the Euler factors at primes dividing $N_1 N_2 N_{\chi}$. This $L$-function differs from the automorphic $L$-function attached to the representation $\pi_f \otimes \pi_g \otimes \chi$ by only finitely many Euler factors. Since we have assumed $k \neq k'$, the function $L(f, g, \chi, -)$ has analytic continuation to all of $\mbb{C}$ (see \cite[\S 2.1]{LoefflerRS}). 

Using the theory of nearly overconvergent families of modular forms as described in \cite{AIFamilies}, Urban has constructed a three-variable $p$-adic $L$-function which interpolates critical values of the above Rankin--Selberg $L$-function. More precisely, suppose that there exist noble $p$-stabilisations of $f$ and $g$ (as in Definition \ref{DefinitionOfNoble}) and let $\invs{F}$ and $\invs{G}$ be Coleman families over affinoid domains $V_1$ and $V_2$, passing through these $p$-stabilisations. We can shrink $V_1$ and $V_2$ to ensure that all classical specialisations of $\invs{F}$ and $\invs{G}$ are noble (see Remark \ref{NobleRemark}) -- if $\invs{F}_{k_1}$ denotes such a specialisation then we let $\invs{F}^\circ_{k_1}$ denote the associated newform, and similarly for $\invs{G}$. 

\begin{theorem}[Urban]
There exists an element $L_p(\invs{F}, \invs{G}, 1 + \mathbf{j}) \in \ordd(V_1 \times V_2 \times \invs{W})$ satisfying the following interpolation property:
\begin{itemize}
    \item For all integers $k_1, k_2, j$ satisfying $k_i \in V_i$ and $0 \leq k_2 + 1 \leq j \leq k_1$, and all Dirichlet characters $\chi$ of $p$-power conductor, we have 
    \[
    L_p(\invs{F}_{k_1}, \invs{G}_{k_2}, 1 + j + \chi) = C(\invs{F}_{k_1}, \invs{G}_{k_2}, 1 + j + \chi ) \cdot \frac{j! (j-k_2 - 1)! i^{k_1 - k_2}}{\pi^{2j + 1 - k_2}2^{2j + 2 + k_1 - k_2} \langle \invs{F}_{k_1}^{\circ}, \invs{F}_{k_1}^{\circ} \rangle_{N_1}} L(\invs{F}_{k_1}^{\circ}, \invs{G}_{k_2}^{\circ}, \chi^{-1}, 1+j)
    \]
    where $C(\invs{F}_{k_1}, \invs{G}_{k_2}, 1 + j + \chi )$ is an explicit product of Euler factors and Gauss sums. 
\end{itemize}
\end{theorem}
\begin{proof}
This follows from the interpolation property in \cite[Theorem 4.4.7]{Urban} (which is valid by the results of \cite[Appendix II]{AIFamilies}). The calculation is only for $N=1$ but its generalisation is immediate. See also the computation of the Rankin--Selberg period in \cite[Prop 2.10]{LoefflerRS}.
\end{proof}

In the following section we will recall the construction of the Beilinson--Flach classes in families. It turns out that this $p$-adic $L$-function $L_p$ is closely related to the images of these Beilinson--Flach classes under Perrin-Riou's ``big logarithm''. This will be important in Proposition \ref{proplinearindependence} later on. 


\section{Beilinson--Flach classes} \label{beilinson-flach}

In this section we recall the construction of classes 
\[
_c \invs{BF}_{m, 1}^{[\invs{F}, \invs{G}]} \in \opn{H}^1(\mbb{Q}(\mu_m), D^{\mathrm{la}}(\Gamma, M) )
\]
where $M = M_{V_1}(\invs{F})^* \hatot M_{V_2}(\invs{G})^*$ following \cite{lz-coleman}. These classes are obtained from so-called Rankin--Iwasawa classes under the pushforward of a certain sequence of morphisms. In particular we show that these classes satisfy certain norm relations which interpolate the (tame) Euler system relations at classical weights.

None of the results in this section are new, apart from perhaps Proposition \ref{normthm} and \S \ref{ESinFamilies}, although we suspect this is already known to the experts.

\subsection{Rankin--Iwasawa classes}

Let $(\invs{E}, P, Q)$ denote the universal triple over the curve $Y\defeq Y(m, mN)$ as defined in section \ref{ModularCurves}, and recall that
\[
\SH = \SH_{\mbb{Z}_p} \defeq \mbf{R}^1\pi_* \mbb{Z}_p(1) 
\] 
denotes the relative $p$-adic Tate module of $\invs{E}/Y$. Here $\pi\colon \invs{E} \to Y$ denotes the structure map and $\mbb{Z}_p(1)$ is the Tate twist by the cyclotomic character. This is a locally free \'{e}tale pro-sheaf on $Y$ of rank $2$.

Let $c \geq 1$ be an integer prime to $6mN$. In \cite{Kings2015}, Kings constructs Eisenstein classes ${_c\opn{Eis}^k_{\mbb{Q}_p}}$ arising from motivic classes whose de Rham realisations recover the usual Eisenstein series of weight $k+2$ (see also \cite[\S 4]{KLZ-modular}). In addition to this, he constructs so-called \emph{Eisenstein--Iwasawa} classes 
\[
{_c \invs{EI}_{m, mN}} \in \opn{H}^1_{\et}(Y(m, mN), \Lambda(\SH)(1) )
\]  
which interpolate ${ _c\opn{Eis}^k_{\mbb{Q}_p}}$ via the ``moment maps''
\[
\opn{mom}^k \colon \Lambda(\SH) \to \opn{TSym}^k \SH.
\]
From these classes one obtains Rankin--Iwasawa classes in the following way.

\begin{definition}
Let $c \geq 1$ be an integer that is coprime to $6mN$. We define the \emph{Rankin--Iwasawa} class to be 
\[
_c \invs{RI}_{m, mN, 1}^{[j]} \defeq \left[ (u_1)_* \circ \Delta_* \circ \opn{CG}^{[j]} \right](_c\invs{EI}_{m, mN})
\]
which lies in the cohomology group $\opn{H}^3_{\et}(Y(m, mN)^2, \Lambda^{[j, j]}(2-j) )$. Here  
\begin{itemize}
    \item $\opn{CG}^{[j]}$ is the Clebsch--Gordon map described in Definition \ref{CGdef}.
    \item $\Delta\colon Y(m, mN) \rightarrow Y(m, mN)^2$ denotes the diagonal embedding where, by abuse of notation, we write $Y(m, mN)^2$ for the fibre product 
    \[
    Y(m, mN) \times_{\mu_m^{\circ}} Y(m, mN).
    \]
    \item $u_1 \colon Y(m, mN)^2 \to Y(m, mN)^2$ denotes the automorphism which is the identity on the first factor and acts on the moduli interpretation as
    \[
    (E, P, Q) \mapsto (E, P + NQ, Q)
    \]
    on the second factor.
\end{itemize}
\end{definition}

The Rankin--Iwasawa classes satisfy the following norm compatibility relations.

\begin{proposition} \label{RankinIwasawaNorm}
Let $c \geq 1$ be an integer prime to $6Np$ and let $m$ be an integer prime to $6cN$. Let $l$ be a prime not dividing $6cNp$ and recall that we have defined the following morphism $\Xi_l\colon Y(lm, lmN) \to Y(m, mN)$ in (\ref{alphad}). 
\begin{enumerate}
\item If $l$ divides $m$ then the Rankin--Iwasawa classes satisfy the following norm compatibility relation
\[
(\Xi_l \times \Xi_l)_*(\rankiw{lm}{lmN}) = (U_l', U_l') \cdot {\rankiw{m}{mN}}.
\]
\item If $l$ does not divide $m$ then the Rankin--Iwasawa classes satisfy the following norm compatibility relation
\[
(\Xi_l \times \Xi_l)_*(\rankiw{lm}{lmN}) = \widetilde{Q}_l \cdot  {\rankiw{m}{mN}}
\]
where $\widetilde{Q}_l$ is the operator 
\begin{equation*}
\begin{aligned}
-l^j \sigma_l + (T_l', T_l') + ((l+1)l^j(\langle l \rangle^{-1} [l]_*, \langle l \rangle^{-1} [l]_*) - (\langle l \rangle^{-1} [l]_*, T_l'^2) - (T_l'^2, \langle l \rangle^{-1} [l]_*) )\sigma_l^{-1}  \\ \\ 
+ (\langle l^{-1} \rangle [l]_* T_l', \langle l^{-1} \rangle [l]_* T_l') \sigma_l^{-2} - l^{1+j}([l^2]_* \langle l^{-2} \rangle, [l^2]_* \langle l^{-2} \rangle)\sigma_l^{-3}  
\end{aligned}
\end{equation*}
and
\begin{itemize}
\item $T_l'$ (resp. $U_l'$) is the transpose of the usual Hecke operator $T_l$ (resp. $U_l$) on $Y(m, mN)$.
\item $[a]_* \colon \Lambda^{[j]}(\SH) \to \Lambda^{[j]}(\SH)$ is the map induced from multiplication by $a$ on the first factor and the identity on the second.
\item $\langle b \rangle$ is the diamond operator on $Y(m, mN)$ which acts on the moduli interpretation as $(E, P, Q) \mapsto (E, b^{-1}P, bQ)$.
\item $\sigma_l$ is the automorphism of $Y(m ,mN)$ which acts on the moduli interpretation as $(E, P, Q) \mapsto (E, lP, Q)$.
\end{itemize}
\end{enumerate}

\end{proposition}
\begin{proof}
In the notation of \cite[\S 5]{KLZ2015}, the map $(\Xi_l \times \Xi_l)_*$ is the composition $(\opn{pr}_1 \times \opn{pr}_1)_* (\hat{\opn{pr}}_2 \times \hat{\opn{pr}}_2)_*$ and if $l | m$ one can check that $(\opn{pr}_1 \times \opn{pr}_1)_*$ commutes with $(U_l', U_l')$. The first part then follows by combining the norm relations in Theorem 5.3.1 and Theorem 5.4.1 in \emph{op.~cit.} For the second part, this is just Proposition 5.6.1 in \emph{op.~cit.}
\end{proof}

\subsection{Beilinson--Flach classes in families} \label{bfclasses}
Let $E$ be a finite extension of $\mbb{Q}_p$ with ring of integers $\ordd_E$, and $\invs{U}_1, \invs{U}_2 \subset (\invs{W}_0)_E$ two wide open discs, where $\invs{W}_0 \subset \invs{W}$ is the wide open subspace of $0$-accessible weights (see Definition \ref{accessible}). Recall that $\Lambda(\SH'_0)$ and $\DD_{\UU_i}^{\circ}(\SH'_0) := \DD_{\UU_i, 0}^{\circ}(\SH'_0)$ are the sheaves of continuous (resp. locally analytic) $\mbb{Z}_p$-valued (resp. $\Lambda_{\UU_i}$-valued) distributions on $\SH'_0$, the subsheaf of $\SH$ which is locally isomorphic to $T'_0$.

Consider the map of sheaves 
\begin{equation} \label{sheafmap}
\Lambda^{[j, j]}(\SH_0') \rightarrow \invs{D}_{[\invs{U}_1, \invs{U}_2]}(\SH_0') \defeq \left( \invs{D}^{\circ}_{\invs{U}_i}(\SH_0') \boxtimes \invs{D}^{\circ}_{\invs{U}_2}(\SH_0') \left) \left[ \frac{1}{p} \right] \right. \right.
\end{equation} 
induced from the composition
\begin{equation} \label{OverconvergentComp}
\Lambda(\SH_0') \otimes \opn{TSym}^j\SH \to \invs{D}^{\circ}_{\invs{U}_i - j}(\SH_0') \otimes \opn{TSym}^j\SH \xrightarrow{\delta^*_j} \invs{D}_{\invs{U}_i}(\SH_0')
\end{equation}
described in \cite[Definition 5.3.1]{lz-coleman}. We will not need an explicit description of these maps, but we do note that we have the following commutative diagram.
\begin{lemma} \label{OverconvergentDiagram} 
We have the following commutative diagram of sheaves
\[
\begin{tikzcd}
\Lambda(\SH_0') \otimes \opn{TSym}^j\SH \arrow[r] \arrow[d, "{[l]_* \otimes \;  \opn{id}}"'] & \invs{D}_{\invs{U}_i}(\SH_0') \arrow[d, "l^{\kappa_i-j}"] \\
\Lambda(\SH_0') \otimes \opn{TSym}^j\SH \arrow[r] & \invs{D}_{\invs{U}_i}(\SH_0')
\end{tikzcd}
\]
where the horizontal arrows are the composition in (\ref{OverconvergentComp}) and, as usual, $l$ is a prime not dividing $Np$ and $\kappa_i$ is the universal character of $\invs{U}_i$.
\end{lemma}
\begin{proof}
One can check this \'{e}tale locally and this follows from the homogeneity condition in the definition of $D^{\circ}_{\invs{U}_i}$ and the fact that $\delta_j^* \circ ([l]_* \otimes \opn{id}) = l^{-j}([l]_* \otimes \opn{id}) \circ \delta_j^*$.
\end{proof}

We are now in a position to define the Beilinson--Flach classes. Consider the following composition, which we will denote by $\tau_m$
\[
\tau_m\colon Y(m, mNp)^2 \xrightarrow{t_m \times t_m} Y_1(Np)^2 \times \mu_m^{\circ} \rightarrow Y_1(N_1(p)) \times Y_1(N_2(p)) \times \mu_m^{\circ}
\] 
where $t_m \colon Y(m, mNp) \to Y_1(Np) \times \mu_m^{\circ}$ is the map defined in (\ref{teem}) and the second map is induced from the maps $s_{N_i}\colon Y_1(Np) \to Y_1(N_i(p))$ as defined in (\ref{esseN}).

\begin{definition}
We define the \emph{Beilinson--Flach} class
\[
\beilf{m} \in \opn{H}^3_{\et}(Y_1(N_1(p)) \times Y_1(N_2(p)) \times \mu_m^{\circ}, \invs{D}_{[\invs{U}_1, \invs{U}_2]}(\SH_0')(2-j) )
\]
to be the pushforward of $\rankiw{m}{mNp}$ under $\tau_m$ composed with the map in (\ref{sheafmap}).
\end{definition}

Let $\invs{F}$ and $\invs{G}$ be Coleman families over $\invs{U}_1$ and $\invs{U}_2$ respectively. To specialise the Beilinson--Flach classes at $\invs{F}$ and $\invs{G}$ one introduces the following differential operators ${\nabla_i \choose j}$ on $\Lambda_{\invs{U}_i}[1/p]$ given by the formula 
\[
{\nabla_i \choose j} \defeq \frac{1}{j!}\prod_{k = 0}^{j-1}{(\nabla_i - k)}
\]
where $\nabla_i$ is given by $(\nabla_if)(x) = \left. \frac{d}{dt}f(tx) \right|_{t=1}$ (see \cite[Proposition 5.1.2]{lz-coleman} for more details). Let $V_1$ and $V_2$ be open affinoid subdomains in $\invs{U}_1$ and $\invs{U}_2$ respectively and $j \geq 0$ an integer not contained in either $V_1$ or $V_2$. Then the operators ${\nabla_1 \choose j}$ and ${\nabla_2 \choose j}$ are injective and there exist unique classes 
\[
\beilfj{m}{\invs{G}} \in \opn{H}^1\left(\mbb{Q}(\mu_m), [M_{V_1}(\invs{F})^* \hatot M_{V_2}(\invs{G})^*](-j) \right)
\]
such that ${\nabla_1 \choose j}{\nabla_2 \choose j} {\beilfj{m}{\invs{G}}}$ equals the image of $\beilf{m}$ under the Abel--Jacobi map $\opn{AJ}_{\invs{F}, \invs{G}}$ defined below.

\begin{definition} 
The Abel--Jacobi map $\opn{AJ}_{\invs{F}, \invs{G}}$ is defined to be the composition 
\[
\begin{aligned}
\opn{H}^3_{\et}(Y_1(N_1(p)) & \times Y_1(N_2(p)) \times \mu_m^{\circ}, \invs{D}_{[\invs{U}_1, \invs{U}_2]}(\SH_0')(2-j)) \\
& \longrightarrow  \opn{H}^1\left(\mbb{Q}(\mu_m), \opn{H}^2_{\et}(Y_1(N_1(p))_{\bar{\mbb{Q}}} \times Y_1(N_2(p))_{\bar{\mbb{Q}}}, \invs{D}_{[\invs{U}_1, \invs{U}_2]}(\SH_0')(2-j) \right) \\
& \xrightarrow{\sim} \opn{H}^1(\mbb{Q}(\mu_m), \opn{H}^1_{\et}(Y_1(N_1(p))_{\bar{\mbb{Q}}}, \invs{D}_{\invs{U}_1}(\SH_0')(1)) \hatot \opn{H}^1_{\et}(Y_1(N_2(p))_{\bar{\mbb{Q}}}, \invs{D}_{\invs{U}_2}(\SH_0')(1))(-j)) \\
& \longrightarrow \opn{H}^1\left(\mbb{Q}(\mu_m), \left[ M_{V_1}(\invs{F})^* \hatot M_{V_2}(\invs{G})^* \right] (-j) \right)
\end{aligned}
\]
where the first map arises from the Leray spectral sequence (using the fact that $Y_1(N_1(p))_{\bar{\mbb{Q}}} \times Y_1(N_2(p))_{\bar{\mbb{Q}}}$ is an affine scheme, so its \'{e}tale cohomology vanishes in degree $3$ and above); the second isomorphism is the K\"{u}nneth formula (again using the fact that $Y_1(N_i(p))_{\bar{\mbb{Q}}}$ is affine) and the third map is the projection down to the Galois representations associated to $\invs{F}$ and $\invs{G}$.
\end{definition}

The Beilinson--Flach classes associated to $\invs{F}$ and $\invs{G}$ satisfy norm compatibility relations similar to those for the Rankin--Iwasawa classes.

\begin{proposition} \label{normrelationfortwo}
Let $c \geq 1$ be an integer prime to $6Np$ and let $m$ be an integer prime to $6cN$. Let $l$ be a prime not dividing $6cNp$ and let $\invs{F}$ and $\invs{G}$ be two Coleman families over the affinoid subdomains $V_1$ and $V_2$ respecively. Suppose that $j \geq 0$ is an integer not contained in $V_1$ or $V_2$. 
\begin{enumerate}
\item If $l$ divides $m$ then the Beilinson--Flach classes satisfy the following norm-compatibility relation
\[
\opn{cores}^{\mbb{Q}(\mu_{ml})}_{\mbb{Q}(\mu_m)} \left( \beilfj{ml}{\invs{G}} \right) = \left(a_l(\invs{F})a_l(\invs{G})\right){\beilfj{m}{\invs{G}}}
\]
\item If $l$ doesn't divide $m$ then the Beilinson--Flach classes satisfy the following norm-compatibility relation
\[
\opn{cores}^{\mbb{Q}(\mu_{ml})}_{\mbb{Q}(\mu_m)} \left( \beilfj{ml}{\invs{G}} \right) = \invs{Q}_l(l^{-j}\sigma_l^{-1}) \cdot {\beilfj{m}{\invs{G}}}
\]
where $\invs{Q}_l(X) \in \ordd(V_1 \times V_2)[X, X^{-1}]$ is the polynomial
\begin{equation*}
\begin{aligned}
\invs{Q}_l(X) = &  -X^{-1} + a_l(\invs{F})a_l(\invs{G}) \\ \\ &  + ((l+1)l^{\kappa_1 + \kappa_2}\varepsilon_{\invs{F}}(l)\varepsilon_{\invs{G}}(l) - l^{\kappa_1}\varepsilon_{\invs{F}}(l)a_l(\invs{G})^2 - l^{\kappa_2}\varepsilon_{\invs{G}}(l)a_l(\invs{F})^2)X \\ \\ & + l^{\kappa_1 + \kappa_2}\varepsilon_{\invs{F}}(l)\varepsilon_{\invs{G}}(l)a_l(\invs{F})a_l(\invs{G}) X^2 - l^{1+2\kappa_1 + 2\kappa_2}\varepsilon_{\invs{F}}(l^{2})\varepsilon_{\invs{G}}(l^{2})X^3 
\end{aligned}
\end{equation*}
where, as before, $\sigma_l$ is the image of the arithmetic Frobenius at $l$ in $\Gal(\mbb{Q}(\mu_m)/\mbb{Q})$.
\end{enumerate}
\end{proposition}
\begin{proof}
Consider the composition 
\[
\begin{aligned}
\opn{H}^3_{\et}(Y(m, mNp)^2, & \Lambda^{[j, j]}(2-j) ) \\
& \xrightarrow{\tau_{m, *}} \opn{H}^3_{\et}(Y_1(N_1(p)) \times Y_1(N_2(p)) \times \mu_m^{\circ}, \Lambda^{[j, j]}(\SH_0') ) \\
& \longrightarrow \opn{H}^3_{\et}(Y_1(N_1(p)) \times Y_1(N_2(p)) \times \mu_m^{\circ}, \invs{D}_{[\invs{U}_1, \invs{U}_2]}(\SH_0')).
\end{aligned}
\]
By applying Lemma \ref{CoresCompatibility}, the morphisms $(\Xi_l)_*$ and $\opn{cores}^{\mbb{Q}(\mu_{lm})}_{\mbb{Q}(\mu_m)}$ are compatible under the map 
\begin{equation} \label{AJCompMap}
{\opn{H}^3_{\et}(Y(m, mN)^2, \Lambda^{[j, j]}(2-j))} \to {\opn{H}^1(\mbb{Q}(\mu_m), [M(\invs{F})^* \hatot M(\invs{G})^*](-j))}
\end{equation}
obtained by composing the above map with $\opn{AJ}_{\invs{F}, \invs{G}}$.

Immediately we see that if $l$ divides $m$ then
\[
\begin{aligned}
{\nabla_1 \choose j} {\nabla_2 \choose j} \opn{cores}^{\mbb{Q}(\mu_{ml})}_{\mbb{Q}(\mu_m)} {\beilfj{lm}{\invs{G}}} & = \opn{cores}^{\mbb{Q}(\mu_{ml})}_{\mbb{Q}(\mu_m)} {\nabla_1 \choose j} {\nabla_2 \choose j} {\beilfj{lm}{\invs{G}}}  \\
& = (a_l(\invs{F})a_l(\invs{G})){\nabla_1 \choose j} {\nabla_2 \choose j}  {\beilfj{m}{\invs{G}}}
\end{aligned}
\]
where the second equality follows from part 1 of Proposition \ref{RankinIwasawaNorm} and the fact that $T_l'$ acts as multiplication by $a_l(\invs{F})$ (resp. $a_l(\invs{G})$) on $M(\invs{F})^*$ (resp. $M(\invs{G})^*$). Note that under the morphism $t_m$ the operators $U_l'$ and $T_l'$ are compatible. Since $j$ is not contained in $V_1$ or $V_2$, the operator ${\nabla_1 \choose j} {\nabla_2 \choose j}$ is invertible and we have the required relation.

For the second part, recall that by part 2 in Proposition \ref{RankinIwasawaNorm}, the Rankin--Iwasawa classes satisfy $(\Xi_l \times \Xi_l)_*(\rankiw{lm}{lmN}) = \widetilde{Q}_l \cdot  {\rankiw{m}{mN}}$. We have the following commutative diagram:
\[
\begin{tikzcd}
{\opn{H}^3_{\et}(Y(m, mpN)^2, \Lambda^{[j, j]}(2-j) )} \arrow[r] \arrow[d] \arrow[d, "\widetilde{Q}_l"'] & {\opn{H}^1(\mbb{Q}(\mu_m), [M(\invs{F})^* \hatot M(\invs{G})^*](-j) )} \arrow[d] \arrow[d, "\invs{Q}_l(l^{-j}\sigma_l^{-1})"] \\
{\opn{H}^3_{\et}(Y(m, mpN)^2, \Lambda^{[j, j]}(2-j) )} \arrow[r] & {\opn{H}^1(\mbb{Q}(\mu_m), [M(\invs{F})^* \hatot M(\invs{G})^*](-j))}
\end{tikzcd}
\] 
where the horizontal arrows are the maps in (\ref{AJCompMap}). Indeed, $M(\invs{F})^*$ can be described as the quotient of
\[
\opn{H}^3_{\et}(Y_1(N_1(p))_{\bar{\mbb{Q}}}, \invs{D}_{V_1}(1))
\]
such that $T'_l$ acts as multiplication by $a_l(\invs{F})$ and $\langle l \rangle$ acts as multiplication by $\varepsilon_{\invs{F}}(l)^{-1}$. We have a similar description for $M(\invs{G})^*$. Furthermore, the action of $\sigma_l$ becomes the natural action of $\sigma_l$ under the horizontal map in the above diagram (this is an application of the push-pull lemma for \'{e}tale cohomology). Finally, by Lemma \ref{OverconvergentDiagram}, $[l]_*$ becomes multiplication by $l^{\kappa_i - j}$. This shows that the above diagram is commutative and completes the proof of the proposition. 
\end{proof}

\subsection{Interpolation in the cyclotomic variable} \label{Interpolation}

In this section we recall how to interpolate the Beilinson--Flach classes $\beilfj{m}{\invs{G}}$ in the cyclotomic variable $j$. As a consequence, we show that the three-variable Beilinson--Flach classes satisfy a norm-compatibility relation closely related to the Euler system relations.  

Let $p^{-\lambda_1} = ||a_p(\invs{F})||$ (resp. $p^{-\lambda_2} = ||a_p(\invs{G})||$) where $||\cdot ||$ denotes the canonical supremum norm on $\ordd(V_i)$ (which exists because we have restricted our Coleman families to reduced affinoid subdomains). Let $\lambda = \lambda_1 + \lambda_2$ and $h \geq \lambda$ a positive integer. Define the following elements 
\[
x_{n, j} \defeq (a_p(\invs{F})a_p(\invs{G}))^{-n} \frac{\beilfj{mp^n}{\invs{G}}}{(-1)^jj!}
\]
for $0 \leq j \leq h$, $n \geq 1$, and set
\[
x_{0, j} \defeq \left(1 - \frac{p^j}{a_p(\invs{F})a_p(\invs{G})} \right) \frac{\beilfj{m}{\invs{G}}}{(-1)^jj!}
\]
for $0 \leq j \leq h$. These elements are compatible under corestriction and satisfy a certain growth bound (see \cite[Proposition 5.4.1]{lz-coleman}), so by Proposition 2.3.3 in \emph{op.~cit.} there exists a \emph{unique} element 
\[
\beilfd{m}{\invs{G}} \in \opn{H}^1(\mbb{Q}(\mu_{mp^{\infty}}), D_{\lambda}(\Gamma, M))^{\Gamma} \cong \opn{H}^1(\mbb{Q}(\mu_m), D_{\lambda}(\Gamma, M))
\]
satisfying 
\[
\int_{\Gamma_n} \chi_{\mathrm{cycl}}^j \; {\beilfd{m}{\invs{G}}} = x_{n, j}
\]
for all $n, j$. Here $M = M_{V_1}(\invs{F})^* \hatot M_{V_2}(\invs{G})^*$ and $\Gamma_n \subset \Gamma$ is the unique subgroup of index $p^{n-1}(p-1)$ (we set $\Gamma_0 = \Gamma$). The class $\beilfd{m}{\invs{G}}$ is independent of the choice of $h$.

\begin{proposition} \label{normthm}
The Beilinson--Flach classes $\beilfd{m}{\invs{G}}$ satisfy the following norm compatibility relations:
\begin{itemize}
\item If $l$ divides $m$ then
\[
\opn{cores}^{\mbb{Q}(\mu_{lm})}_{\mbb{Q}(\mu_m)} \left( \beilfd{lm}{\invs{G}} \right) = (a_l(\invs{F})a_l(\invs{G})) \cdot {\beilfd{m}{\invs{G}}}.
\]
\item If $l$ doesn't divide $m$ then
\[
\opn{cores}^{\mbb{Q}(\mu_{lm})}_{\mbb{Q}(\mu_m)} \left( \beilfd{lm}{\invs{G}} \right) = \invs{Q}_l(l^{-\mathbf{j}}\sigma_l^{-1}) \cdot {\beilfd{m}{\invs{G}}}
\]
where $\invs{Q}_l(X)$ is the polynomial defined in Proposition \ref{normrelationfortwo} and $\mathbf{j}$ is the universal character $\Gamma \to D_{\lambda}(\Gamma, E)$ (i.e. the homomorphism taking $x \mapsto \opn{ev}_x$ where $\opn{ev}_x$ is the evaluation-at-$x$ map).
\end{itemize}
\end{proposition}
\begin{proof}
Let $\invs{Q}_l(X)$ denote the polynomial appearing in Proposition \ref{normrelationfortwo} and let $l$ be a prime not dividing $6mpNc$. Set 
\[
\nu_m = \invs{Q}_l(l^{-\mathbf{j}}\sigma_l^{-1}) \cdot {\beilfd{m}{\invs{G}}} - \opn{cores}^{\mbb{Q}(\mu_{lm})}_{\mbb{Q}(\mu_m)} \left( \beilfd{lm}{\invs{G}} \right).
\]  
Then for all $n,j \geq 0$ the specialisation $\int_{\Gamma_n}{\chi^j \nu_m}$ is zero; so $\nu_m$ interpolates only zero classes. By uniqueness, this implies that $\nu_m = 0$. A similar argument works for $l | m$.
\end{proof}

\subsection{Euler system relations in families} \label{ESinFamilies}

In Proposition \ref{normthm}, we showed that the Beilinson--Flach classes satisfy norm compatible relations. It turns out that we can adjust these classes so that we obtain cohomology classes satisfying the Euler system relations. 

As before, let $M = M_{V_1}(\invs{F})^* \hatot M_{V_2}(\invs{G})^*$ and let $\invs{P}_l(X)$ denote the polynomial $\opn{det}\left(1 - \opn{Frob}_l^{-1} X | M^*(1) \right)$, where $\opn{Frob}_l$ denotes any lift of the arithmetic Frobenius at $l$. Then one can observe that
\[
\invs{Q}_l(X) = X^{-1}\left( (l-1)(1-l^{\kappa_1 + \kappa_2 + 2}\varepsilon_{\invs{F}}(l)\varepsilon_{\invs{G}}(l) X^2 ) - l \invs{P}_l(X) \right) 
\]
so in particular $\invs{Q}_l(X) \equiv -X^{-1}\invs{P}_l(X)$ modulo $l-1$. Such a congruence allows us to adjust the classes $\beilfd{m}{\invs{G}}$ so that we obtain Euler system relations.

\begin{proposition} \label{ExistenceOfESClasses}
Let $c \geq 1$ be an integer prime to $6pN$ and let $A$ denote the set of all square-free positive integers which are coprime to $6pNc$. Then for all $m \in A$, there exist cohomology classes ${_c\invs{Z}_m^{[\invs{F}, \invs{G}]}} \in \opn{H}^1(\mbb{Q}(\mu_m), D^{\mathrm{la}}(\Gamma, M) )$ such that 
\begin{enumerate}
\item The bottom class satisfies $_c\invs{Z}_1^{[\invs{F}, \invs{G}]} = {\beilfd{1}{\invs{G}}}$.
\item If $l$ is a prime such that $lm \in A$ (so in particular $l \nmid m$), we have the following Euler system relation
\[
\opn{cores}^{\mbb{Q}(\mu_{lm})}_{\mbb{Q}(\mu_m)}{_c\invs{Z}_{lm}^{[\invs{F}, \invs{G}]}} = \invs{P}_l(l^{-\mathbf{j}} \sigma_l^{-1} ) \cdot {_c\invs{Z}_{m}^{[\invs{F}, \invs{G}]}}
\]
Note that $\invs{P}_l(l^{-\mathbf{j}}X) = \opn{det}\left( 1 - \opn{Frob}_l^{-1} X | M^*(1 + \mathbf{j}) \right)$.
\end{enumerate}
Moreover, ${_c\invs{Z}_m^{[\invs{F}, \invs{G}]}}$ differs from ${\beilfd{m}{\invs{G}}}$ by an element of $\ordd(V_1 \times V_2 \times \invs{W})^{\circ}[(\mbb{Z}/m\mbb{Z})^{\times}]$. 
\end{proposition}
\begin{proof}
This follows from the same argument in \cite[\S 7.3]{BeilinsonFlach}.
\end{proof}

Unfortunately, in general, there is no way to force these classes to lie in a Galois stable lattice inside $D^{\mathrm{la}}(\Gamma, M)$, so we do not get an Euler system for this representation. However, this is possible after specialisation so long as we use a weaker notion of an Euler system.

\begin{corollary}
Let $\mathbf{x} = (k_1, k_2, \eta) \in V_1 \times V_2 \times \invs{W}$ defined over a finite extension $E/\mbb{Q}_p$, and let $T(\eta^{-1})$ be a Galois stable lattice inside $M_E(\invs{F}_{k_1})^* \otimes M_E(\invs{G}_{k_2})^* (\eta^{-1})$. Assume that $k_1 \neq k_2$. Let $c \geq 1$ be an integer prime to $6Np$ and let $\invs{N}$ be a finite product of primes containing all primes dividing $6cNp$. Let $S$ denote the set of positive integers divisible only by primes not dividing $\invs{N}$. Then for $m \in S$ and $V_1$ and $V_2$ small enough, there exist cohomology classes 
\[
c_m \in \opn{H}^1\left(\mbb{Q}(\mu_m), T(\eta^{-1}) \right) 
\]
which satisfy
\[
\opn{cores}^{\mbb{Q}(\mu_{lm})}_{\mbb{Q}(\mu_m)} {c_{lm}} = \left\{ \begin{array}{cc} c_m & \text{ if } l | m \\ P_l(\eta^{-1}(l)\sigma_l^{-1}) \cdot c_m & \text{ if } l \nmid m \end{array} \right.
\]
where $l$ is a prime not dividing $\invs{N}$ and $P_l(\eta^{-1}(l)X)$ is the specialisation of $\invs{P}_l(l^{-\mathbf{j}}X)$ at $(k_1, k_2, \eta)$. Furthermore, the bottom class $c_1$ is a non-zero multiple of ${\beilfd{1}{\invs{G}}}$. 
\end{corollary}
\begin{proof}
Firstly, note that $\opn{H}^0(\mbb{Q}(\mu_{mp^{\infty}}), M_{\mathbf{x}}) = 0$ for all $m \in S$, where 
\[
M_{\mathbf{x}} = M_E(\invs{F}_{k_1})^* \otimes M_E(\invs{G}_{k_2})^* (\eta^{-1}).
\]
Indeed this is true because we have assumed $k_1 \neq k_2$, for the following reason. Shrinking $V_1$ and $V_2$ if necessary, we can assume that $M_{\mathbf{x}}$ is (absolutely) irreducible. Hence any twist of $M_{\mathbf{x}}$ by a character is also irreducible. But if $M_{\mathbf{x}}$ has any non-trivial invariants under the group $G_{\mbb{Q}^{ab}}$ then there is a one-dimensional submodule of $M_{\mathbf{x}}$ on which $G_{\mbb{Q}}$ acts via a character. This is a contradiction to irreducibility. 

Therefore, by applying Proposition 2.4.7 in \cite{lz-coleman}, there exists a constant $R > 0$ independent of $m$ such that $R \cdot {\beilfd{m}{\invs{G}}}$ specialised at $\mathbf{x}$ lands in the cohomology of the Galois stable lattice $T(\eta^{-1})$.   

But since ${_c\invs{Z}_m^{[\invs{F}, \invs{G}]}}$ differs from ${\beilfd{m}{\invs{G}}}$ by an element of $\ordd(V_1 \times V_2 \times \invs{W})^{\circ}[(\mbb{Z}/m\mbb{Z})^{\times}]$, this implies that the specialisation of $R \cdot {_c\invs{Z}_m^{[\invs{F}, \invs{G}]}}$ lands in the cohomology of the Galois stable lattice $T(\eta^{-1})$. We set $c_m$ to be the specialisation of $R \cdot {_c\invs{Z}_m^{[\invs{F}, \invs{G}]}}$ at $\mathbf{x}$.

By Proposition \ref{ExistenceOfESClasses}, we obtain the Euler system relations for the classes $c_m$, and this proves the Corollary for all $m \in A$, where $A$ is the subset of $S$ consisting of all square-free integers.

But we can extend the classes $\{c_m : m \in A\}$ to a collection of classes indexed over the set $S$ by defining $c_m$ to be 
\[
c_m := \prod_{l | m}\frac{1}{l^{v_l(m) - 1}} \opn{res}^{\mbb{Q}(\mu_m)}_{\mbb{Q}(\mu_{m'})} c_{m'}
\]
where $m'$ is the radical of $m$ (i.e. the product of all prime factors that divide $m$) and $v_l(m)$ is the $l$-adic valuation of $m$. We don't lose integrality of the classes because all integers in $S$ are coprime to $p$.
\end{proof}


\section{Preliminaries on $(\varphi, \Gamma)$-modules}

\subsection{Period rings}

In this section we (briefly) recall the period rings that will be used throughout the paper. Since we only work with representations of $G_{\mbb{Q}_p}$, we specialise immediately to this case and refer the reader to \cite{berger3} for the definitions over more general $p$-adic fields; proofs and their corresponding references for all of the assertions in this section can also be found in \emph{op. cit.}.

Let $\mbb{C}_p$ denote a fixed completed algebraic closure of $\mbb{Q}_p$ and let $v_p$ denote the unique valuation on $\mbb{C}_p$ such that $v_p(p) = 1$. Set $\ordd_{\mbb{C}_p}^{\flat} \defeq \varprojlim \ordd_{\mbb{C}_p}/p$, where the inverse limit is over the $p$-th power map, and fix $\varepsilon \defeq (\varepsilon_n) \in \ordd_{\mbb{C}_p}^{\flat}$ a compatible system of $p$-th roots of unity, i.e. $\varepsilon_n$ is a $p^n$-th root of unity such that $\varepsilon_{n+1}^p = \varepsilon_n$.

Let $\widetilde{\mathbf{A}}^+$ denote the ($p$-typical) Witt vectors of $\ordd_{\mbb{C}_p}^{\flat}$ and let $\widetilde{\mathbf{B}}^+ = \widetilde{\mathbf{A}}^+[\frac{1}{p}]$. One has a canonical homomorphism
\[
\theta\colon \widetilde{\mathbf{B}}^+ \rightarrow \mbb{C}_p
\] 
whose kernel is a principal ideal of $\widetilde{\mathbf{B}}^+$ generated by $\omega \defeq ([\varepsilon] - 1)/([\varepsilon'] - 1)$ where $\varepsilon' = (\varepsilon_{n+1})_{n \geq 0}$ and $[\cdot]$ denotes the Teichm\"{u}ller lift. We set $t \defeq \opn{log}([\varepsilon])$ to be the period of Fontaine. 

We define the de Rham period rings $\bdr^+$ and $\bdr$ to be the $(\opn{ker}\theta)$-adic completion of $\widetilde{\mathbf{B}}^+$ and the fraction field of $\bdr^+$ respectively. Similarly we define the crystalline period rings $\bcris^+$ and $\bcris$ to be the divided-power envelope of $\widetilde{\mathbf{B}}^+$ (with respect to $\omega$) and $\bcris = \bcris^+[t^{-1}]$ respectively. The field $\bdr$ comes equipped with a decreasing filtration given by $\opn{Fil}^i\bdr = t^i\bdr^+$ and since $\bcris \subset \bdr$, the ring $\bcris$ also inherits a filtration from $\bdr$. Furthermore, $\bcris$ comes equipped with a Frobenius endomorphism $\varphi$ extending the natural Frobenius on $\widetilde{\mathbf{B}}^+$. 

In addition to these constructions, we define overconvergent period rings as follows. For $r, s$ two positive rational numbers, let $\widetilde{\mathbf{A}}^{[r, s]}$ denote the $p$-adic completion of $\widetilde{\mathbf{A}}^+\left[\frac{p}{[\varepsilon - 1]^r}, \frac{[\varepsilon - 1]^s}{p}\right]$ and set $\widetilde{\mathbf{B}}^{[r, s]} = \widetilde{\mathbf{A}}^{[r, s]}[1/p]$. We define $\widetilde{\mathbf{B}}^{\dagger, r}_{\text{rig}}$ to be the intersection $\widetilde{\mathbf{B}}^{\dagger, r}_{\text{rig}} \defeq \cap_{r \leq s < +\infty}\widetilde{\mathbf{B}}^{[r, s]}$ and if we set $\pi = [\varepsilon] -1$, then this ring contains the following subring 
\[
\mathbf{B}^{\dagger, r}_{\text{rig}, \mbb{Q}_p} \defeq \left\{ f(\pi) = \sum_{k = -\infty}^{\infty}{a_k \pi^k} : \begin{array}{c} a_k \in \mbb{Q}_p \\ f(x) \text{ converges for } 0 < v_p(x) \leq \frac{1}{r} \end{array}  \right\}.
\] 
Note that the period $t$ defined previously is equal to $\opn{log}(1+\pi)$. 

The union of these rings, namely $\bdrig \defeq \varinjlim{\mathbf{B}^{\dagger, r}_{\mathrm{rig}, \mbb{Q}_p}}$, is the \emph{Robba ring} associated to the $p$-adic field $\mbb{Q}_p$ and can be identified with all power series in $\mbb{Q}_p\db{\pi, \pi^{-1}}$ that converge on an annulus of the form $\{x \in \mbb{C}_p : 0 < v_p(x) \leq 1/r \}$ for some positive rational number $r$. We let $\bprig \subset \bdrig$ denote the subring of power series which converge on the whole open unit disc, i.e. $f(x)$ converges for all $v_p(x) > 0$. Both of these rings come equipped with an action of Frobenius given by the formula
\[
(\varphi \cdot f) (\pi) = f((1+\pi)^p - 1)
\] 
and $\Gamma \defeq \Gal(\mbb{Q}_p(\mu_{p^{\infty}})/\mbb{Q}_p)$ given by the formula 
\[
(\gamma \cdot f) (\pi) = f((1+\pi)^{\chi_{\mathrm{cycl}}(\gamma)} - 1)
\]
where $\gamma \in \Gamma$ and $\chi_{\mathrm{cycl}}\colon \Gamma \to \mbb{Z}_p^{\times}$ is the cyclotomic character. The morphism $\varphi$ has a left inverse, denoted by $\psi$, which satisfies the following relation
\[
(\varphi \psi \cdot f)(\pi) = \frac{1}{p} \sum_{\xi^p = 1}{f((1+\pi)\xi - 1)}
\]
where the sum is over all $p$-th roots of unity.

If $A$ is an affinoid $\mbb{Q}_p$-algebra, we define the Robba ring over $A$ to be the completed tensor product $\bdrigA{A}$, and similarly we denote the subring of bounded power series by $\bprig \hatot A$. As above, both of these rings come equipped with an action of $\varphi$, $\psi$ and $\Gamma$ by the exact same formulae. 

\subsection{Overconvergent $(\varphi, \Gamma)$-modules over affinoid algebras} \label{OverconvergentPhiGammaModules}

\begin{definition} 
Let $A$ be an affinoid algebra over $\mbb{Q}_p$. We say that $D$ is a $(\varphi, \Gamma)$-module over $\bdrigA{A}$ if $D$ is a finite projective $(\bdrigA{A})$-module with commuting semilinear actions of $\varphi$ and $\Gamma$, such that $\varphi(D)$ generates $D$ as a $(\bdrigA{A})$-module. 
\end{definition}

If $M$ is a Galois representation over an affinoid algebra $A$ (i.e. a finite projective $A$-module with a continuous $A$-linear action of $G_{\mbb{Q}}$) then Berger and Colmez \cite{Berger-Colmez} (and more generally Kedlaya and Liu \cite{familiesphiGamma}) have constructed a functor 
\[
M \mapsto \ddrig{M}
\]
which associates to a Galois representation a $(\varphi, \Gamma)$-module over the Robba ring $\bdrigA{A}$. This agrees with the usual functor as constructed by Berger \cite{berger2} when $A$ is a finite field extension of $\mbb{Q}_p$. 

For the rest of this section let $A=E$ be a finite field extension of $\mbb{Q}_p$ and let $D$ be a $(\varphi, \Gamma)$-module over $\bdrigE$. By taking the ``stalk at $\zeta_{p^n} - 1$'' one can define a $\mbb{Q}_p(\mu_{p^{\infty}})(\!(t)\!)$-module $D_{\mathrm{dif}}$ with a semilinear action of $\Gamma$ (see \cite{Nakamura2014}). We set 
\[
\derham{D} \defeq D_{\mathrm{dif}}^{\Gamma} \; \; \; \; \; \;  \text{ and } \; \; \; \; \; \;  \Cris{D} \defeq D[t^{-1}]^{\Gamma}
\]
and note that both $\derham{D}$ and $\Cris{D}$ are finite-dimensional vector spaces over $E$.

\begin{definition}
Let $E/\mbb{Q}_p$ be a finite extension and $D$ a $(\varphi, \Gamma)$-module over $\bdrigE$. 
\begin{enumerate}
\item We say that $D$ is de Rham (resp. crystalline) if the $E$-dimension of $\derham{D}$ (resp. $\Cris{D}$) is equal to the rank of $D$ as a $\bdrigE$-module. 
\item If $D$ is de Rham then $\derham{D}$ comes equipped with a decreasing filtration induced from the $t$-adic filtration on $\mbb{Q}_p(\mu_{p^{\infty}})\db{t}$. The Hodge--Tate weights of $D$ are defined to be the negatives of the jumps in the filtration on $\derham{D}$ (so in particular the cyclotomic character has Hodge--Tate weight $1$). 
\end{enumerate}
\end{definition}

If $D$ is a crystalline $(\varphi, \Gamma)$-module then one can associate a sub-$(\bprig \hatot E)$-module of $D$, denoted $\nrig{D}$, that is free of rank equal to the rank of $D$ and is stable under $\Gamma$. Furthermore $\varphi$ restricts to a morphism 
\[
\varphi \colon \nrig{D} \to \nrig{D}[q^{-1}]
\]
where $q = \varphi(\pi)/\pi$; and if $D$ has non-negative Hodge--Tate weights, then $\nrig{D}$ is in fact stable under $\varphi$. This submodule is called the Wach module associated to $D$ and will be important in section \ref{prlog} when we recall the construction of the Perrin-Riou logarithm. For more details on the construction and properties of this module, see \cite[\S 3]{pottc}. Note that if $D$ comes from a crystalline $p$-adic representation $V$ and $\mathbf{N}(V)$ denotes the usual Wach module (over $\mathbf{B}_{\mbb{Q}_p}^+$) associated with $V$, then $\nrig{D}$ satisfies the relation
\[
\nrig{D} = \mathbf{N}(V) \otimes_{\mathbf{B}^+_{\mbb{Q}_p}} \bprig .
\]
Here $\mathbf{B}^+_{\mbb{Q}_p}$ is a period ring which can be identified with the subring $\mbb{Z}_p \db{\pi}[1/p] \subset \bprig$.

\subsection{Cohomology of $(\varphi, \Gamma)$-modules} \label{CohomologyPHIGAMMA}
If $D$ is a \pgp-module over $\bdrigA{A}$, then we define
the Herr complex
\[ C^{\bullet}_{\varphi, \gamma}(D) \defeq D \xrightarrow{\varphi-1,\gamma-1} D \oplus D \xrightarrow{1-\gamma,\varphi-1} D \, \]
concentrated in degrees $0, 1, 2$, where $\gamma$ is a choice of topological generator for $\Gamma$ (such an element exists because we have assumed $p > 2$). We define $\rcont{\mbb{Q}_p}{D}$ to be the corresponding object in the derived category of bounded complexes of (continuous) $A$-modules, and we denote the cohomology of this complex by $\opn{H}^i(\mbb{Q}_p, D) \defeq \opn{H}^i(C^{\bullet}_{\varphi, \gamma}(D))$.

By \cite[Thm.~2.8]{pottharst-selmer}, for any $A$-representation $M$ of $G_{\Qp}$ there is a canonical
quasi-isomorphism
\[ \rcont{\mbb{Q}_p}{M} \cong \rcont{\mbb{Q}_p}{\ddrig{M}} \,. \]

Since $\varphi(D)$ generates $D$ as an $\bdrigA{A}$-module, the map
$\varphi \colon D \to D$ has a unique semilinear left inverse $\psi$.
As the following lemma shows, the cohomology of $D$ can be computed
with $\psi$ in place of $\varphi$.
\begin{lemma}[{\cite[Proposition 2.3.6]{kpx}}] \label{CohomologyPsi}
Let $C^{\bullet}_{\psi,\gamma}(D)$ be the bottom row of the commutative
diagram
\[
\begin{tikzcd}
D \arrow[r,"{\varphi-1,\gamma-1}"] \arrow[d,"\operatorname{id}"] & 
D \oplus D \arrow[r,"{1-\gamma,\varphi-1}"]  \arrow[d,"-\psi \oplus \operatorname{id}"] &
D \arrow[d,"-\psi"] \\
D \arrow[r,"{\psi-1,\gamma-1}"] & D \oplus D \arrow[r,"{1-\gamma,\psi-1}"] & D
\end{tikzcd}
\]
(here the top row is $C^{\bullet}_{\varphi, \gamma}(D)$).
Then the map $C^{\bullet}_{\varphi, \gamma}(D) \to C^{\bullet}_{\psi,\gamma}(D)$
defined above is a quasi-isomorphism.
\end{lemma}

We finish this section by stating an Euler--Poincar\'e characteristic formula for $(\varphi, \Gamma)$-modules, which will be used in the proof of Theorem \ref{TheoremVanishing}.
\begin{proposition}[{\cite[Theorem 5.3]{Liuduality}}]  \label{EulerCharFormula}
Let $A=E$ be a finite field extension of $\mbb{Q}_p$ and $D$ a $(\varphi, \Gamma)$-module over $\bdrigE$. Then, for $i=0, 1, 2$, $\opn{H}^i(\mbb{Q}_p, D)$ are finite-dimensional vector spaces over $E$ and 
\[
\chi(D) \defeq \sum_{i=0}^{2}(-1)^i\opn{dim}_E\opn{H}^i(\mbb{Q}_p, D) = -\opn{rank}D
\]
where $\opn{rank}D$ is the rank of $D$ as a $(\bdrigE)$-module. 
\end{proposition}

\subsection{Iwasawa cohomology}

Recall that for an affinoid algebra $A$, we denote its unit ball by $A^\circ$. Let $M$ be a Galois representation over $A$ and let $T$ be a Galois stable lattice inside $M$ (i.e. a sub-$A^\circ$-module that is stable under the action of $G_{\mbb{Q}}$ and satisfies $T[1/p] = M$). The classical Iwasawa cohomology of $M$ is defined to be 
\[
\opn{H}^i_{\mathrm{cl.Iw}}(\mbb{Q}_{\infty}, M) \defeq \left( \varprojlim_n \opn{H}^i(\mbb{Q}(\mu_{p^n}), T)\right) \left[\frac{1}{p}\right]
\]
where the inverse limit is over the corestriction maps, and the (analytic) Iwasawa cohomology of $M$ is
\[
\opn{H}^i_{\Iw}(\mbb{Q}_{\infty}, M) \defeq \opn{H}^i(\mbb{Q}, D^{\mathrm{la}}(\Gamma, M) )
\]
where $D^{\mathrm{la}}(\Gamma, M)$ denotes the space of locally analytic distributions on $\Gamma$, valued in $M$. For a finite place $v$ of $\mbb{Q}$, the Iwasawa cohomology groups $\opn{H}^i_{\mathrm{cl.Iw}}(\mbb{Q}_{v, \infty}, M)$ and $\opn{H}^i_{\Iw}(\mbb{Q}_{v, \infty}, M)$ are defined analogously.
These two constructions satisfy the relation
\[
\opn{H}^i_{\Iw}(\mbb{Q}_{\infty}, M) = \opn{H}^i_{\mathrm{cl.Iw}}(\mbb{Q}_{\infty}, M) \hatot_{\mbb{Z}_p\db{\Gamma}} D^{\mathrm{la}}(\Gamma, A).
\]
\begin{remark}
Our notation for Iwasawa cohomology differs from that in \cite{lz-coleman}; in \emph{op. cit.} $\opn{H}_{\Iw}$ denotes \emph{classical} Iwasawa cohomology, whereas in this paper it refers to \emph{analytic} Iwasawa cohomology.
\end{remark}
\begin{definition}
Let $D$ be a $(\varphi, \Gamma)$-module over $\bdrigA{A}$. The Iwasawa Herr complex is defined to be 
\[
C_{\Iw}^{\bullet}(D)\colon D \xrightarrow{\psi - 1} D
\]
concentrated in degrees $1$ and $2$, where $\psi$ is the left inverse to $\varphi$ as discussed in the previous section. We denote the cohomology of this complex by $\opn{H}^i_{\Iw}(\mbb{Q}_p, D)$. 
\end{definition}

We have the following relation between Iwasawa cohomology for $M$ and Iwasawa cohomology for $\ddrig{M}$.

\begin{proposition}
Let $M$ be a Galois representation over an affinoid algebra $A$. Then one has the following isomorphism 
\[
\opn{H}^i_{\Iw}(\mbb{Q}_{p, \infty}, M) \cong \opn{H}^i_{\Iw}(\mbb{Q}_p, \ddrig{M}).
\]
In particular $\opn{H}^i_{\Iw}(\mbb{Q}_{p, \infty}, M)$ vanishes for $i \neq 1, 2$.
\end{proposition}
\begin{proof}
See Corollary 4.4.11 in \cite{kpx}.
\end{proof}


\section{Some p-adic Hodge theory} \label{prlog}

In this section we recall the construction of the Perrin-Riou logarithm following \cite[\S 3]{pottc} and use this map to show that if the $p$-adic $L$-function doesn't vanish then we obtain two linearly independent classes in $\opn{H}^1(\mbb{Q}_p, D^-(\eta^{-1}))$, where $D^-(\eta^{-1})$ is a certain $2$-dimensional $(\varphi, \Gamma)$-module defined in (\ref{defofD}). Throughout this section, $E$ is a finite extension of $\mbb{Q}_p$ and $D$ will denote a $(\varphi, \Gamma)$-module over the Robba ring $\bdrigE$. 

Recall from section \ref{OverconvergentPhiGammaModules} that if $D$ is a crystalline $(\varphi, \Gamma)$-module then one can associate to $D$ the Wach module $\nrig{D}$. This is a sub $(\bprig \hatot E)$-module of $D$ that is free of rank equal to the rank of $D$ and is stable under $\Gamma$. It will be useful to impose the following hypothesis on $D$:  
\begin{itemize}
\item[(H)] The Hodge--Tate weights of $D$ are non-negative and $p^n$ is not an eigenvalue of $\varphi$ on $D_{\mathrm{cris}}$ for all integers $n \geq 0$.
\end{itemize}   
The reason for imposing this hypothesis is to ensure that the ``big logarithm'' as constructed in \cite[\S 3]{pottc} lands in the lattice $D_{\mathrm{cris}} \otimes \Lambda_{\infty}$. 

\begin{lemma} \label{WachProperties}
Let $D$ be a crystalline $(\varphi, \Gamma)$-module satisfying hypothesis (H). Then
\begin{enumerate}
\item The inclusion $\nrig{D} \rightarrow D$ induces an isomorphism $\nrig{D}^{\psi = 1} \cong D^{\psi = 1}$, where $\psi$ is the left inverse to $\varphi$ coming from the trace map (see \S \ref{OverconvergentPhiGammaModules}). 

\item Let $\varphi^*\nrig{D}$ denote the sub $(\bprig \hatot E)$-module of $D$ generated by $\varphi(\nrig{D})$. Then there is an inclusion 
\[
(\varphi^*\nrig{D})^{\psi = 0} \subset D_{\mathrm{cris}} \otimes_{E} (\bprig \hatot E)^{\psi = 0}
\]   
\item We have an inclusion 
\[
\nrig{D} \subset q^{h_1} \varphi^*\nrig{D} \subset \varphi^*\nrig{D} 
\]
where $h_1$ is the smallest Hodge--Tate weight of $D$ and $q = \varphi(\pi)/\pi$. 
\end{enumerate}
\end{lemma}
\begin{proof}
Let $h_1 \leq h_2 \leq \cdots \leq h_d$ be the Hodge--Tate weights of $D$. Since $D$ has non-negative Hodge--Tate weights and $p^n$ is not an eigenvalue for $\varphi$ for all $n \geq 0$, the quantity $a(D) \defeq \opn{max}\{-h_1, \lambda(D) +1 \}$ is non-positive, where $\lambda(D)$ is the largest integer (or $-\infty$ if there is no such one) such that $\varphi - p^{\lambda(D)}$ is not bijective on $D_{\mathrm{cris}}$. The first part then follows from \cite[Theorem 3.3]{pottc}. 

For the second part, we note that from the second bullet point in Theorem 3.1 in \emph{op. cit.}
\[
\nrig{D} \subset \left(\frac{t}{\pi}\right)^{h_1} (D_{\mathrm{cris}} \otimes (\bprig \hatot E) ) \subset (D_{\mathrm{cris}} \otimes (\bprig \hatot E))
\]
where the last inclusion follows because the Hodge--Tate weights are non-negative. Since $\varphi$ is an isomorphism on $D_{\mathrm{cris}}$, $\psi$ is also an isomorphism and we have $\varphi^*(D_{\mathrm{cris}} \otimes (\bprig \hatot E)) = (D_{\mathrm{cris}} \otimes (\bprig \hatot E))$. Combining these facts we obtain the inclusion in part $2$.

Similarly the third part follows from the third bullet point in Theorem 3.1 in \emph{loc.~cit.}, using the fact that $h_1 \geq 0$ and that $q \in \bprig$. 
\end{proof}

\subsection{Perrin-Riou's big logarithm}

Let $\Lambda_{\infty}$ denote the global sections of $\invs{W}_E$ - this can be identified with a subring of the ring of power series $E[\Delta]\db{\gamma - 1}$, where $\Delta$ is the torsion subgroup of $\Gamma$ and $\gamma$ is a topological generator of $\Gamma/\Delta$. The ring $(\bprig \hatot E)^{\psi = 0}$ has an action of $\Gamma$ which extends to an action of $\Lambda_{\infty}$ via the Mellin transform:
\begin{eqnarray}
\ide{M}\colon & \Lambda_{\infty} & \xrightarrow{\sim}  (\bprig \hatot E)^{\psi = 0} \nonumber \\
 & f(\gamma - 1) & \mapsto  f(\gamma - 1) \cdot (1 + \pi) \nonumber
\end{eqnarray}
Therefore $(\bprig \hatot E)^{\psi = 0}$ can be viewed as a free $\Lambda_{\infty}$-module of rank one.

\begin{definition}
Suppose that $D$ is a crystalline $(\varphi, \Gamma)$-module satisfying hypothesis (H). The Perrin-Riou logarithm $\invs{L} = \invs{L}_D : \opn{H}^1_{\Iw}(\mbb{Q}_{p}, D) \to D_{\mathrm{cris}} \hatot \Lambda_{\infty}$ is defined as the composition of the following maps:
\begin{eqnarray}
\opn{H}^1_{\Iw}(\mbb{Q}_{p}, D)  \cong  D^{\psi = 1} & = & \nrig{D}^{\psi = 1} \nonumber \\
 & \xrightarrow{1 - \varphi} & (\varphi^*\nrig{D})^{\psi = 0} \nonumber \\
 & \rightarrow & D_{\mathrm{cris}} \hatot_{E} (\bprig \hatot E)^{\psi = 0} \xrightarrow{\sim}  D_{\mathrm{cris}} \hatot_{E} \Lambda_{\infty} \nonumber    
\end{eqnarray}
where the second, third and fourth maps exist by Lemma \ref{WachProperties}. Here the last map is given by the inverse of the Mellin transform described above. Since all of the maps above are functorial in $D$, we see that $\invs{L}$ is a map of $\Lambda_{\infty}$-modules that is also functorial in $D$.
\end{definition} 

In the case that $D$ comes from a crystalline $p$-adic representation, the map $\invs{L}$ is a special case of a more general construction by Perrin-Riou (\cite{PerrinRiou}) which was later interpreted using $(\varphi, \Gamma)$-modules by Lei, Loeffler and Zerbes (\cite{LLZ-coleman-map}). This was generalised to potentially crystalline $(\varphi, \Gamma)$-modules by Pottharst (\cite{pottc}) and to de Rham $(\varphi, \Gamma)$-modules by Nakamura (\cite{Nakamura2014}). The key property of this map is that it interpolates the Bloch--Kato logarithm and dual exponential maps at certain classical specialisations. More precisely, let $\eta \colon \Gamma \to (E')^{\times}$ be a continuous character, where $\Gamma = \Gal(\mbb{Q}_p(\mu_{p^{\infty}})/\mbb{Q}_p)$, and consider the induced map $\Lambda_{\infty} \to E'$ which we will also denote by $\eta$. Then we have two specialisation maps: the first on Iwasawa cohomology 
\[
\opn{H}^1_{\Iw}(\mbb{Q}_{p}, D) \xrightarrow{\mathrm{pr}_{\eta}} \opn{H}^1(\mbb{Q}_p, D_{E'}(\eta^{-1}) )
\]
induced from the map on the Herr complexes $C^{\bullet}_{Iw}(D) \to C^{\bullet}_{\psi, \gamma}(D(\eta^{-1}))$ which in degree one is given by $x \mapsto (0, x)$ (recall that by Lemma \ref{CohomologyPsi} the cohomology of a $(\varphi, \Gamma)$-module can be calculated with $\psi$ in place of $\varphi$). The second specialisation map is 
\[
D_{\mathrm{cris}} \hatot \Lambda_{\infty} \xrightarrow{\mathrm{ev}_{\eta}} D_{\mathrm{cris}} \otimes_{\mbb{Q}_p} E' 
\]
induced from $\eta \colon \Lambda_{\infty} \to E'$. Then for classical $\eta$ (i.e. $\eta = \chi_{\mathrm{cycl}}^j$ for some integer $j$) there is a commutative diagram
\[
\begin{tikzcd}
{\opn{H}^1_{\Iw}(\mbb{Q}_{p}, D)} \arrow[d, "\mathrm{pr}_{\eta}"'] \arrow[r, "\invs{L}"] & D_{\mathrm{cris}} \hatot \Lambda_{\infty} \arrow[d, "\mathrm{ev}_{\eta}"] \\
{\opn{H}^1(\mbb{Q}_p, D_{E'}(\eta^{-1}))} \arrow[r, dashed] & D_{\mathrm{cris}} \otimes E'
\end{tikzcd}
\] 
where the dotted arrow is (up to some Euler factors) the Bloch--Kato logarithm in the range $j < h_1$, and the dual exponential map in the range $j \geq h_1$, where $h_1$ is the smallest Hodge--Tate weight of $D$ (see \cite{Nakamura2014}). We will study this map at non-classical specialisations.

\begin{proposition} \label{propertiesofL}
Let $D$ be a crystalline $(\varphi, \Gamma)$-module that satisfies hypothesis (H). The map $\invs{L} = \invs{L}_D$ satisfies the following properties:
\begin{enumerate}
\item Let $k \geq 0$ be an integer and let $\omega_{-k}$ denote the automorphism of $\Lambda_{\infty}$ which sends $\gamma \in \Gamma$ to the element $\chi_{\mathrm{cycl}}(\gamma)^{-k}\gamma$. Then we have a commutative diagram 
\[
\begin{tikzcd}
{\omega_{-k}^*\opn{H}^1_{\Iw}(\mbb{Q}_{p}, D)} \arrow[r, "\sim"] \arrow[d, "\omega_{-k}^*\invs{L}_D"'] & {\opn{H}^1_{\Iw}(\mbb{Q}_{p}, D(k))} \arrow[d, "\invs{L}_{D(k)}"] \\
\omega_{-k}^*\left( D_{\mathrm{cris}} \hatot \Lambda_{\infty} \right) \arrow[r, "\sim"] & D(k)_{\mathrm{cris}} \hatot \Lambda_{\infty}
\end{tikzcd}
\]
where for a $\Lambda_{\infty}$-module $M$, $\omega_{-k}^*M$ denotes the pull-back $M \otimes_{\Lambda_{\infty}, \omega_{-k}} \Lambda_{\infty}$.
\item For any character $\eta\colon \Gamma \to (E')^{\times}$ there exists a $\Lambda_{\infty}$-linear morphism making the following diagram commute:
\[
\begin{tikzcd}
{\opn{H}^1_{\Iw}(\mbb{Q}_{p}, D)} \arrow[d, "\mathrm{pr}_{\eta}"'] \arrow[r, "\invs{L}"] & D_{\mathrm{cris}} \hatot \Lambda_{\infty} \arrow[d, "\mathrm{ev}_{\eta}"] \\
{\opn{im}(\mathrm{pr}_{\eta})} \arrow[r, dashed] & D_{\mathrm{cris}} \otimes E'
\end{tikzcd}
\]
where $E'$ is a finite extension of $E$.
\end{enumerate}
\end{proposition}
\begin{proof}
The first part follows from carefully tracing through the definitions.  

For the second part it is enough to show that if $\mathrm{pr}_{\eta}(x) = 0$ then $\mathrm{ev}_{\eta}(\invs{L}(x)) = 0$, because then we can just define the map by taking a lift to $\opn{H}^1_{\Iw}(\mbb{Q}_{p}, D)$. If $\mathrm{pr}_{\eta}(x) = 0$ then there exists $y \in D^{\psi = 1}$ such that 
\[
x = (\gamma - \eta(\gamma) ) y.
\]
But $\invs{L}$ is $\Lambda_{\infty}$-linear so (after base-changing $D$ and $\Lambda_{\infty}$ to $D_{E'}$ and $(\Lambda_{\infty})_{E'}$ respectively) 
\[
\invs{L}(x) = (\gamma - \eta(\gamma) ) \invs{L}(y).
\]
But this is precisely mapped to zero under $\mathrm{ev}_{\eta}$. 
\end{proof}

From the above proposition we obtain the following corollary which will be useful in later sections.

\begin{corollary} \label{corlinearindependence}
Let $D$ be a $(\varphi, \Gamma)$-module satisfying hypothesis (H) and $D_1, D_2 \subset D$ two sub-$(\varphi, \Gamma)$-modules satisfying $D_{\mathrm{cris}} = (D_1)_{\mathrm{cris}} \oplus (D_2)_{\mathrm{cris}}$. Let $x_1$ and $x_2$ be two elements of $\opn{H}^1_{\Iw}(\mbb{Q}_p, D)$ that lie in the images of the maps $\opn{H}^1_{\Iw}(\mbb{Q}_p, D_1) \to \opn{H}^1_{\Iw}(\mbb{Q}_p, D)$ and $\opn{H}^1_{\Iw}(\mbb{Q}_p, D_2) \to \opn{H}^1_{\Iw}(\mbb{Q}_p, D)$ respectively. Then for any character $\eta\colon G \to E^{\times}$ the elements $\mathrm{pr}_{\eta}(x_1)$ and $\mathrm{pr}_{\eta}(x_2)$ are linearly independent if both $\mathrm{ev}_{\eta}\invs{L}(x_1)$ and $\mathrm{ev}_{\eta}\invs{L}(x_2)$ are non-zero.
\end{corollary}
\begin{proof}
This follows from functoriality of $\invs{L}$ and the existence of the bottom map in part (2) of Proposition \ref{propertiesofL}. Indeed, if both $\mathrm{ev}_{\eta}\invs{L}(x_1)$ and $\mathrm{ev}_{\eta}\invs{L}(x_2)$ are non-zero, then because they lie in different direct summands of $D_{\mathrm{cris}}$ they must be linearly independent in $D_{\mathrm{cris}} \hatot \Lambda_{\infty}$.  
\end{proof}

\subsection{Application to Beilinson--Flach classes} \label{ApplicationTo}

Returning to the situation in the paper, let $f$ and $g$ denote eigenforms satisfying the assumptions in section \ref{SummaryOfResults}, so we have Coleman families $\invs{F}, \invs{G}_{\alpha}, \invs{G}_{\beta}$ over $V_1, V_2^{\alpha}, V_2^{\beta}$ passing through the $p$-stabilisations $f_{\alpha}, g_{\alpha}, g_{\beta}$ respectively. Here $f_{\alpha}$ denotes the $p$-stabilisation of $f$ that satisfies $U_p f = \alpha f$; and similarly for $g_{\alpha}$ and $g_{\beta}$. Since we have assumed that the weights of $f$ and $g$ are not equal, we can choose $V_1$ such that it doesn't contain the character $k'$. In this subsection $\invs{G}$ and $V_2$ will denote either $\invs{G}_{\alpha}$ and $V_2^{\alpha}$ or $\invs{G}_{\beta}$ and $V_2^{\beta}$ respectively. 

Let $M_{V_1}(\invs{F})$ and $M_{V_2}(\invs{G})$ denote the Galois representations associated to $\invs{F}$ and $\invs{G}$ and let $\ddrig{\invs{F}}^*$ and $\ddrig{\invs{G}}^*$ denote the $(\varphi, \Gamma)$-modules of $M_{V_1}(\invs{F})^*$ and $M_{V_2}(\invs{G})^*$ respectively. If $V_1$ and $V_2$ are small enough, both of these modules come with a canonical triangulation which we will denote as
\[
0 \rightarrow \F{+}{\ddrig{?}^*} \rightarrow \ddrig{?}^* = \F{o}{\ddrig{?}^*} \rightarrow \F{-}{\ddrig{?}^*} \rightarrow 0
\]
for $? = \invs{F}$ or $?=\invs{G}$. In fact there is an explicit description for both the kernel and the cokernel (see \cite[Theorem 6.3.2]{lz-coleman}). 

Let $k_1$ be a (not necessarily classical) weight in $V_1$ and $\eta$ a character of $\Gamma = \Gal(\mbb{Q}_p(\mu_{p^{\infty}})/\mbb{Q}_p)$, and let $E$ be a finite extension of $\mbb{Q}_p$ that contains the fields of definition of $k_1$ and $\eta$. Recall that $k'+2$ denotes the weight of $g$. Let $M(\invs{F}_{k_1})$ and $M(\invs{G}_{k'})$ denote the specialisations of $M_{V_1}(\invs{F})$ and $M_{V_2}(\invs{G})$ at $k_1$ and $k'$ respectively, and note that we have isomorphisms 
\begin{eqnarray}
M((\invs{G}_{\alpha})_{k'}) & \xrightarrow{\mathrm{pr}^{\alpha}} & M_E(g) \nonumber \\
M((\invs{G}_{\beta})_{k'}) & \xrightarrow{\mathrm{pr}^{\beta}} & M_E(g) \nonumber  
\end{eqnarray}
both of which follow from the fact that $g$ is classical and the Galois representation doesn't change after $p$-stabilisation. We let $\ddrig{\invs{F}_{k_1}}^*$, $\ddrig{\invs{G}_{k'}}^*$ and $\ddrig{g}^*$ denote the $(\varphi, \Gamma)$-modules associated to $M(\invs{F}_{k_1})^*$, $M(\invs{G}_{k'})^*$ and $M_E(g)^*$ respectively. By specialising the triangulation above and applying $\mathrm{pr}^{\alpha}$ or $\mathrm{pr}^{\beta}$ if necessary, we obtain triangulations for these three $(\varphi, \Gamma)$-modules. 

Let $D^-$ be the $(\varphi, \Gamma)$-module
\begin{equation} \label{defofD}
D^- = \F{-}{\ddrig{\invs{F}_{k_1}}^*} \otimes \F{o}{\ddrig{g}^*}.
\end{equation}
Since $g$ is classical, $D^-$ is crystalline with Hodge--Tate weights ${0, 1+k'}$ and by the explicit description in \cite[Theorem 6.3.2]{lz-coleman}, $p^n$ is not an eigenvalue for $\varphi$ on $D^-_{\mathrm{cris}}$ for any integer $n \geq 0$ (so $D^-$ satisfies hypothesis (H)). Consider the following submodules 
\begin{eqnarray}
D_1 = D^{\alpha} & = & \mathrm{pr}^{\alpha} \left( \F{-}{\ddrig{\invs{F}_{k_1}}^*} \otimes \F{+}{\ddrig{(\invs{G}_{\alpha})_{k'}}^*} \right) \nonumber \\
D_2 = D^{\beta} & = & \mathrm{pr}^{\beta} \left( \F{-}{\ddrig{\invs{F}_{k_1}}^*} \otimes \F{+}{\ddrig{(\invs{G}_{\beta})_{k'}}^*} \right) \nonumber 
\end{eqnarray}
where $\mathrm{pr}^{\alpha}$ and $\mathrm{pr}^{\beta}$ are the isomorphisms described above. Again, by the explicit description in \emph{loc. cit.}, $\Cris{(D_1)}$ and $\Cris{(D_2)}$ are both rank one sub $\varphi$-modules of $\Cris{D}^-$ on which $\varphi$ acts by multiplication by $\alpha_{\invs{F}_{k_1}}^{-1}\beta_g^{-1}$ and $\alpha_{\invs{F}_{k_1}}^{-1}\alpha_g^{-1}$ respectively. Since we have assumed that $g$ is $p$-regular (i.e. $\alpha_g \neq \beta_g$) we must have $\Cris{D}^- = \Cris{(D_1)} \oplus \Cris{(D_2)}$. 

Let $c > 6$ be an integer that is coprime to $6Np$ and let $z_1^{\alpha}$ be the image of the Beilinson--Flach class $_c \invs{BF}^{[\invs{F}, \invs{G}_{\alpha}]}_{1, 1}$ (see \S \ref{Interpolation}) under the composition
\[
\opn{H}^1(\mbb{Q}, D^{\mathrm{la}}(\Gamma, M(\invs{F}_{k_1})^* \otimes M((\invs{G}_{\alpha})_{k'})^*)) \to \opn{H}^1_{\Iw}(\mbb{Q}_{p, \infty}, M(\invs{F}_{k_1})^* \otimes M(g)^*) \to \opn{H}^1_{\Iw}(\mbb{Q}_p, D^-)  
\] 
and similarly for $z_1^{\beta}$, where the first map is restriction to the decomposition group at $p$ composed with the isomorphism $\mathrm{pr}^{\alpha}$. Recall that $L_p(\invs{F}, g, 1 +\mathbf{j})$ is the two-variable $p$-adic $L$-function associated to the Coleman family $\invs{F}$ and the universal twist $\mathbf{j}$ (see \cite{Urban} or \cite[\S 9]{lz-coleman}).  

\begin{proposition} \label{proplinearindependence}
We can choose the auxiliary integer $c$ such that, if $L_p(\invs{F}_{k_1}, g, 1 + \eta) \neq 0$ then $\mathrm{pr}_{\eta}(z_1^{\alpha})$ and $\mathrm{pr}_{\eta}(z_1^{\beta})$ are linearly independent in $\opn{H}^1(\mbb{Q}_p, D^-(\eta^{-1}))$.   
\end{proposition}
\begin{proof}
Recall that $D^-$ is a crystalline $(\varphi, \Gamma)$-module satisfying hypothesis (H). Therefore we have the Perrin-Riou logarithm
\[
\invs{L}\colon \opn{H}^1_{\Iw}(\mbb{Q}_{p}, D^-) \rightarrow D^-_{\mathrm{cris}} \hatot \Lambda_{\infty}.
\]
By part (1) in Proposition \ref{propertiesofL}, this map agrees with the map (also denoted by $\invs{L}$) constructed in \cite[Theorem 7.1.4]{lz-coleman} after specialising at $(k_1, k')$. Indeed the map in \emph{loc. cit.} is defined as the pull back of $\invs{L}_{D^-(-1-k')}$ under the automorphism $\omega_{-1-k'}$.

By the ``explicit reciprocity law'' of Theorem 7.1.5 in \emph{op. cit.} we see that $\mathrm{ev}_{\eta}(\invs{L}(z_1^{\alpha}))$ and $\mathrm{ev}_{\eta}(\invs{L}(z_1^{\beta}))$ are both non-zero if the quantity
\begin{equation} \label{ExplicitRecFactor}
\left( c^2 - c^{-(k_1 + k' - 2\eta)} \varepsilon_{\invs{F}_{k_1}}(c)^{-1}\varepsilon_g(c)^{-1} \right) (-1)^{1+ \eta} \lambda_N(\invs{F}_{k_1} )^{-1} L_p(\invs{F}_{k_1}, g, 1+ \eta )
\end{equation}
is non-zero, where $\lambda_N(\invs{F}_{k_1} )$ denotes the specialisation of the Atkin--Lehner pseudo-eigenvalue of $\invs{F}$ (see \cite[\S 2.5]{KLZ2015} for the definition of the Atkin--Lehner operators). Since our assumption at the start of \S \ref{ApplicationTo} implies that $k_1 \neq k'$, we can choose the integer $c$ such that the first factor in (\ref{ExplicitRecFactor}) is non-zero. The result then follows by applying Corollary \ref{corlinearindependence}.
\end{proof}


\section{Bounding the Selmer group}
Let $f$ and $g$ be two cuspidal new eigenforms satisfying the assumptions in section \ref{SummaryOfResults}, and let $\invs{F}$ be a Coleman family over $V_1 \subset \invs{W}_E$ passing through a $p$-stabilisation of $f$. In this section we show that, if $V_1$ is taken to be small enough and certain hypotheses are satisfied, then the cohomology group $\widetilde{\opn{H}}^2_f(\mbb{Q}, \overline{M}_{\mathbf{x}})$ vanishes if $L_p(\mathbf{x}) = 0$. Here $\mathbf{x} = (k_1, k', \eta) \in V_1 \times \{k'\} \times \invs{W}$ is a tuple of weights, $\overline{M}_{\mathbf{x}}$ denotes the representation $[M(\invs{F}_{k_1})^* \otimes M(g)^*](\eta^{-1})$ and $\widetilde{\opn{H}}^2_f(\mbb{Q}, \overline{M}_{\mathbf{x}})$ is the second cohomology group of the Selmer complex defined in section \ref{ConvolutionOfTwo} below.

\subsection{Cohomological preliminaries} \label{CohomologicalPreliminaries}
In \cite{nekovar-selmer}, Nekov\'{a}\v{r} defined the concept of a Selmer complex - an object in a certain derived category whose cohomology is closely related to the usual definition of Selmer groups. This construction is useful because the resulting complex has nice base-change and duality properties; attributes that one doesn't necessarily have for the classical Selmer groups. In \cite{pottharst-selmer}, Pottharst extends this construction to families of Galois representations over well-behaved rigid analytic spaces. This is the tool we will use to construct a sheaf interpolating the Bloch--Kato Selmer groups. We now summarise this construction.  

Let $A$ be an affinoid algebra over $\mbb{Q}_p$ and let $M$ be an $A$-valued representation of $G_{\mbb{Q}} = \Gal(\bar{\mbb{Q}}/\mbb{Q})$ (i.e. a finitely generated, projective $A$-module with a continuous action of $G_{\mbb{Q}}$). Let $\Sigma$ be a set of places of $\mbb{Q}$ containing $p, \infty$ and all primes where $M$ is ramified, and assume that $\Sigma$ is finite. Let $G_{\Sigma}$ denote the Galois group of the maximal algebraic unramified-outside-$\Sigma$ extension of $\mbb{Q}$ and, for a place $v \in \Sigma$, let $G_v$ denote a fixed decomposition group in $G_{\Sigma}$ associated to the place $v$. 

\begin{definition}
A collection of local conditions $\Delta$ for $M$ is a set of pairs $\{ (\Delta_v, \iota_v): v \in \Sigma \}$ where $\Delta_v$ is an object in the derived category of bounded complexes of (continuous) $A$-modules, and $\iota_v$ is a morphism
\[
\Delta_v \xrightarrow{\iota_v} \rcont{G_v}{M}.
\]
\end{definition}

One defines the Selmer complex $\selcomo{M}$ in the following way.

\begin{definition}
Let $\Delta$ be a set of local conditions for $M$. Then the Selmer complex $\selcomo{M}$ is the mapping fibre
\[
\selcomo{M} := \opn{Cone}\left( \rcont{G_{\Sigma}}{M} \oplus \bigoplus_{v \in \Sigma}{\Delta_v} \xrightarrow{\opn{res}_v - \iota_v} \bigoplus_{v \in \Sigma} \rcont{G_v}{M} \right)[-1].
\]
We denote the $i$-th cohomology of this complex by $\esel{i}{M}$.

If $\Delta_v$ is quasi-isomorphic to a complex of finitely generated $A$-modules concentrates in degrees $[0, 2]$ (all our local conditions in this paper will satisfy this), then $\selcomo{M}$ is quasi-isomorphic to a complex of finitely generated $A$-modules, concentrated in degrees $[0, 3]$ (see \cite[\S 1.5]{pottharst-selmer}).
\end{definition}

This construction also works in more general situations. For example, if $X$ is a quasi-Stein rigid analytic space then $\selcomo{M}$ is quasi-isomorphic to a complex of coherent $\ordd_X$-modules, concentrated in degrees $[0, 3]$ (this is situation (4) described in \cite[\S 1.5]{pottharst-selmer}).

\begin{proposition} \label{PropertiesOfSelmerComplexes}
Selmer complexes satisfy the following properties:
\begin{enumerate}
    \item (Duality, \cite[Theorem 1.16]{pottharst-selmer}) Suppose that the local condition $\Delta_v$ is quasi-isomorphic to a perfect complex of $A$-modules concentrated in degrees $[0, 2]$, for all $v \in \Sigma$. We define the dual local conditions $\Delta^*(1)$ to be $\{(\Delta_v^*(1), j_v^*[-2]) \}$, where 
    \[
    \Delta_v^*(1) := Q_v^*[-2] \xrightarrow{j_v^*[-2]} \rcont{G_v}{M}^*[-2] \cong \rcont{G_v}{M^*(1)}
    \]
    and $Q_v$ is the mapping cone of $\iota_v$, we write $j_v: \rcont{G_v}{M} \to Q_v$ for the natural map and $(-)^*$ denotes the dual (in the underived sense). 
    
    One has an isomorphism 
    \[
    \selcom{G_{\Sigma}}{M^*(1)}{\Delta^*(1)} \cong \selcomo{M}^*[-3].
    \]
    \item (Comparison of local conditions) If $\Delta' = \{(\Delta'_v, \iota'_v): v \in \Sigma\}$ is another set of local conditions and $\{\tau_v\}$ are morphisms such that $\iota'_v$ is equal to the composition
\[
\Delta'_v \xrightarrow{\tau_v} \Delta_v \xrightarrow{\iota_v} \rcont{G_v}{M}
\]
then we obtain a Poitou--Tate style long exact sequence
\[
\begin{aligned}
\bigoplus_{v \in \Sigma}\opn{H}^0(Q_v) \rightarrow \eseldel{1}{M}{\Delta'} \rightarrow \eseldel{1}{M}{\Delta} \xrightarrow{\xi} \bigoplus_{v \in \Sigma}\opn{H}^1(Q_v) \rightarrow \\
\rightarrow \eseldel{2}{M}{\Delta'} \rightarrow \eseldel{2}{M}{\Delta} \rightarrow \bigoplus_{v \in \Sigma}\opn{H}^2(Q_v)
\end{aligned}
\]
where $Q_v$ is the mapping cone of $\tau_v$.
\end{enumerate}
\end{proposition}
\begin{proof}
For the second part, this is immediate from the definition of the Selmer complexes associated to the local conditions $\Delta$ and $\Delta'$.
\end{proof}

In \S \ref{AVanishingResult} we will use the Poitou--Tate long exact sequence described above to show that the cohomology of the Selmer complex vanishes in degree two if the corresponding value of the $p$-adic $L$-function is non-zero. In particular we will use the following result:

\begin{proposition} \label{PropositionPT}
If the map $\xi$ in part (2) of Proposition \ref{PropertiesOfSelmerComplexes} is surjective then we have an injective map
\[
\eseldel{2}{M}{\Delta'} \hookrightarrow \eseldel{2}{M}{\Delta}.
\]
In particular, if $\eseldel{2}{M}{\Delta}$ vanishes then so does $\eseldel{2}{M}{\Delta'}$. 
\end{proposition}

\subsection{Convolution of two Coleman families} \label{ConvolutionOfTwo}

Let $\mathbf{x} = (k_1, k', \eta) \in V_1 \times \{k'\} \times \invs{W}$ be a tuple of weights defined over a finite extension $E / \mbb{Q}_p$.

Let $\overline{M}_{\mathbf{x}}$ denote the representation $[M(\invs{F}_{k_1})^* \otimes M(g)^*](\eta^{-1})$ and let $\Sigma$ be a finite set of places of $\mbb{Q}$ that contains $p, \infty$ and all the primes where $\overline{M}_{\mathbf{x}}$ ramifies. Let $D_{\mathbf{x}} \defeq \ddrig{\overline{M}_{\mathbf{x}}}$ denote the $(\varphi, \Gamma)$-module associated to $\overline{M}_{\mathbf{x}}$ and recall that we have a two dimensional quotient $\ddrig{\overline{M}_{\mathbf{x}}} \rightarrow D^-(\eta^{-1}) =: D^-_{\mathbf{x}}$, where $D^-$ is defined as in (\ref{defofD}). We denote the kernel of this quotient by $D^+_{\mathbf{x}}$. For $v \in \Sigma \backslash \{p\}$ we call $(\rcont{G_v/I_v}{{\overline{M}_{\mathbf{x}}}^{I_v}}, \iota_v)$ the \emph{unramified} local condition at $v$, where $\iota_v$ is the natural map induced by inflation. We are interested in the following examples of local conditions:
\begin{itemize}
    \item \emph{(Relaxed)} For $v \in \Sigma$ take $\Delta_{\mathrm{rel}}$ to be the set of unramified local conditions for $v \neq p$ and
    \[
    \Delta_{\mathrm{rel}, p} := \rcont{G_p}{\overline{M}_{\mathbf{x}}} \xrightarrow{\sim} \rcont{G_p}{\overline{M}_{\mathbf{x}}}.
    \]
    We denote the cohomology of the associated Selmer complex by $\selmergrprel{i}{\overline{M}_{\mathbf{x}}}$.
    \item \emph{(Strict)} For $v \in \Sigma$ take $\Delta_{\mathrm{str}}$ to be the set of unramified local conditions for $v \neq p$ and
    \[
    \Delta_{\mathrm{str}, p} := 0 \rightarrow \rcont{G_p}{\overline{M}_{\mathbf{x}}}.
    \]
    We denote the cohomology of the associated Selmer complex by $\selmergrpstr{i}{\overline{M}_{\mathbf{x}}}$.
    \item \emph{(Panchishkin)} For $v \in \Sigma$ take $\Delta_{f}$ to be the set of unramified local conditions for $v \neq p$ and
    \[
    \Delta_{f, p} := \rcont{G_p}{D^+_{\mathbf{x}}} \rightarrow \rcont{\mbb{Q}_p}{\ddrig{\overline{M}_{\mathbf{x}}}} \cong \rcont{G_p}{\overline{M}_{\mathbf{x}}}.
    \]
    We denote the cohomology of the associated Selmer complex by $\selmergrp{i}{\overline{M}_{\mathbf{x}}}$. The reason for choosing this local condition is because it is closely related to the Bloch--Kato local condition when $\mathbf{x}$ lies in the \emph{critical range} (the range where the $p$-adic $L$-function interpolates critical values of the global $L$-function). We will discuss this relation in \S \ref{RelationBKSelmer}.
\end{itemize}

\begin{remark}
\begin{enumerate} 
\item All three of the above Selmer complexes do not change if we enlarge the set $\Sigma$, so we suppress this auxiliary set from the notation.
\item The relaxed and strict conditions are dual to each other. The dual of the Panchishkin local condition is a Panchishkin local condition for $\ddrig{\overline{M}_{\mathbf{x}}}^*(1) = \ddrig{{\overline{M}_{\mathbf{x}}}^*(1)}$.
\end{enumerate}
\end{remark}

Let $c > 6$ be an integer prime to $6Np$ and recall from \S \ref{Interpolation} that, for an integer $m$ coprime to $6Ncp$, there is a Beilinson--Flach class 
\[
_c\invs{BF}^{[\invs{F}, \invs{G}]}_{m, 1} \in \opn{H}^1(\mbb{Q}(\mu_m), D^{\mathrm{la}}(\Gamma, M) )
\]
where $M = M_{V_1}(\invs{F})^* \hatot M_{V_2}(\invs{G})^*$. By specialising these Beilinson--Flach classes at $\mathbf{x}$ and identifying $M((\invs{G}_{\alpha})_{k'})$ and $M((\invs{G}_{\beta})_{k'})$ with $M_E(g)$ as before, we obtain classes in $\opn{H}^1(\mbb{Q}(\mu_m), \overline{M}_{\mathbf{x}} )$.

In section \ref{ESinFamilies} we showed that these classes satisfy certain norm relations and that we could produce an Euler system from these classes. More precisely, let $\overline{T}_{\mathbf{x}}$ be a Galois stable lattice inside $\overline{M}_{\mathbf{x}}$. Then there exist collections $\{c^{\alpha}_m \in \opn{H}^1(\mbb{Q}(\mu_m), \overline{T}_{\mathbf{x}}) : m \in S \}$ and $\{c^{\beta}_m \in \opn{H}^1(\mbb{Q}(\mu_m), \overline{T}_{\mathbf{x}}) : m \in S \}$ satisfying the Euler system relations, and $c_1^{\alpha}$ and $c_1^{\beta}$ are equal to non-zero multiples of the specialisations of ${\beilfd{1}{\invs{G}_{\alpha}}}$ and ${\beilfd{1}{\invs{G}_{\beta}}}$ at $\mathbf{x}$ respectively. 

We choose the integer $c$ in such a way that the conclusion of Proposition \ref{proplinearindependence} holds at the point $\mathbf{x}$ (this choice may depend on $\mathbf{x}$). After making this choice, the classes $c_{m}^{\alpha}, c_{m}^{\beta}$ satisfy the following local conditions. 

\begin{proposition} \label{PropositionUnramified}
Keeping the same notation as above, the classes $c_{m}^{\alpha}, c_{m}^{\beta}$ satisfy the following properties:
\begin{enumerate}
\item Both $c_{m}^{\alpha}$ and $c_{m}^{\beta}$ are unramified outside $p$. In particular, both $c_1^{\alpha}$ and $c_1^{\beta}$ lie in $\selmergrprel{1}{\overline{M}_{\mathbf{x}}}$. This implies that the collection $\{c_m : m \in S \}$ forms an Euler system in the sense of Definition 2.1.1 in \cite{Rubin}, with condition (ii) replaced by (ii)'(b) (see \S 9.1 in \emph{op.~cit.}).
\item Let $\bar{c}_{1}^{\alpha}, \bar{c}_{1}^{\beta}$ denote the images of $c_{1}^{\alpha}, c_{1}^{\beta}$ under the map
\[
\selmergrprel{1}{\overline{M}_{\mathbf{x}}} \xrightarrow{\xi} \opn{H}^1(\mbb{Q}_p, D^-_{\mathbf{x}})
\] 
where $\xi$ is given by first restricting to $p$ and then mapping to the quotient (this map is the same map as in Proposition \ref{PropositionPT} if we compare the relaxed and Panchishkin local conditions defined above). Then, if $\bar{c}_{1}^{\alpha}$ and  $\bar{c}_{1}^{\beta}$ are both non-zero, they are linearly independent. In particular this happens when $L_p(\invs{F}_{k_1}, g, 1 +\eta) \neq 0$. 
\end{enumerate}
\end{proposition}
\begin{proof}
The first part is the same proof as in \cite[Theorem 8.1.4]{lz-coleman}. For the second part, note that $\bar{c}_{1}^{\alpha}$ and  $\bar{c}_{1}^{\beta}$ are two elements satisfying the conditions of Corollary \ref{corlinearindependence}. The result then follows from Proposition \ref{proplinearindependence}.
\end{proof}

\subsection{A vanishing result} \label{AVanishingResult}

Let $\invs{F}$ be a Coleman family over an affinoid domain $V$. For ease of notation we set $\alpha_{\invs{F}} \defeq a_p(\invs{F})$ and similarly for specialisations of $\invs{F}$. Moreover, recall that if the specialisation of $\invs{F}$ at $k_1$ is a noble eigenform (so it is the $p$-stabilisation of an eigenform $h$) then $\alpha_{\invs{F}_{k_1}} = \alpha_h$ and $\beta_h$ denote the roots of the Hecke polynomial at $p$ associated to $h$, and satisfy $\alpha_h \beta_h = p^{k_1 + 1}\varepsilon_h(p)$. In this case, we will also write $\beta_{\invs{F}_{k_1}} \defeq \beta_h$ (although the notation $\beta_{\invs{F}}$ is of course meaningless).

The goal of this section is to show that if the $p$-adic $L$-function doesn't vanish at $\mathbf{x}$ then the Selmer group $\selmergrp{2}{\overline{M}_{\mathbf{x}}}$ is trivial. The strategy is to combine Propositions \ref{PropositionPT} and \ref{PropositionUnramified} by comparing the Panchishkin and relaxed local conditions. In particular, we must show that the hypothesis in Proposition \ref{PropositionPT} is satisfied. Unfortunately this is not true in general and fails when $\overline{M}_{\mathbf{x}}$ has a ``local zero'', i.e. the local Euler factor of $\overline{M}_{\mathbf{x}}$ at $p$ vanishes at $s = 1$. Therefore we impose the following hypothesis on $\overline{M}_{\mathbf{x}}$: 

\begin{itemize}
\item[(NLZ)] None of the products
\[
\{ \alpha_{\invs{F}_{k_1}} \alpha_{g}, \; \;  \alpha_{\invs{F}_{k_1}} \beta_{g}, \; \; \alpha_{\invs{F}_{k_1}}^{-1} \varepsilon_{\invs{F}_{k_1}}(p) \alpha_{g}, \; \; \alpha_{\invs{F}_{k_1}}^{-1} \varepsilon_{\invs{F}_{k_1}}(p) \beta_{g} \}
\]
are equal to $p^j$ for some integer $j$ (recall that $\mathbf{x} = (k_1, k', \eta)$). 
\end{itemize}

\begin{remark}
The (NLZ) hypothesis is an open condition, i.e. if it holds at the point $\mathbf{x}$ then it also holds for all specialisations in an open neighbourhood of $\mathbf{x}$. In particular, if $\invs{F}$ is a Coleman family passing through a $p$-stabilisation of $f$ defined over an affinoid subdomain $V_1 \subset \invs{W}_E$, and if the (NLZ) hypothesis holds for $f$ and $g$, i.e. none of the products 
\[
\{ \alpha_f \alpha_g, \alpha_f \beta_g, \beta_f \alpha_g, \beta_f \beta_g \}
\]
are equal to a power of $p$, then we can shrink $V_1$ so that the (NLZ) hypothesis holds for all specialisations of $\overline{M}$ at $\mathbf{x} = (k_1, k', \eta) \in V_1 \times \{k'\} \times \WW$.
\end{remark}

The second ingredient to proving the vanishing result is to apply the ``Euler system machine'' to the representation $\overline{M}_{\mathbf{x}}$. To be able to apply this we need to assume the following ``Big Image'' hypothesis.
\begin{itemize}
\item[(BI)] There exists an element $\sigma \in \Gal(\bar{\mbb{Q}}/\mbb{Q}(\mu_{p^{\infty}}))$ such that $\overline{M}_{\mathbf{x}}/(\sigma - 1)\overline{M}_{\mathbf{x}}$ is one-dimensional (over $E$).
\end{itemize} 

\begin{remark}
It turns out that for the ``Big Image'' hypothesis to hold we only need to assume that the image of the mod $p$ representation of $\overline{M}_{\mathbf{x}}$ is sufficiently large, and this is almost always the case provided that $\FF_{k_1}$ and $g$ are not of CM type and $\FF_{k_1}$ is not Galois conjugate to a twist of $g$. In particular, since the mod $p$ representation of a Coleman family is locally constant, this implies that the ``Big Image'' hypothesis will hold in an open neighbourhood of the point $\mathbf{x}$. We provide justifications for this in the appendix (\S \ref{Appendix}). 
\end{remark}

Under these two assumptions we have the following vanishing result.

\begin{theorem} \label{TheoremVanishing}
Keeping the same notation at the start of section \ref{ConvolutionOfTwo}, assume that the (NLZ) and (BI) hypotheses hold. If $L_p(\invs{F}_{k_1}, g, 1+\eta) \neq 0$ then $\selmergrp{2}{\overline{M}_{\mathbf{x}}} = 0$.
\end{theorem}
\begin{proof}
Consider the local conditions $\Delta = \Delta_{\mathrm{rel}}$ and $\Delta' = \Delta_f$ and suppose for the moment that, as in the statement of Proposition \ref{PropositionPT}, the map $\xi$ is surjective. Then there is an injective map
\[
\selmergrp{2}{\overline{M}_{\mathbf{x}}} \hookrightarrow \selmergrprel{2}{\overline{M}_{\mathbf{x}}}
\]
so it is enough to show that $\selmergrprel{2}{\overline{M}_{\mathbf{x}}} = 0$. 

Let $\overline{T}_{\mathbf{x}}$ be a Galois stable lattice inside $\overline{M}_{\mathbf{x}}$ and set $A = \overline{M}_{\mathbf{x}}^*(1)/\overline{T}_{\mathbf{x}}^*(1)$. Then by the duality of the relaxed and strict local conditions we see that $\selmergrprel{2}{\overline{T}_{\mathbf{x}}}^{\vee} := \opn{Hom}_{\ordd_E}(\selmergrprel{2}{\overline{T}_{\mathbf{x}}}, E/\ordd_E)$ is equal to $\widetilde{\opn{H}}^1_{\mathrm{str}}(\mbb{Q}, A)$ (see \cite{nekovar-selmer} for more details). Furthermore, since $\selmergrprel{2}{\overline{M}_{\mathbf{x}}} = \selmergrprel{2}{\overline{T}_{\mathbf{x}}}[1/\varpi]$, where $\varpi$ is a uniformiser for $\ordd_E$, it is enough to show $\widetilde{\opn{H}}^1_{\mathrm{str}}(\mbb{Q}, A)$ is finite.

Recall that $\{c^{\alpha}_m\}$ forms an Euler system for $\overline{T}_{\mathbf{x}}$ and the bottom class $c^{\alpha}_1$ is non-zero because $L_p(\invs{F}_{k_1}, g, 1+\eta)$ is non-zero. Coupling this with the (BI) assumption, we can apply \cite[Theorem 2.2.3]{Rubin} and conclude that $\widetilde{\opn{H}}^1_{\mathrm{str}}(\mbb{Q}, A)$ is finite. Indeed, by comparing the strict and relaxed local conditions, the group $\widetilde{\opn{H}}^1_{\mathrm{str}}(\mbb{Q}, A)$ differs from the strict Selmer group in \emph{op.cit.} by the group $\opn{H}^0(\mbb{Q}_p, A)$, which is finite because $\opn{H}^0(\mbb{Q}_p, \overline{M}_{\mathbf{x}}^*(1)) = 0$. So we are left to show the map $\xi$ is surjective. 

The mapping cone $Q_p$ is precisely the same thing as the image of the Herr complex $C_{\varphi, \gamma}^{\bullet}(D^-_{\mathbf{x}})$ and by the local Euler characteristic formula (Proposition \ref{EulerCharFormula}) for $(\varphi, \Gamma)$-modules, we have $\chi(D^-_{\mathbf{x}}) = -2$. Therefore if we show that $\opn{H}^0(\mbb{Q}_p, D^-_{\mathbf{x}})$ and $\opn{H}^2(\mbb{Q}_p, D^-_{\mathbf{x}})$ both vanish, then this would imply that $\opn{H}^1(\mbb{Q}_p, D^-_{\mathbf{x}})$ is two-dimensional. Combining this with part 2 of Proposition \ref{PropositionUnramified}, this would imply that the map $\xi$ is surjective.

By duality we have $\opn{H}^2(\mbb{Q}_p, D^-_{\mathbf{x}}) \cong \opn{H}^0(\mbb{Q}_p, (D^-_{\mathbf{x}})^*(1))^*$ and from the explicit description of the triangulation (\cite[Theorem 6.3.2]{lz-coleman}) we have the following short exact sequences:
\[
\begin{aligned}
\exactseq{[\bdrigE](\alpha_{\invs{F}_{k_1}}^{-1}\alpha_g \varepsilon_g(p)^{-1})(\chi_{\mathrm{cycl}}^{1+k'}\cdot\eta^{-1})}{D^-_{\mathbf{x}}}{[\bdrigE](\alpha_{\invs{F}_{k_1}}^{-1}\alpha_g^{-1})(\eta^{-1})}  \\ 
 \\
\exactseq{[\bdrigE](\alpha_{\invs{F}_{k_1}}\alpha_g)(\chi_{\mathrm{cycl}} \cdot \eta)}{(D^-_{\mathbf{x}})^*(1)}{[\bdrigE](\alpha_{\invs{F}_{k_1}}\alpha_g^{-1} \varepsilon_g(p))(\chi_{\mathrm{cycl}}^{-k'} \cdot \eta)} \nonumber 
\end{aligned}
\]
where we denote by $[\bdrigE](\lambda)(\omega)$ the one-dimensional $(\varphi, \Gamma)$-module over $\bdrigE$ with a basis $e$ such that $\varphi(e) = \lambda e$ and $\gamma \cdot e = \omega(\gamma)e$ for all $\gamma \in \Gamma$.

From the above sequences, one sees that if either $\opn{H}^0(\mbb{Q}_p, D^-_{\mathbf{x}})$ or $\opn{H}^2(\mbb{Q}_p, D^-_{\mathbf{x}})$ didn't vanish then this would contradict the (NLZ) hypothesis. 
\end{proof}

\begin{remark}
To prove the above theorem, we only needed to assume that the two products $\alpha_{\FF_{k_1}} \alpha_g$ and $\alpha_{\FF_{k_1}} \beta_g$ are not equal to a power of $p$. However, in the following section we will relate $\selmergrp{2}{\overline{M}_{\mathbf{x}}}$ to the usual Bloch--Kato Selmer group at classical specialisations, and for this we will need to assume that all four products in the statement of (NLZ) are not equal to a power of $p$. 

Furthermore, for most non-classical specialisations we do not have to impose a (NLZ) condition. Indeed by Proposition 2.1 and Th\'{e}or\`{e}me 2.9 in \cite{ColmezRepTrianguline}, it is often the case that $\opn{H}^1(D^-_{\mathbf{x}})$ is automatically two-dimensional unless the weights in $\mathbf{x}$ are classical.
\end{remark}

\subsection{Relation to the Bloch--Kato Selmer group} \label{RelationBKSelmer}

Theorem \ref{TheoremVanishing} is a generalisation of \cite[Theorem 8.2.1]{lz-coleman} to non-classical specialisations. Indeed, suppose that $k_1$ and $\eta = \chi_{\mathrm{cycl}}^j$ are classical and we have $k' +1 \leq j \leq k_1$. Then by the duality property of Selmer complexes and the fact that the Panchishkin condition is self-dual, we have $\selmergrp{2}{\overline{M}_{\mathbf{x}}}^* \cong \selmergrp{1}{\overline{M}_{\mathbf{x}}^*(1)}$. But by the (NLZ) hypothesis, we have the following equalities:
\begin{itemize}
    \item $\opn{H}^0(\mbb{Q}_p, (D^+_{\mathbf{x}})^*(1)) = 0$.
    \item $\opn{H}^0(\mbb{Q}_p, D_{\mathbf{x}}/D^+_{\mathbf{x}}) = \opn{H}^0(\mbb{Q}_p, D^-_{\mathbf{x}}) = 0$.
\end{itemize}
Indeed, by the conditions on the Hodge--Tate weights, we have 
\begin{itemize}
    \item $\opn{H}^0(\mbb{Q}_p, (D^+_{\mathbf{x}})^*(1)) = \Cris{(D^+_{\mathbf{x}})^*(1)}^{\varphi = 1}$.
    \item $\opn{H}^0(\mbb{Q}_p, D_{\mathbf{x}}/D^+_{\mathbf{x}}) = \Cris{(D_{\mathbf{x}}/D^+_{\mathbf{x}})}^{\varphi = 1}$.
\end{itemize}
But $\varphi$ has eigenvalues $\{p^{-1-j}\beta_{\invs{F}_{k_1}}\alpha_g, p^{-1-j}\beta_{\invs{F}_{k_1}} \beta_g \}$ and $\{p^j \alpha_{\invs{F}_{k_1}}^{-1}\alpha_g^{-1}, p^j \alpha_{\invs{F}_{k_1}}^{-1}\beta_g^{-1} \}$ on $\Cris{(D^+_{\mathbf{x}})^*(1)}$ and $\Cris{(D_{\mathbf{x}}/D^+_{\mathbf{x}})}$ respectively, and these products can never be equal to $1$ by the (NLZ) hypothesis. Therefore, by \cite[Proposition 3.7]{pottharst-selmer}, we see that
\[
\selmergrp{1}{\overline{M}_{\mathbf{x}}^*(1)} = \opn{H}^1_f(\mbb{Q}, \overline{M}_{\mathbf{x}}^*(1))
\]
where the latter is the Bloch--Kato Selmer group. This recovers Theorem 8.3.1 in \cite{lz-coleman}. In fact the proof of Theorem \ref{TheoremVanishing} is modelled on the proof in \emph{loc. cit.}.


\section{The Selmer sheaf} \label{selmer sheaf}

In the previous section we showed that (under certain hypotheses) if the specialisation of $L_p$ is non-zero then $\selmergrp{2}{\overline{M}_{\mathbf{x}}} = 0$. It turns out that we can package together all of these cohomology groups into a coherent sheaf over $V_1 \times V_2 \times \invs{W}$ using the machinery of Selmer complexes. We follow closely the construction in \cite{pottharst-selmer}. 

\subsection{Assumptions} \label{SelmerSheafAssumptions}

Recall that $\invs{W}/\mbb{Q}_p$ denotes the rigid analytic space parameterising continuous characters $\Gamma = \Gal(\mbb{Q}(\mu_{p^{\infty}})/\mbb{Q}) \to \mbb{C}_p^{\times}$. Let $V_1$ and $V_2$ be two affinoid subdomains of $(\invs{W})_E$ and set $X = V_1 \times V_2 \times \invs{W}$. Then $X$ has admissible cover $\invs{U} = (Y_n)_{n \geq 1}$ given by 
\[
Y_n = V_1 \times V_2 \times \invs{W}_n
\]  
where $\invs{W}_n$ is the open affinoid subdomain of $\invs{W}$ parameterising all characters $\eta$ that satisfy $|\eta(\gamma)^{p^{n-1}} - 1|_p \leq p^{-1}$, where $\gamma$ is a topological generator of $\Gamma/\Gamma_{\mathrm{tors}}$.
The restriction maps $\ordd(Y_{n+1}) \to \ordd(Y_n)$ have dense image.  Hence $X$ is a quasi-Stein space.

Let $A_{\infty} \defeq \OO(X)$, $A_n \defeq \OO(Y_n)$.  Note that for all $n \geq 1$, $A_n$ is flat over $A_{\infty}$.

Let $\invs{F}$ and $\invs{G}$ be two Coleman families over $V_1$ and $V_2$ passing through $p$-stabilisations of $f$ and $g$ respectively, and let $M$ denote the representation $M(\invs{F})^* \hatot M(\invs{G})^*$. Fix a Galois stable lattice $T$ inside $M$, i.e. a free rank four $\ordd(V_1 \times V_2)^\circ$-submodule that is stable under the action of $G_{\mbb{Q}}$.

Let $\overline{M} = D^{\mathrm{la}}(\Gamma, M)$ denote the cyclotomic deformation of $M$, and for any $n \geq 1$ we set $M_n = M(-\kappa_n)$, where $(-\kappa_n)$ denotes the twist by the inverse of the universal character of $\invs{W}_n$. Then we also obtain a Galois stable lattice $T_n := T(-\kappa_n)$ inside $M_n$.

Let $\Sigma$ be a finite set of places containing $p, \infty$ and all primes where $M$ ramifies. Then $\overline{M}$ is a family of $G_{\Sigma}$-representations over the space $X$ and we place ourselves in situation (4) in \cite[\S 1.5]{pottharst-selmer}.

By an $\ordd_{\invs{U}}$-module we mean a compatible system of $A_n$-modules. Let $\rcont{G_{\Sigma}}{\overline{M}}$ denote the image of the complex of continuous cochains $C^\bullet_{\mathrm{cont}}(G_{\Sigma}, \overline{M})$ in the derived category of $\ordd_{\invs{U}}$-modules. Explicitly, $C^\bullet_{\mathrm{cont}}(G_{\Sigma}, \overline{M})$ is defined by the rule
\[
Y_n \mapsto C^\bullet_{\mathrm{cont}}(G_{\Sigma}, M_n)
\]
where we note that $M_n = \Gamma(Y_n, \overline{M})$.

\begin{lemma} \label{PullbackLemma}
\begin{enumerate}
    \item $\rcont{G_{\Sigma}}{\overline{M}}$ is a perfect complex, in the sense that it is quasi-isomorphic to a complex $D^\bullet$, concentrated in finitely many degrees, such that $\Gamma(Y_n, D^\bullet)$ is a finite projective $A_n$-module. 
    \item Let $\iota_n: Y_n \hookrightarrow X$ denote the inclusion. Then
    \[
    \mathbf{L}\iota_n^* \rcont{G_{\Sigma}}{\overline{M}} \cong \rcont{G_{\Sigma}}{M_n}.
    \]
\end{enumerate}
\end{lemma}
\begin{proof}
For the first part, this follows from the discussion in \cite[\S 1.2]{pottharst-selmer}, and the second part is just Theorem 1.6 in \emph{op.cit.}.
\end{proof}

Since $X$ is a quasi-Stein space we also have an alternative description of $\rcont{G_{\Sigma}}{\overline{M}}$, namely as the image of the complex $\varprojlim_n C^\bullet_{\mathrm{cont}}(G_{\Sigma}, M_n)$ in the derived category of $A_{\infty}$-modules. By the above lemma and Kiehl's theorem, $\rcont{G_{\Sigma}}{\overline{M}}$ is quasi-isomorphic to a complex of locally free (of finite rank) $\ordd_X$-modules, so in particular its cohomology groups are coherent sheaves on $X$. Furthermore, since $X$ is quasi-Stein, a coherent sheaf on $X$ is determined by its global sections, so we will often use these two descriptions interchangeably. We say an $A_{\infty}$-module is \emph{coadmissible} if it arises as the global sections of a coherent sheaf on $X$.

As in \S \ref{CohomologicalPreliminaries}, for a collection $\Delta = \{\Delta_v\}_{v \in \Sigma}$ of local conditions
\[
\begin{aligned}
\Delta_v \xrightarrow{\iota_v} \rcontl{\overline{M}} & \; \; \; \;  & \Delta_v \in \mathbf{D}_{\mathrm{ft}}^{[0, 2]}(A_{\infty}\text{-Mod})
\end{aligned}
\]
where $\mathbf{D}_{\mathrm{ft}}^{[0, 2]}(A_{\infty}\text{-Mod})$ is the derived category of complexes of $A_{\infty}$-modules concentrated in degrees $[0, 2]$ whose cohomology groups are coadmissible, we can construct the Selmer complex $\selcomo{\overline{M}}$ which is an object in the derived category of $A_{\infty}$-modules, concentrated in degrees $[0, 3]$, whose cohomology groups are coadmissible (see \cite[\S 1.5]{pottharst-selmer}). 

We impose the following assumptions on $f$ and $g$. Let $x_0 = (k, k', 0) \in X$, where $k+2$ and $k'+2$ are the weights of $f$ and $g$ respectively, and choose $n$ such that $x_0 \in Y_n$.
\begin{enumerate}
\item \emph{(Flatness of inertia)} If $\ide{p}_0$ denotes the prime ideal of $A_n^\circ$ corresponding to the point $x_0$, then we let $A_{n, \ide{p}_0}^\circ$ and $T_{n, \ide{p}_0}$ denote the localisations of $A_n^\circ$ and $T_n$ at $\ide{p}_0$. For every place $v \in \Sigma$ not equal to $p$, we let $I_v$ denote the inertia subgroup of the fixed decomposition group at $v$.

Then we assume that $T_{n, \ide{p}_0}^{I_v}$ is a \emph{flat} $A_n^\circ$-module, for all $v \in \Sigma \backslash \{p\}$. Since $A_n^\circ$ is a commutative Noetherian local ring, this is equivalent to $T_{n, \ide{p}_0}^{I_v}$ being free.  

In particular, by generic flatness this implies that there exists a Zariski open subset $U$ of $Y_n$ containing $x_0$ such that $(T_n)_U^{I_v}$ is a flat $\ordd(U)^\circ$-module.

\item \emph{(Minimally ramified)} Let $\ide{m}_0$ denote the maximal ideal in $A_n^{\circ}$ containing a uniformiser $\varpi$ of $E$ and the prime ideal $\ide{p}_0$ corresponding to $x_0$. Then $A_n^{\circ}/\ide{m}_0 = k$, where $k$ is the residue field of $\ordd_E$, and the ``mod $p$ representation'' of $T_n$ at the point $x_0$ is defined to be
\[
T_{\bar{\mbb{F}}_p} := (T_n \otimes_{A_n^{\circ}} \bar{\mbb{F}}_p)^{\mathrm{ss}}
\]
where ss stands for semi-simplification and the tensor product is via the map $A_n^{\circ} \to A_n^{\circ}/\ide{m}_0 \hookrightarrow \bar{\mbb{F}}_p$. 

We assume that we have the following equality 
\[
\opn{dim}_k T_n^{I_v}/\ide{m}_0 = \opn{dim}_{\bar{\mbb{F}}_p} (T_{\bar{\mbb{F}}_p})^{I_v}
\]
for all $v \in \Sigma \backslash \{p\}$. In other words, $f$ and $g$ are not congruent to forms of a lower level.    
\end{enumerate}

\begin{remark}
We give examples of pairs of modular forms satisfying the above assumptions in the appendix (\S \ref{JustificationForFI}).
\end{remark}

\begin{lemma} \label{FlatnessLemma}
Assume that the conditions (1) and (2) above hold. Then, after possibly shrinking $V_1$ and $V_2$, we have 
\begin{itemize}
    \item For all $n \geq 1$ and $v \in \Sigma \backslash \{p\}$, $T_n^{I_v}$ is a flat $A_n^\circ$-module.
    \item Let $\eta \in \invs{W}_n$ be a closed point and let $\mathbf{x} = (k, k', \eta)$ (so $\mathbf{x} \in Y_n$). If $\ide{m}$ denotes the maximal ideal of $A_n^\circ$ containing $\varpi$ and the prime ideal associated to $\mathbf{x}$, then for all $v \in \Sigma \backslash \{p\}$
    \[
    \opn{dim}_k T_n^{I_v}/\ide{m} = \opn{dim}_{\bar{\mbb{F}}_p} \left( (T_n \otimes_{A_n^\circ} \bar{\mbb{F}}_p )^{\mathrm{ss}} \right)^{I_v}
    \]
    where $k = A_n^\circ/\ide{m}$ and the tensor product in the right-hand side is via the map $A_n^\circ \to k \hookrightarrow \bar{\mbb{F}}_p$.
\end{itemize}
\end{lemma}
\begin{proof}
By shrinking $V_1$ and $V_2$ if necessary, we can assume that the set $U$ in (1) above contains $V_1 \times V_2 \times U'$ for some open affinoid $U' \subset \invs{W}_n$. The lemma then follows from the fact that any character of $\Gamma$ restricted to $I_v$ is trivial, for $v \neq p$.
\end{proof}

\subsection{The Selmer sheaf} \label{ConstructionOfSelmerSheaf}

We first fix some notation. If $D_1$ and $D_2$ are rank two $(\varphi, \Gamma)$-modules over $\bdrigE$ equipped with a triangulation
\[
0 \rightarrow \F{+}{D_i} \rightarrow D_i \rightarrow \F{-}{D_i} \rightarrow 0
\]
where $\mathscr{F}^{\pm}D_i$ are rank one $(\varphi, \Gamma)$-modules, then we set 
\[
\Ff{\diamondsuit}{\clubsuit}{D} = \F{\diamondsuit}{D_1} \hatot \F{\clubsuit}{D_2}
\]
for $\diamondsuit, \clubsuit \in \{+, -, o \}$, where $D = D_1 \hatot D_2$.

Let $\ddrig{M}$ denote the $(\varphi, \Gamma)$-module over $\bdrigA{A}$ associated to the representation $M$, where $A = \ordd(V_1) \hatot \ordd(V_2)$. As in section \ref{ApplicationTo}, assuming $V_1$ and $V_2$ are small enough we have two possible triangulations for $\ddrig{M}$, namely 
\[
\{0\} \subset \Ff{+}{+}{\ddrig{M}} \subset \Ff{+}{o}{\ddrig{M}} \subset \Ff{+}{o}{\ddrig{M}} + \Ff{o}{+}{\ddrig{M}} \subset \ddrig{M}
\]
and
\[
\{0\} \subset \Ff{+}{+}{\ddrig{M}} \subset \Ff{o}{+}{\ddrig{M}} \subset \Ff{+}{o}{\ddrig{M}} + \Ff{o}{+}{\ddrig{M}} \subset \ddrig{M}
\]
differing only by the middle term in the filtration. 

Let $\mathbf{x} = (k_1, k_2, j)$ be a classical point in $X$ satisfying $0 \leq k_2 < k_1$ and denote the specialisation of $\overline{M}$ at $\mathbf{x}$ by $\overline{M}_{\mathbf{x}}$. The Hodge--Tate weights of $\overline{M}_{\mathbf{x}}$ are 
\[
-j, \; \;  k_2+1 - j, \; \;  k_1+1 - j, \; \;  k_1 + k_2 +2 - j
\]
We want to define a sheaf that interpolates the classical Bloch--Kato Selmer group; the correct local condition that we will need to take will therefore depend on the range we want to interpolate over. For example
\begin{itemize}
\item Suppose that $\mbf{x}$ lies in the \emph{geometric range}, i.e. one has $0 \leq j \leq \opn{min}\{k_1, k_2\}$. Then one can take the local condition at $p$ to be the cohomology of $\mathscr{F}^{+o} + \mathscr{F}^{o+}$. Indeed, this specialises to a Panchishkin submodule at $\mbf{x}$ (recall that a Panchishkin submodule of a de Rham $(\varphi, \Gamma)$-module $D$ is a submodule $D^+$ such that $D^+$ (resp. $D/D^+$) has positive (resp. non-negative) Hodge--Tate weights).
\item Suppose that $\mbf{x}$ lies in the \emph{critical range}, i.e. one has $k_2+1 \leq j \leq k_1$. Then one can take the local condition at $p$ to be the cohomology of $\mathscr{F}^{+o}$.
\end{itemize}
In this paper we are interested in interpolating in the critical range, since it is precisely the range where the $p$-adic $L$-function interpolates (critical) values of the global $L$-function.

We denote by $\ddrig{\overline{M}}$ the family of $(\varphi, \Gamma)$-modules over $X$ satisfying $\Gamma(Y_n, \ddrig{\overline{M}}) = \ddrig{M_n}$. This comes equipped with the triangulations $\Ff{\diamondsuit}{\clubsuit}{\ddrig{\overline{M}}}$ defined previously, i.e. $\Ff{\diamondsuit}{\clubsuit}{\ddrig{\overline{M}}}$ is the family of $(\varphi, \Gamma)$-modules satisfying 
\[
\Gamma(Y_n, \Ff{\diamondsuit}{\clubsuit}{\ddrig{\overline{M}}}) = \Ff{\diamondsuit}{\clubsuit}{\ddrig{M}}(-\kappa_n)
\]
where $(-\kappa_n)$ denotes the twist by the inverse of the universal character of $\invs{W}_n$.

We consider the following set of local conditions $\Delta = \{\Delta_v\}_{v \in \Sigma}$ where 
\begin{itemize}
\item For $v \neq p$, $\Delta_v$ is the unramified condition, i.e. $\Delta_v$ is the complex
\[
\Delta_v := \rcont{G_v/I_v}{\overline{M}^{I_v}} \rightarrow \rcontl{\overline{M}}
\]
\item For $v = p$ we take $\Delta_p$ to be the Panchishkin local condition given by
\[
\Delta_p := \rcont{G_p}{\Ff{+}{o}{\ddrig{\overline{M}}}} \rightarrow \rcont{G_p}{\ddrig{\overline{M}}} \cong \rcont{G_p}{\overline{M}}
\]
where $\rcont{G_p}{\Ff{+}{o}{\ddrig{\overline{M}}}}$ denotes the image of the family of Herr complexes $\invs{C}_{\varphi, \gamma}^{\bullet}(\Ff{+}{o}{\ddrig{\overline{M}}})$ in the derived category of $\ordd_{\invs{U}}$-modules, as defined in \S \ref{CohomologyPHIGAMMA}.
\end{itemize}

By Lemma \ref{FlatnessLemma}, if $V_1$ and $V_2$ are small enough (which we will assume from now on) then the above local conditions lie in $\mathbf{D}_{\mathrm{ft}}^{[0, 2]}(A_{\infty}\mathrm{-Mod})$ so we can talk about the corresponding Selmer complex.

\begin{remark}
We have defined the local conditions in terms of $\ordd_{\invs{U}}$-modules, but this is equivalent to specifying local conditions in terms of coadmissible modules by \cite[Theorem 1.13]{pottharst-selmer} and the discussion preceding it.
\end{remark}

\begin{definition}
Let $\selcomof{\overline{M}}$ denote the Selmer complex associated to $\overline{M}$ and the local conditions $\Delta$, with cohomology groups denoted by $\widetilde{H}_f^i(\mbb{Q}, \overline{M})$. 
\end{definition}

As explained in the paragraph preceding Proposition \ref{PropertiesOfSelmerComplexes}, the groups $\widetilde{H}_f^i(\mbb{Q}, \overline{M})$ are coherent sheaves on $X$ and satisfy $\widetilde{H}_f^i(\mbb{Q}, \overline{M}) = 0$ for $i \neq 0, 1, 2, 3$. 

\begin{theorem} \label{SheafTheorem}
We can take $V_1$ and $V_2$ small enough such that the following hold:
\begin{enumerate}
\item For each $n \geq 0$, $A_n$ is a flat $A_{\infty}$-module and the natural map 
\[
\selcomof{\overline{M}} \otimes_{A_{\infty}} A_n \rightarrow \selcomof{M_n}
\]
is an isomorphism, where $\selcomof{M_n}$ denotes the Selmer complex associated to $M_n = \Gamma(Y_n, \overline{M})$ with unramified local conditions at $v \neq p$ and the Panchishkin condition $\invs{C}_{\varphi, \gamma}^{\bullet}(\Ff{+}{o}{\ddrig{M_n}})$ at $p$. 
\item Let $\mbf{x}$ be an $E'$-valued point in $Y_n$ and $v$ a prime in $\Sigma$ not equal to $p$. Then the natural map 
\[
(M_n^{I_v})_\mbf{x} \rightarrow (\overline{M}_{\mathbf{x}})^{I_v}
\] 
is an isomorphism and hence we have an isomorphism
\[
\selcomof{M_n} \otimes^{\mathbf{L}}_{A_n, \mbf{x}} E' \cong \selcomof{\overline{M}_{\mathbf{x}}}
\]
where $\selcomof{\overline{M}_{\mathbf{x}}}$ is the Selmer complex associated to $\overline{M}_{\mathbf{x}}$ with unramified local conditions away from $p$ and the Panchishkin condition $\invs{C}_{\varphi, \gamma}^{\bullet}(\Ff{+}{o}{\ddrig{\overline{M}_{\mathbf{x}}}})$ at $p$ (compare with the definition in section \ref{ConvolutionOfTwo}). 
\end{enumerate}
\end{theorem}

We will prove this theorem in the next section. Combining this with Theorem \ref{TheoremVanishing}, we obtain the following corollary. 

\begin{corollary}
Let $f$ and $g$ be two modular forms as in section \ref{SummaryOfResults} and let $\invs{F}$ and $\invs{G}$ be Coleman families over $V_1$ and $V_2$ passing through $p$-stabilisations of $f$ and $g$ respectively. Let $\overline{M}$ denote the cyclotomic deformation of $M(\invs{F})^* \hatot M(\invs{G})^*$ as above, and let $\invs{S} = \selmergrp{2}{\overline{M}}$ denote the coherent sheaf obtained as the second cohomology group of the Selmer complex attached to $\overline{M}$.
\begin{enumerate}
    \item Suppose that the (NLZ) hypothesis holds for $f$ and $g$ and that the ``flatness of inertia'' and ``minimally ramified'' hypotheses hold for $\overline{M}$ and $x_0 = (k, k', 0)$. Then, shrinking $V_1$ and $V_2$ if necessary, for all $\mathbf{x} = (k_1, k_2, j) \in X$ with $k_1, k_2, j$ integers satisfying $1 \leq k_2 + 1 \leq j \leq k_1$, the specialisation of $\invs{S}$ satisfies 
    \[
    \invs{S}_{\mathbf{x}} \cong \opn{H}^1_f(\mbb{Q}, [M(\invs{F}_{k_1}) \otimes M(\invs{G}_{k_2})](1+j) )^*
    \]
    where the right-hand side is the (dual of the) Bloch--Kato Selmer group.
    \item (Theorem \ref{SecondMainTheorem}) Suppose that the (NLZ) hypothesis holds for $f$ and $g$, and that the (BI), ``flatness of inertia'' and ``minimally ramified'' hypotheses hold for $\overline{M}$ and $x_0 = (k , k', 0)$. Then, shrinking $V_1$ if necessary, we have the following inclusion 
\[
\opn{supp}\invs{S}_{k'} \subset \{L_p = 0\} 
\]
where ``$\,\opn{supp}$'' denotes the support of a sheaf, $L_p$ is the three-variable $p$-adic $L$-function and $\invs{S}_{k'}$ denotes the specialisation of $\invs{S}$ at $k'$ in the second variable.
\end{enumerate}
\end{corollary}

\begin{remark} \label{NobleRemark}
If we take $V_2$ to be small enough then the conclusion of part (2) of the above Corollary holds for all classical specialisations in $V_2$ (provided that the same hypotheses also hold). More precisely, if $k_2$ is a classical weight in $V_2$ then 
\[
\opn{supp}\invs{S}_{k_2} \subset \{L_p = 0 \},
\]
or to put it another way, the ``slices'' of $\invs{S}$ in the second variable are controlled by the $p$-adic $L$-function $L_p$ provided that the weight in the second variable is classical. 

The reason for this is as follows. By \cite[Lemma 2.7]{bellpadic}, if $V_2$ is small enough then any classical specialisation of $\invs{G}$ is the $p$-stabilisation of a new eigenform $h$ of level $\Gamma_1(N_2)$, and \emph{both} $p$-stabilisations of $h$ are noble. This allows us to prove the analogue of Theorem \ref{TheoremVanishing} for the pair $f, h$ of modular forms (instead of $f$ and $g$). Here we are crucially using the fact that $k \neq k'$.
\end{remark}

\begin{proof}
Let $\mathbf{x} \in X$ be an $E'$-valued point. By Theorem \ref{SheafTheorem}, we have the following isomorphism
\[
\selcomof{\overline{M}} \otimes_{A_{\infty}, x}^{\mathbf{L}} E' \xrightarrow{\sim} \selcomof{\overline{M}_{\mbf{x}}}
\]
which gives the following Tor-spectral sequence 
\[
E^{i, j}_2 \colon \opn{Tor}_{-i}^{A_{\infty}}\left(\widetilde{H}_f^j(\mbb{Q}, \overline{M}), E' \right) \Rightarrow \widetilde{H}_f^{i+j}(\mbb{Q}, \overline{M}_{\mbf{x}}).
\]
If $\mathbf{x}$ does not lie in the support of $\widetilde{H}_f^3(\mbb{Q}, \overline{M})$ then, since $\widetilde{H}_f^i(\mbb{Q}, \overline{M}) = 0$ for $i \geq 4$, we see that 
\[
\widetilde{H}_f^2(\mbb{Q}, \overline{M}) \otimes_{A_{\infty}, x} E' \cong \widetilde{H}_f^2(\mbb{Q}, \overline{M}_{\mbf{x}}).
\]
If $\mathbf{x} = (k_1, k_2, j)$ where $k_1, k_2, j$ are integers satisfying $1 \leq k_2 + 1 \leq j \leq k_1$, then by the discussion in section $\ref{RelationBKSelmer}$ we see that the right hand side of the above isomorphism is isomorphic to the dual of the Bloch--Kato Selmer group for the representation $\overline{M}_{\mathbf{x}}^*(1)$. This proves part $1$ assuming that $\mathbf{x}$ does not lie in the support of $\widetilde{H}_f^3(\mbb{Q}, \overline{M})$. 

Now let $\mbf{x} = (k_1, k', \eta) \in X$ with $k_1$ and $\eta$ not necessarily classical. Assume that $L_p(\mbf{x}) \neq 0$. Then we can apply Theorem \ref{TheoremVanishing}, which says that $\widetilde{H}_f^2(\mbb{Q}, \overline{M}_{\mbf{x}}) = 0$. 

Let $\ide{m}$ denote the kernel of the map $A_{\infty} \to E'$ and let $A_{\infty, \ide{m}}$ denote the localisation of $A_{\infty}$ at $\ide{m}$. Since $\invs{S} = \widetilde{H}_f^2(\mbb{Q}, \overline{M})$ is a coadmissible module
\[
\invs{S} \otimes_{A_{\infty}} A_{\infty, \ide{m}}
\]
is a finitely generated $A_{\infty, \ide{m}}$-module; so by Nakayama's lemma we must have $\invs{S} \otimes_{A_{\infty}} A_{\infty, \ide{m}} = 0$. But this precisely means that $\mbf{x}$ is not in the support of $\invs{S}$ (because the set of points where the stalk of a coherent sheaf is non-zero is automatically closed). This proves part 2 for the points that don't lie in the support of $\selmergrp{3}{\overline{M}}$.

But since $M(\invs{F})$ and $M(\invs{G})$ are irreducible, we have $\widetilde{H}_f^3(\mbb{Q}, \overline{M}) = 0$. Indeed by duality  
\[
\widetilde{H}_f^3(\mbb{Q}, M_n)^* \cong \widetilde{H}_f^0(\mbb{Q}, M_n^*(1)) \subset \opn{H}^0(\mbb{Q}, M_n^*(1)) 
\]
and $\opn{H}^0(\mbb{Q}, M_n^*(1)) = \opn{Hom}_{G_{\mbb{Q}}}(M(\invs{F})^*, M(\invs{G})(\kappa_n + 1))$. If this group is non-zero then because $M(\invs{F})^*$ and $M(\invs{G})(\kappa_n +1)$ are irreducible, they must be isomorphic as representations. But (taking $V_1$ and $V_2$ to be small enough) the generalised Hodge--Tate weights of these representations can never be the same (because $k \neq k'$). 
\end{proof}

\subsection{Proof of Theorem \ref{SheafTheorem}}

We start by proving the following lemma.

\begin{lemma} \label{KeyLemma}
Suppose that $V_1$ and $V_2$ are small enough so that $T_n^{I_v}$ is a flat $A^{\circ}_n$-module, and suppose that the ``minimally ramified'' hypothesis is satisfied. Then for any maximal ideal $\ide{m}$ in $A_n$, the natural map
\[
(M_n^{I_v})/\ide{m} \rightarrow (M_n/\ide{m})^{I_v} 
\] 
is an isomorphism.
\end{lemma}
\begin{proof}
As explained previously (see Lemma \ref{FlatnessLemma} following the ``minimally ramified'' assumption), the result is invariant under twisting in the third variable, so we may assume that we're working over the space
\[
V = V_1 \times V_1 \times \{0\}.
\]
Shrinking $V_1$ and $V_2$ if necessary, we may assume that $V$ is an irreducible affinoid space over $E$. Let $R$ denote its global sections - this is an integral domain. Let $I = I_v$ be an inertia group and set $M = M(\invs{F})^* \hatot M(\invs{G})^*$, thought of as a representation over $R$. It is enough to prove that the natural map 
\[
(M^I)/\ide{m} \rightarrow (M/\ide{m})^I
\]
is an isomorphism, for all maximal ideals $\ide{m}$ of $R$. 

Let $|\cdot|_E$ denote the norm on $E$ (normalised so that $|p|_E = 1/p$) and for any finite field extension $E'$ of $E$, let $|\cdot|_{E'}$ denote the unique norm on $E'$ extending $|\cdot|_E$. Since $R$ is a reduced affinoid algebra it comes with a Banach norm given by
\[
||f|| := \opn{sup}_{\ide{m}}|f \!\!\! \mod{\ide{m}}|_{k(\ide{m})}
\]
where $k(\ide{m})$ denotes the residue field of $\ide{m}$.

Let $R^{\circ}$ denote the unit ball inside $R$ and from now on $E'$ will denote the residue field of a maximal ideal $\ide{m}$ inside $R$. If $R \to E'$ is the continuous surjective homomorphism corresponding to $\ide{m}$ then, by the description of $||\cdot||$ above, we see that $R^{\circ}$ is mapped into the unit ball $\ordd_{E'}$ inside $E'$. Let $\ide{p} = \ide{m} \cap R^{\circ}$. We also let $\ide{m}_0$ denote the maximal ideal corresponding to the point $(k, k', 0)$, and let $\ide{p}_0 = \ide{m}_0 \cap R^{\circ}$. 

Since $\ordd_E$ embeds isometrically into $R^{\circ}$ we have 
\[
\ordd_E \hookrightarrow R^{\circ}/\ide{p} \hookrightarrow \ordd_{E'}.
\]  
Now $R^{\circ}/\ide{p}$ is an integral domain with fraction field $E'$, so $R^{\circ}/\ide{p}$ is an $\ordd_E$-algebra that is finite free of rank $[E':E]$ as an $\ordd_E$-module ($R^{\circ}/\ide{p}$ is torsion-free and $\ordd_E$ is a principal ideal domain). Let $\varpi'$ be a uniformiser of $E'$ and set $J = (\varpi') \cap R^{\circ}/\ide{p}$. This corresponds to a maximal ideal of $R^{\circ}$ which we will denote by $\ide{n}$. 

Let $k$ and $k'$ denote the residue fields of $\ordd_E$ and $\ordd_{E'}$ respectively. Then we have 
\[
k \hookrightarrow R^{\circ}/\ide{n} \hookrightarrow k'
\]
Take $T \subset M$ to be a Galois stable $R^{\circ}$-lattice (which exists by compactness). The representation $T/\ide{p}T \otimes \ordd_{E'}$ is a Galois stable lattice inside $M/\ide{m}M$ and the ``mod $p$ representation'' attached to $M/\ide{m}M$ is 
\[
\left( \left( (T/\ide{p}T \otimes \ordd_{E'}) \otimes \ordd_{E'}/\varpi' \right) \otimes \bar{\mbb{F}}_p \right)^{\mathrm{ss}} = \left( T/\ide{n}T \otimes_{R^{\circ}/\ide{n}} \bar{\mbb{F}}_p \right)^{\mathrm{ss}}
\]
where ss stands for semi-simplification. Since the ``mod $p$ representation'' for a Coleman family is constant, we have
\[
\left( T/\ide{n}T \otimes_{R^{\circ}/\ide{n}} \bar{\mbb{F}}_p \right)^{\operatorname{ss}} \cong \left( T/\ide{n}_0T \otimes_{R^{\circ}/\ide{n}_0} \bar{\mbb{F}}_p \right)^{\operatorname{ss}}.
\]
Here we are using the property that if $\nu$ and $\nu'$ are two representations then $(\nu \otimes \nu')^{\mathrm{ss}} = (\nu^{\mathrm{ss}} \otimes (\nu')^{\mathrm{ss}})^{\operatorname{ss}}$. Now consider the short exact sequence
\begin{equation} 
\exactseq{\ide{n}T}{T}{T/\ide{n}T} \nonumber
\end{equation}
Taking inertia invariants and using the fact that $(\ide{n}T)^I = \ide{n}T^I$, we see that
\begin{equation} \label{firsteq}
T^I/\ide{n}T^I \hookrightarrow \left(T/\ide{n}T\right)^I.
\end{equation}
Similarly we have two more injective maps 
\begin{eqnarray} 
T^I/\ide{p}T^I & \hookrightarrow & \left(T/\ide{p}T\right)^I \label{secondeq} \\
\left(T/\ide{p}T\right)^I/J & \hookrightarrow & \left(T/\ide{n}T\right)^I. \label{thirdeq}
\end{eqnarray}
The map in (\ref{firsteq}) factors as (\ref{secondeq}) modulo $J$ followed by (\ref{thirdeq}), i.e. it factors as
\[
T^I/\ide{n}T^I \rightarrow \left(T/\ide{p}T\right)^I/J \hookrightarrow \left(T/\ide{n}T\right)^I
\]
and the first map is injective. We then have 
\[
T^I/\ide{n}T^I \otimes_{R^{\circ}/\ide{n}} \bar{\mbb{F}}_p \hookrightarrow (T/\ide{n}T)^I \otimes \bar{\mbb{F}}_p \hookrightarrow (T/\ide{n}T \otimes \bar{\mbb{F}}_p)^I
\]
and we obtain the following sequence of inequalities:
\begin{eqnarray}
\opn{dim}_{R^{\circ}/\ide{n}}T^I/\ide{n}T^I & \leq & \opn{dim}_{R^{\circ}/\ide{n}}(T/\ide{n}T)^I \nonumber \\
 & \leq & \opn{dim}_{\bar{\mbb{F}}_p}\left(T/\ide{n}T \otimes \bar{\mbb{F}}_p \right)^I \nonumber \\
 & \leq & \opn{dim}_{\bar{\mbb{F}}_p}\left((T/\ide{n}T \otimes \bar{\mbb{F}}_p)^{\operatorname{ss}} \right)^I \nonumber \\
 & = & \opn{dim}_{\bar{\mbb{F}}_p}\left((T/\ide{n}_0T \otimes \bar{\mbb{F}}_p)^{\operatorname{ss}} \right)^I. \label{inequalities}
\end{eqnarray}
All of these inequalities become equalities when $\ide{n} = \ide{n}_0$ by the ``minimially ramified'' assumption. 

Now we use the fact that $T^I$ is a flat $R^{\circ}$-module. In particular, the localisation $(T^I)_{\ide{n}}$ is free and so
\[
\opn{dim}_{R^{\circ}/\ide{n}}T^I/\ide{n} = \opn{dim}_{R^{\circ}/\ide{n}}(T^I)_{\ide{n}}/\ide{n} = \opn{dim}_{\invs{K}}T^I \otimes_{R^{\circ}} \invs{K}
\]
where $\invs{K}$ denotes the fraction field of $R^{\circ}$ (recall that $R^{\circ}$ is an integral domain). Hence the quantity $\opn{dim}_{R^{\circ}/\ide{n}}T^I/\ide{n}$ is constant and so all inequalities in (\ref{inequalities}) become equalities, for general $\ide{n}$. This implies that the map
\[
T^I/\ide{n}T^I \hookrightarrow \left(T/\ide{p}T\right)^I/J
\]
is an isomorphism.  

Consider the exact sequence
\[
\exactseq{T^I/\ide{p}T^I}{(T/\ide{p}T)^I}{W}
\]
where $W$ is a finitely generated $R^{\circ}/\ide{p}$-module. Then we have $W/JW = 0$, which implies that $W[\frac{1}{\varpi}] = 0$. Indeed, $W$ is a finitely generated $\ordd_E$-module and $\varpi \in J$. Since inverting $\varpi$ commutes with taking inertia invariants, we localise the above sequence and we see that the natural map
\[
M^I/\ide{m}M^I \rightarrow \left(M/\ide{m}M\right)^I
\]
is an isomorphism, as required.
\end{proof}

We are now in a position to prove Theorem \ref{SheafTheorem}.

\begin{proof}[Proof of Theorem \ref{SheafTheorem}]
It is clear from the definition of $A_n$ as the global sections of $Y_n$ that $A_n$ is a flat $A_{\infty}$-module. So for the first part we need to check that taking cohomology and constructing the local conditions both commute with $- \otimes_{A_{\infty}} A_n$. 

By part 2 in Lemma \ref{PullbackLemma} and finiteness, we have 
\[
\rcont{G}{\overline{M}} \otimes_{A_{\infty}} A_n \cong \rcont{G}{M_n}
\]
for $G = G_{\Sigma}$ or $G = G_v$. Similarly $\overline{M}^{I_v}$ is a family of representations over $X$ of the group $G_v/I_v$ and, after shrinking $V_1$ and $V_2$, we can assume that $\overline{M}^{I_v}$ is a flat family, in the sense that 
\[
\Gamma(Y_n, \overline{M}^{I_v}) = M_n^{I_v}
\]
is a flat $A_n$-module, for all $n$. The pair $(G_v/I_v, M_n^{I_v})$ satisfies ``hypothesis A'' in \cite{pottharst-selmer}, so in fact (by the same proof) the conclusion in part 2 of Lemma \ref{PullbackLemma} holds for $G_v/I_v$ and $\overline{M}^{I_v}$ in place of $G_v$ and $\overline{M}$, i.e. 
\[
\mathbf{L}\iota_n^* \rcont{G_v/I_v}{\overline{M}^{I_v}} \cong \rcont{G_v/I_v}{M_n^{I_v}}.
\]
Again by finiteness this implies that 
\[
\rcont{G_v/I_v}{\overline{M}^{I_v}} \otimes_{A_{\infty}} A_n \cong \rcont{G_v/I_v}{M_n^{I_v}}. 
\]

To complete the proof of part 1, we just need to check that the local condition at $p$ commutes with base change to $A_n$. By \cite[Theorem 2.5]{pottharst-selmer}, the Herr complex $\invs{C}^{\bullet}_{\varphi, \gamma}(\Ff{+}{o}\ddrig{\overline{M}})$ is quasi-isomorphic to a complex $C^{\bullet}$ of coadmissible $A_{\infty}$-modules. Since $A_n$ is a flat $A_{\infty}$-module we have
\[
C^{\bullet} \hatot_{A_{\infty}} A_n \cong C^\bullet \hatot^{\mathbf{L}}_{A_{\infty}} A_n \cong \invs{C}^{\bullet}_{\varphi, \gamma}(\Ff{+}{o}\ddrig{M_n}).
\]
By finiteness of cohomology, this implies that the natural map 
\[
\rcont{\mbb{Q}_p}{\Ff{+}{o}{\ddrig{\overline{M}}}} \otimes_{A_{\infty}} A_n \rightarrow \rcont{\mbb{Q}_p}{\Ff{+}{o}{\ddrig{M_n}}}
\]
is an isomorphism. 

Let $\mbf{x}$ be an $E'$-valued point in $Y_n$. We now restrict ourselves to the setting where we have a representation $M_n$ over $A_n$ which comes from a $A_n^{\circ}$-lattice $T_n \subset M_n$. Shrinking $V_1$ and $V_2$ if necessary, we can assume that $T_n^{I_v}$ is a flat $A_n^{\circ}$-module for all $v \in \Sigma$ not equal to $p$. Furthermore, by Lemma \ref{KeyLemma}, specialisation at $\mbf{x}$ commutes with taking inertia invariants. The result then follows from \cite[\S 3.4]{pottharst-selmer} (see in particular equation (3.3)). 
\end{proof}


\appendix

\section{Justification of hypotheses} \label{Appendix}

In this appendix we give justifications for the hypotheses made throughout the paper.

\subsection{The ``Big Image'' hypothesis} \label{BIJustification}

Let $f$ and $g$ be normalised new cuspidal eigenforms of levels $\Gamma_1(N_1)$ and $\Gamma_1(N_2)$, weights $k+2, k'+2$ and characters $\varepsilon_f$ and $\varepsilon_g$ respectively. Let $L_f$ and $L_g$ be the coefficient fields of $f$ and $g$. Assume that $f$ and $g$ are not of CM type and that $f$ is not a Galois twist of $g$, i.e. there doesn't exist an embedding $\gamma: L_f \to \mbb{C}$ and a Dirichlet character $\chi$ such that $f^{\gamma} = g \otimes \chi$.

Let $V$ be a $p$-adic representation of $G_{\mbb{Q}}$ with coefficients in a finite extension $E$ of $\mbb{Q}_p$. Recall the ``Big Image'' hypothesis from section \ref{AVanishingResult}:
\begin{itemize}
\item[(BI)] There exists an element $\sigma \in \Gal(\bar{\mbb{Q}}/\mbb{Q}(\mu_{p^{\infty}}))$ such that $V/(\sigma - 1)V$ is one-dimensional (over $E$).
\end{itemize}

Let $\ide{p}$ be a prime in the compositum $L = L_f L_g$ and consider the representation $V = M_{L_{\ide{p}}}(f)^* \otimes M_{L_{\ide{p}}}(g)^*$. Then it is shown in \cite{LoefflerBI} that for all but finitely many primes $\ide{p}$, the ``Big Image'' hypothesis holds for the representation $V$. Let $\ide{p}$ be such a prime (lying above a prime $p \geq 7$ say) and suppose that we have a Coleman family $\invs{F}$ defined over $V_1$ passing through a $p$-stabilisation of $f$. Then it is not immediately obvious whether we can shrink $V_1$ such that the ``Big Image'' hypothesis holds for the representation
\[
M(\invs{F}_{k_1})^* \otimes M(g)^*
\]
for all specialisations $k_1 \in V_1$, even if we were to restrict $k_1$ to just classical weights. In this section we show that this is indeed possible by using the fact that the mod $p$ reduction of the above representation is constant, for $V_1$ small enough.

\begin{lemma} \label{teichmuller}
Let $G$ be a profinite group, let $\rho \colon G \to \GL_n(\OO_E)$ be a continuous representation, and let
$\bar{\rho} \colon G \to \GL_n(k_E)$ be the corresponding
residual representation.  Suppose there exists $g_0 \in G$
so that $\bar{\rho}(g_0)$ has eigenvalue $1$ with multiplicity
one.  Then there exists $g \in G$ such that $\rho(g)$ has
eigenvalue $1$ with multiplicity one.
\end{lemma}
\begin{proof}
Take $g=\lim_{n \to \infty} (g_0)^{p^{n!}}$. Indeed, to show this sequence converges in $G$ it is enough to show that its image in any finite quotient of $G$ is eventually constant, and this is a routine check. Furthermore, the
eigenvalues of $\rho(g)$ are the Teichm\"{u}ller lifts of
the eigenvalues of $\bar{\rho}(g_0)$, so $1$ is an eigenvalue for $\rho(g)$ with multiplicity one.
\end{proof}

In particular, the above lemma can be applied to the tensor
product of a pair of Galois representations whose
residual representations are ``good'' in the following sense.

\begin{definition}
For $i=1, 2$ let $\sigma_i: G_{\mbb{Q}} \to \opn{GL}_2(\bar{\mbb{F}}_p)$ be (continuous) Galois representations. We say the pair $(\sigma_1, \sigma_2)$ is \emph{good} if
\begin{itemize}
\item $\chi_i \defeq \opn{det} \circ \sigma_i : G_{\mbb{Q}(\mu_{p^{\infty}})} \to \bar{\mbb{F}}_p^{\times}$ is a Dirichlet character of conductor $N_i$, and $p$ doesn't divide the order of the group $(\mbb{Z}/N\mbb{Z})^{\times}$, where $N$ is the lowest common multiple of $N_1$ and $N_2$.
\item There exists an element $u \in (\mbb{Z}/N\mbb{Z})^{\times}$ such that the group
\[
\{(\sigma_1(g), \sigma_2(g)) \in \opn{GL}_2(\bar{\mbb{F}}_p) \times \opn{GL}_2(\bar{\mbb{F}}_p) : g \in G_{\mbb{Q}(\mu_{p^{\infty}})} \}
\]
contains the subgroup generated by $\opn{SL}_2(\mbb{F}_p) \times \opn{SL}_2(\mbb{F}_p)$ and the element 
\[
\left( \tbyt{1}{0}{0}{\chi_1(u)}, \tbyt{1}{0}{0}{\chi_2(u)} \right).
\]
\end{itemize} 
If we want to specify the element $u$ we also call $(\sigma_1, \sigma_2, u)$ good. 
\end{definition}

\begin{lemma} \label{SMainTheorem}
Assume $p \ge 7$ and let $E$ be a finite extension of $\mbb{Q}_p$. Let $\rho_1, \rho_2 \colon G_{\QQ} \to \GL_2(\OO_E)$ be two $p$-adic representations, and let $\bar{\rho}_1$, $\bar{\rho}_2$ denote the corresponding residual representations. If there exists an element $u \in (\ZZ/N\ZZ)^{\times}$ such that $(\bar{\rho}_1, \bar{\rho}_2, u)$ is good and $\chi_1(u) \chi_2(u) \neq 1$, then $\rho_1 \otimes_E \rho_2$ satisfies condition (BI).
\end{lemma}
\begin{proof}
Since $p \ge 7$, there exists
$x \in \mathbb{F}_p^{\times}$ such that
$x^{-2} \chi_1(u)$ and $x^2 \chi_2(u)$
are different from $1$. Since $(\bar{\rho}_1, \bar{\rho}_2, u)$ is a good triple, there exists an element $g_0 \in G_{\QQ(\mu_{p^\infty})}$ such that
\[ \bar{\rho}_1(g_0) = \begin{pmatrix} x & 0 \\ 0 & x^{-1} \chi_1(u) \end{pmatrix} \, \quad \bar{\rho}_2(g_0) = \begin{pmatrix} x^{-1} & 0 \\ 0 & x \chi_2(u) \end{pmatrix} \,. \]
The eigenvalues of $\bar{\rho}_1(g_0) \otimes \bar{\rho}_2(g_0)$ are $\{1,x^{-2} \chi_1(u),x^2 \chi_2(u),\chi_1(u)\chi_2(u)\}$, so the eigenvalue
$1$ has multiplicity one and we may apply Lemma \ref{teichmuller}.
\end{proof}

\subsubsection{Examples of good triples $(\sigma_1, \sigma_2, u)$}

Returning to the situation at the start of section \ref{BIJustification}, let $\ide{p}$ be a prime of the compositum $L = L_fL_g$ lying above a prime $p \geq 7$. We have Galois representations
\begin{eqnarray}
\rho_{f, \ide{p}}^* \colon G_{\mbb{Q}}  \to  \opn{GL}(M_{L_{\ide{p}}}(f)^*) \nonumber \\
\rho_{g, \ide{p}}^* \colon  G_{\mbb{Q}}  \to  \opn{GL}(M_{L_{\ide{p}}}(g)^*) \nonumber 
\end{eqnarray}
which satisfy $\opn{det} \circ \rho_{f, \ide{p}}^* = \varepsilon_f^{-1} \chi_{\mathrm{cycl}}^{1+k}$ and $\opn{det} \circ \rho_{g, \ide{p}}^* = \varepsilon_g^{-1} \chi_{\mathrm{cycl}}^{1+k'}$. Let $\sigma_1$ and $\sigma_2$ denote the reductions modulo $\ide{p}$ of $\rho_{f, \ide{p}}^*$ and $\rho_{g, \ide{p}}^*$ respectively. 

Then it is shown in \cite{LoefflerBI} that, for a very large amount of primes $\ide{p}$, $(\sigma_1, \sigma_2)$ is a good pair.\footnote{The Galois representations considered in \cite{LoefflerBI} are actually $\rho_{f, \ide{p}}$ and $\rho_{g, \ide{p}}$ but the results easily carry over to our situation.} In particular, for all but finitely many $\ide{p}$ which split completely in $L/\mbb{Q}$ the triple $(\sigma_1, \sigma_2, u)$ is a good triple for any $u \in (\mbb{Z}/N\mbb{Z})^{\times}$, where $N = 4 \opn{lcm}(N_1, N_2)$. 

Now suppose that $\invs{F}$ and $\invs{G}$ are two Coleman families over open affinoids $V_1, V_2$ passing through $p$-stabilisations of $f$ and $g$ respectively. Let 
\[
M = M_{V_1}(\invs{F})^* \hatot M_{V_2}(\invs{G})^*
\]
be the tensor product of the Galois representations attached to $\invs{F}$ and $\invs{G}$ and take $V_1$ and $V_2$ to be small enough such that $M$ is constant modulo $p$. Note that the representations 
\begin{eqnarray}
\rho_{\invs{F}}^* \colon G_{\mbb{Q}} \rightarrow \opn{GL}(M_{V_1}(\invs{F})^*) \nonumber \\
\rho_{\invs{G}}^* \colon G_{\mbb{Q}} \rightarrow \opn{GL}(M_{V_2}(\invs{G})^*) \nonumber
\end{eqnarray}
satisfy $\opn{det} \circ \rho_{\invs{F}}^*(g) = \varepsilon_{\invs{F}}^{-1}(g)$ and $\opn{det} \circ \rho_{\invs{G}}^*(g) = \varepsilon_{\invs{G}}^{-1}(g)$, for all $g \in G_{\mbb{Q}(\mu_{p^{\infty}})}$. In particular, we can shrink $V_1$ and $V_2$ so that $\varepsilon_{\invs{F}_{k_1}} = \varepsilon_f$ and $\varepsilon_{\invs{G}_{k_2}} = \varepsilon_g$ for all specialisations $\mathbf{x} = (k_1, k_2) \in V_1 \times V_2$. 

Then, assuming $\varepsilon_f(u)\varepsilon_g(u) \not\equiv 1$ modulo $p$, by Lemma \ref{SMainTheorem} we have that the (BI) condition holds for the representation $M_{\mathbf{x}}$ for any specialisation at $\mathbf{x} \in V_1 \times V_2$.

\subsection{The ``flatness of inertia'' and ``minimally ramified'' hypotheses} \label{JustificationForFI}

An example of a pair of modular forms $f$ and $g$ that satisfy the ``flatness of inertia'' and ``minimally ramified'' hypotheses are as follows. Let $\ell_1$ and $\ell_2$ be two distinct primes $\geq 7$ both different from $p$, and let $f$ and $g$ be two normalised cuspidal new eigenforms of levels $\Gamma_1(\ell_1)$ and $\Gamma_1(\ell_2)$, weights $k+2$ and $k'+2$, and characters $\varepsilon_f = \varepsilon_1$ and $\varepsilon_g = \varepsilon_2$ respectively. Suppose that $\varepsilon_1$ and $\varepsilon_2$ are both non-trivial modulo $p$. Let $E$ be a $p$-adic field containing $a_n(f)$, $a_n(g)$ and the images of $\varepsilon_1$ and $\varepsilon_2$, and suppose that $a_{\ell_1}(f)$ and $a_{\ell_2}(g)$ are both non-zero. 

Let $\rho_1$ and $\rho_2$ denote the restriction of $M_E(f)$ and $M_E(g)$ to the inertia group at $\ell_1$ and $\ell_2$ respectively. By \cite[\S 5]{localcomponent}, the local components at $\ell_i$ of the automorphic representations associated to $f$ and $g$ are prinicipal series representations; therefore by the local Langlands correspondence (and local-global compatibility) we have 
\[
\rho_i \cong \mathbf{1} \oplus \varepsilon_i^{-1}
\]
where $\mathbf{1}$ is the trivial character.  

Let $\invs{F}$ and $\invs{G}$ be Coleman families over $V_1$ and $V_2$ passing through $p$-stabilisations of $f$ and $g$ respectively. Let $M$ denote the representation $M_{V_1}(\invs{F})^* \hatot M_{V_2}(\invs{G})^*$ and let $M_i$ denote the restriction of $M$ to the inertia group $I_{\ell_i}$.

Since inertial types are locally constant, we can shrink $V_1$ and $V_2$ so that for every classical weight $\mathbf{k} = (k_1, k_2) \in V_1 \times V_2$ the specialisation of $M$ satisfies 
\[
M_{i, \mathbf{k}} \cong \mathbf{1} \oplus \mathbf{1} \oplus \varepsilon_i \oplus \varepsilon_i
\]
for $i=1, 2$. It is not hard to see that the action of $I_{\ell_i}$ on $M$ factors through a finite quotient isomorphic to $\left(\mbb{Z}/\ell_i \mbb{Z}\right)^{\times}$ (it is true on a Zariski dense subset) and that $M_i$ must decompose as 
\[
M_i \cong \mathbf{1} \oplus \mathbf{1} \oplus \varepsilon_i \oplus \varepsilon_i.
\]
Indeed the action factors though a finite group and we can define idempotents corresponding to each direct summand. Taking $\Sigma = \{\ell_1, \ell_2, p, \infty \}$ we see that the ``flatness of inertia'' and ``minimally ramified'' hypotheses hold for $f$ and $g$ (provided that $V_1$ and $V_2$ are small enough). 

A similar argument can be applied if either (or both) of $\varepsilon_i$ are trivial, except now the local component can be an unramified twist of the Steinberg representation and the action of inertia factors through a (not necessarily finite) abelian quotient. However we are primarily interested in the case $\varepsilon_1 \cdot \varepsilon_2 \neq 1$ anyway, otherwise the ``Big Image'' hypothesis would not hold for the representation $M_E(f)^* \otimes M_E(g)^*$.


\renewcommand{\MR}[1]{}
\providecommand{\bysame}{\leavevmode\hbox to3em{\hrulefill}\thinspace}
\providecommand{\MR}{\relax\ifhmode\unskip\space\fi MR }
\providecommand{\MRhref}[2]{%
  \href{http://www.ams.org/mathscinet-getitem?mr=#1}{#2}
}
\providecommand{\href}[2]{#2}


\Addresses

\end{document}